\documentclass[a4paper,11pt]{article}
\usepackage[a4paper]{geometry}
\usepackage{amssymb,amsmath,amsfonts,amsthm,textcomp,enumerate,mathrsfs,dsfont,eufrak,version,mathtools}
\usepackage{graphicx}
\usepackage{tabularx}
\usepackage[english,francais]{babel}
\usepackage[T1]{fontenc}
\usepackage{color}
\usepackage[nottoc]{tocbibind}
\usepackage[all,cmtip]{xy}
\usepackage[utf8]{inputenc}
\newtheorem{theo}{Th\'eor\`eme}[section]
\newtheorem{deftn}{Definition}[section]
\newtheorem{prop}{Proposition}[section]
\newtheorem{cor}{Corollaire}[section]
\newtheorem{lem}{Lemme}[section]

\title{Construction d'un complexe diff\'erentiel pour des modules de Speh
$\theta$-invariants}
\author{Nicol\'as Arancibia Robert}
\date{\vspace{-5ex}}
\begin{document}

\maketitle
\begin{abstract}
Soit $\pi$ un module de Speh de $\mathbf{GL}(2n,\mathbb{R})$ basé sur une série discrète de $\textbf{GL}(2,\mathbb{R})$.
Cet article a pour objectif de construire un complexe différential  pour $\pi$ par de sommes de modules standards auto-duals,
\begin{align}
0\rightarrow \pi\rightarrow X_{0}\rightarrow \cdots\rightarrow X_{i}\xrightarrow{\phi_{i}} X_{i+1}\rightarrow\cdots\rightarrow 0.\end{align}
Les modules standard apparaissant dans \textbf{(1)} sont les modules standard auto-duals intervenant dans la résolution de Johnson de $\pi$,
ils sont paramétrés par l'ensemble des involutions $\mathfrak{I}_{n}$ du groupe symétrique $\mathfrak{S}_{n}$. 
Via le paramétrage précédent, il est possible de montrer que l'ordre de Bruhat inversé sur $\mathfrak{I}_{n}$, co\"incide
avec l'ordre de Vogan défini sur l'ensemble des représentations irréductibles de $\mathbf{GL}(\mathbb{R})$. 
De ce résultat nous ramenons la construction de \textbf{(1)} 
à l'étude des propriétés combinatoires de l'ordre de Bruhat sur $\mathfrak{I}_{n}$. 

Dans le dernier chapitre nous montrons que le complexe différentiel \textbf{(1)}, pour $n\leq 4$,
est $\theta$-exact, c'est-à-dire tel que la trace tordue de $\ker\phi_{i+1}/\text{im}\phi_{i}$ est nulle.
Ceci nous permet d'écrire la trace tordue de $\pi$ comme combinaison linéaire de traces tordues 
de modules standard auto-duals, ce qui implique pour $\pi$ un des principaux résultats de l'article, 
\textit{Paquets d'Arthur des Groupes classiques et unitaires} \textbf{\cite{AMR}}.
\end{abstract}
\selectlanguage{english}
\begin{abstract}
Let $\pi$ be a Speh module of $\mathbf{GL}(2n,\mathbb{R})$ based on a discrete series of $\textbf{GL}(2,\mathbb{R})$.
The aim of this paper is to build a chain complex of $\pi$ by direct sum of auto-duals standard modules,
\begin{align}
0\rightarrow \pi\rightarrow X_{0}\rightarrow \cdots\rightarrow X_{i}\xrightarrow{\phi_{i}} X_{i+1}\rightarrow\cdots\rightarrow 0.\end{align}
The standard modules of \textbf{(2)} are the auto-duals standard modules which occurs in the Johnson's resolution of $\pi$,
they are parameterized by the set of involutions $\mathfrak{I}_{n}$ of the symmetric group $\mathfrak{S}_{n}$. 
Under this parametrization one can show that the inversion of the Bruhat order in $\mathfrak{I}_{n}$ coincide
with the Vogan order defined over the set of irreducible representations of $\mathbf{GL}(\mathbb{R})$. 
This allows us to reduce the construction of \textbf{(2)} to the study of combinatorial properties 
of the Bruhat order on $\mathfrak{I}_{n}$. 

In the last chapter we show that the chain complex \textbf{(2)}, for $n\leq 4$,
is $\theta$-exact i.e. the twisted trace of $\ker\phi_{i+1}/\text{im}\phi_{i}$ is trivial.
This allows us to write the twisted trace of $\pi$ as a linear combination of twisted 
traces of standard modules, wich implies for $\pi$ one of the the main results of the paper, 
\textit{Paquets d'Arthur des Groupes classiques et unitaires} \textbf{\cite{AMR}}.
\end{abstract}
\selectlanguage{french}
\tableofcontents
\section{Introduction}

Soit $\mathbf{G}$ un groupe alg\'ebrique r\'eductif d\'efini sur 
$\mathbb{R}$, notons par $\mathfrak{g}$  la complexification de l'alg\`ebre de Lie
de $\mathbf{G}$ et fixons une involution de Cartan $\tau$ de $\mathbf{G}(\mathbb{R})$
de sorte que le sous-groupe $K=\mathbf{G}(\mathbb{R})^{\tau}$ soit un compact maximal de $\mathbf{G}(\mathbb{R})$.

Les repr\'esentation irr\'eductibles et unitaires de $\mathbf{G}(\mathbb{R})$ ayant de la $(\mathfrak{g},K)$-cohomologie non
triviale ont \'et\'e classifi\'e par Vogan et Zuckerman. Elles sont obtenues de la mani\`ere suivante.
Soit $\mathfrak{q}$ une sous-alg\`ebre parabolique $\tau$-stable de $\mathfrak{g}$ de d\'ecomposition 
de Levi $\mathfrak{q}=\mathfrak{l}\oplus\mathfrak{u}$ avec $\mathfrak{l}$ d\'efinie sur $\mathbb{R}$ 
et stable par $\tau$. Posons $\mathbf{L}=N_{\mathbf{G}}(\mathfrak{l})$. Soit $\lambda$ un caract\`ere
unitaire de $\mathbf{L}(\mathbb{R})$ et $H$ un sous-groupe de Cartan de $\mathbf{L}(\mathbb{R})$.
Notons $d\lambda$ la diff\'erentielle de $\lambda|_{H}$ et supposons que la condition de positivit\'e
suivante soit v\'erifi\'ee,
\begin{align}\label{eq:carctcohomo}
\text{Re}(\left<\alpha,d\lambda\right>)\geq 0,\quad \alpha\in\Delta(\mathfrak{h},\mathfrak{u}).
\end{align}
Alors pour tout caract\`ere unitaire $\lambda$ de $\mathbf{L}(\mathbb{R})$, v\'eerifiant \textbf{(\ref{eq:carctcohomo})}, 
posons,
\begin{align}\label{eq:formulerepcoho}
A_{\mathfrak{q}}(\lambda)=\mathscr{R}_{\mathfrak{q}}^{S}(\lambda).
\end{align}
o\`u $\mathscr{R}_{\mathfrak{q}}^{i}$ d\'esigne le $i$-eme foncteur d'induction cohomologique et $S=\dim(\mathfrak{u}\cap\mathfrak{k})$
La repr\'esentation $A_{\mathfrak{q}}(\lambda)$ est unitaire et avec de la $(\mathfrak{g},K)$-cohomologie non
triviale. De plus toute repr\'esentation irr\'eductible et unitaire ayant de la 
$(\mathfrak{g},K)$-cohomologie non triviale en est obtenue par
la formule \textbf{(\ref{eq:formulerepcoho})} apr\`es induction cohomologique d'un caract\`ere 
$\lambda$ de $\mathbf{L}(\mathbb{R})$ v\'eerifiant \textbf{(\ref{eq:carctcohomo})}. 

Fixons une repr\'esentation irr\'eductible et unitaire $\pi$ de $\mathbf{G}(\mathbb{R})$,
ayant de la $(\mathfrak{g},K)$-cohomologie non triviale. Soit $\lambda$ 
un caract\`ere unitaire de $\mathbf{L}(\mathbb{R})$ tel que 
$\pi=A_{\mathfrak{q}}(\lambda)$. L'ensemble $\Pi_{\lambda}$
des repr\'esentation irr\'eductibles de $\mathbf{L}(\mathbb{R})$ 
ayant un caract\`ere infinit\'esimal \'egal \`a celle de $\lambda$ est parametre
par l'ensemble de param\`etres de Beilison Bernsetein $\mathscr{P}_{\lambda}$.
Dans cette parametrization \`a chaque param\`etre $\gamma\in \mathscr{P}_{\lambda}$, Beilinson et Bernstein 
associent un module standard $X^{L}(\gamma)$. Ce module standard poss\`ede la propri\'et\'e de ne contenir
qu'un seul sous-module irr\'eductible que l'on note $\overline{X^{L}(\gamma)}$. La bijection 
\'etablissant la classification est,
\begin{align*}
\gamma\in\mathscr{P}_{\lambda}\mapsto \overline{X^{L}(\gamma)}\in\Pi_{F}(\mathbf{L}(\mathbb{R})). 
\end{align*}
L'ensemble $\mathscr{P}_{\lambda}$ est muni 
d'un ordre partiel $\leq$ appel\'e $\mathcal{G}$-ordre de Bruhat ainsi que d'une fonction longueur, $l$.

Dans \textbf{\cite{Johnson}} J. Johnson donne une r\'esolution de $\lambda$ 
en termes des modules standard. Cette r\'esolution est un complexe exact
de profondeur $q(L)=\frac{1}{2}(\text{dim}(\mathbf{L}(\mathbb{R}))-\text{dim}(\mathbf{L}(\mathbb{R})\cap K))-c_{0}$ ($c_{0}$ la moitié de la dimension de la partie
déployeé d'un sous-groupe Cartan fondamental de $\mathbf{L}(\mathbb{R})$),
\begin{align}\label{eq:suiteJohnsoncaract}
0\rightarrow \lambda\rightarrow \cdots\rightarrow X_{i-1}^{L}\rightarrow X_{i}^{L}\rightarrow\cdots\rightarrow X_{q(L)}^{L}\rightarrow 0,
\end{align}
o\`u $X_{0}^{L}=X^{L}(\gamma_{0}),~\gamma_{0}\in\mathscr{P}_{\lambda}$ est le module standard dont $\lambda$ est l'unique sous module 
irr\'eductible, et pour $i\geq1$, $X_{i}^{L}$ est la somme directe des modules standard $X^{L}(\gamma),~\gamma\in\mathscr{P}_{\lambda}$,
tels que $\gamma\leq\gamma_{0}$ et $l(\gamma)=l(\gamma_{0})-i$. 

En utilisant le foncteur d'induction cohomologique $\mathscr{R}_{\mathfrak{q}}^{S}$ sur \textbf{(\ref{eq:suiteJohnsoncaract})} J. Johnson dans \textbf{\cite{Johnson}} construit pour $\pi=A_{\mathfrak{q}}(\lambda)$ 
une r\'esolution,
\begin{align}\label{eq:resolutionspehintro}
0\rightarrow \pi\rightarrow \cdots\rightarrow X_{i-1}\rightarrow X_{i}\rightarrow\cdots\rightarrow X_{q(L)}\rightarrow 0.
\end{align}
o\`u pour tout $i\in[0,q(L)]$,
$$X_{i}=\oplus_{\gamma\in\mathscr{P}_{\lambda},~l(\gamma)=q(G)-i}\mathscr{R}_{\mathfrak{q}}^{S}(X^{L}(\gamma))$$ 
et, 
\begin{align*}
\phi_{j}&:X_{j-1}\longrightarrow X_{j},\\
\phi_{j}&=\oplus_{\gamma'\in\mathscr{P}_{\lambda},~l(\gamma')=q(G)-j+1}\oplus_{\gamma\in\mathscr{P}_{\lambda},~l(\gamma)=q(G)-j}\phi_{\gamma,\gamma'},\quad
\phi_{\gamma,\gamma'}^{J}:X(\gamma')\longrightarrow X(\gamma)
\end{align*}
Le morphisme $\phi_{\gamma,\gamma'}^{L}$ on l'appele le morphisme de Johnson entre $X(\gamma')$ et $X(\gamma)$.\\
 
Consid\'erons à présent le groupe lin\'eaire  $\mathbf{GL}(N),~N=2n$. Soit $J_{N}\in\mathbf{GL}(N,\mathbb{R})$ la matrice antidiagonale,
\begin{equation}\label{eq:mat}
J_{N}=\left(\begin{array}{cccc}
     ~ &~ &~ &1\\
      ~ &~   &-1  &~\\
  ~&{\mathinner{\mkern2mu\raise1pt\hbox{.}\mkern2mu
\newline \raise4pt\hbox{.}\mkern2mu\raise7pt\hbox{.}\mkern1mu}}&~ &~\\
(-1)^{N+1}&~&~&~
    \end{array}\right).
\end{equation}
Sur $\mathbf{GL}(N)$ d\'efinissons l'automorphisme involutif, $g\mapsto J_{N}^{-1}(^{t}g^{-1})J_{N}$.

Soit $\delta$ une s\'erie discr\`ete de $\mathbf{GL}(2,\mathbb{R})$ de caract\`ere infinit\'esimal
$(p/2,-p/2)$. Consid\'erons le module standard, 
\begin{align}\label{eq:modulestandardSpeh}
\delta\nu^{-n+1}\times\cdots\times\delta\nu^{n-1},\quad \nu=\left|\det\right|.
\end{align}
L'unique sous module irr\'eductible de \textbf{(\ref{eq:modulestandardSpeh})} 
on l'appel le module de Speh de $\mathbf{GL}(2n,\mathbb{R})$ bas\'e sur $\delta$, on le note $\textbf{Speh}(\delta,n)$.

Le module de Speh, $\textbf{Speh}(\delta,n)$, est une repr\'esentation unitaire admettant une $(\mathfrak{g},K)$-cohomologie 
non triviale. Elle dispose en cons\'equence d'une r\'esolution de Johnson.
Dans cette r\'esolution l'ensemble des modules standard intervenant est param\`etre par le groupe symetrique $\mathfrak{S}_{N/2}$. 
Celle-ci est un groupe de Coxeter,  qui est donc muni d'une fonction longueur $l_{\mathfrak{S}}$ 
et d'un ordre $\leq_{B}$. Dans cette param\'etrisation \`a tout $s\in\mathfrak{S}_{N/2}$ on associe,
\begin{align*}
X(s)=\times_{i=1}^{n}\delta\left(\frac{-p-(n-1)}{2}+(i-1),\frac{p-(n-1)}{2}+(s(i)-1)\right), 
\end{align*}
et pour tout $i\in\{1,\cdots,\frac{n(n-1)}{2}\}$ le module $X_{i}$ de \textbf{(\ref{eq:resolutionspehintro})} est d\'efinit comme
la somme directe des modules standard $X(s)$ avec $l_{\mathfrak{S}}(s)=i$.

On dit d'une repr\'esentation $\pi$ de $\mathbf{GL}(N,\mathbb{R})$ qu'elle est $\theta_{N}$-invariante si $\pi$ est \'equivalente \`a 
$\pi\circ\theta_{N}$. La repr\'esentation $\textbf{Speh}(\delta,n)$ est $\theta_{N}$-invariante, de plus,
pour tout $s\in\mathfrak{S}_{N}$ le module standard $X(s)$  
est $\theta_{N}$-invariant si et seulement si $s$ appartient \`a $\mathfrak{I}_{n}$, l'ensemble des involutions de
$\mathfrak{S}_{n}$. Munissons $\mathfrak{I}_{n}$ de l'ordre induit par $\leq_{B}$.
Sur $\mathfrak{I}_{n}$ on peut aussi d\'efinir une fonction longueur, que l'on note, $l_{\mathfrak{I}}$. 
L'objectif du pr\'esent article est de prouver,
\begin{theo}\label{theo:intro}
Soit $s_{\text{max}}$ l'élément maximal de $\mathfrak{S}_{n}$. 
Pour tout $i\in\{1,\cdots,l_{\mathfrak{I}}(s_{\text{max}})\}$ notons,
\begin{align*}
X_{i,\theta}=\oplus_{s\in\mathfrak{I}_{n},l_{\mathfrak{I}}(s)=i}X(s).
\end{align*}
Alors il existe pour tout $i$, une fl\`eche, 
\begin{align*}
X_{i,\theta}\rightarrow X_{i+1,\theta}
\end{align*}
de fa\c con \`a ce que la suite,
\begin{align}\label{eq:complexetheointro}
0\rightarrow \pi\rightarrow \cdots\rightarrow X_{i-1,\theta}\rightarrow X_{i,\theta}\rightarrow\cdots\rightarrow X_{l_{\mathfrak{I}}(s_{\text{max}}),\theta}\rightarrow 0,
\end{align}
d\'efinisse un complexe diff\'erentiel.
\end{theo}
Le complexe \textbf{(\ref{eq:complexetheointro})} n'est certainement pas exact, mais on aimerait bien pouvoir montrer 
qu'il est $\theta_{N}$-exacte, c'est-\`a-dire tel que la trace tordue des noyaux modulo les images 
des fl\`eches soit nulle.
Pour le moment on a reussi a prouver la $\theta_{N}$-exactitude de \textbf{(\ref{eq:complexetheointro})} 
pour $\textbf{Speh}(\delta,n)$ avec $n$ plus petit ou \'egal a 4; $n=4$ n'est pas un blocage, 
les m\'ethodes employ\'es pourraient certainement donner aussi $n=5$ mais il faudrait une id\'ee suplementaire
pour traiter le cas g\'en\'eral.

Sur chaque module standard consid\'er\'e, on a une action de $\theta_{N}$ normalis\'ee 
par le choix d'un mod\`ele de Whittaker comme explique par J. Arthur. Ainsi chaque
$X_{i}$ a une action de $\theta_{N}$ not\'ee $A_{i}(\theta_{N})$.  

Une cons\'equence de la $\theta$-exactitude de \textbf{(\ref{eq:complexetheointro})} 
est la obtention dans le groupe de Grothendieck des repr\'esentation tordues de $\textbf{GL}(N,\mathbb{R})$ 
de l'identit\'e,
\begin{align*}
\textbf{Speh}(\delta,n)=\sum_{i=0}^{l_{\mathfrak{I}}(s_{\text{max}})} (-1)^{i}\text{Tr}(X_{i}A_{i}(\theta_{N})).
\end{align*}
Cette \'egalit\'e entre \'el\'ements du groupe de Grothendieck 
est obtenue diff\'eremment dans \textbf{\cite{AMR}}, voir th\'eor\`emes \textbf{(9.5)} et \textbf{(9.7)} pour une preuve.
La m\'ethode utilis\'ee dans \textbf{\cite{AMR}} se base sur le calcul de l'action de l'automorphisme exterieur $\theta_{N}$ de $\mathbf{GL}(N)$
sur le complexe de Johnson  qui r\'esout $\textbf{Speh}(\delta,n)$.

Parlons un peut sur la preuve du th\'eor\`eme \textbf{(\ref{theo:intro})}. 
Pour toute paire $s,s'\in\mathfrak{I}_{n}$ v\'erifiant $s'\leq_{B}s$ et $l_{\mathfrak{I}}(s')=l_{\mathfrak{I}}(s)-1$,
soit $s'=s_{0}<s_{1}<\cdots<s_{m-1}<s_{m}=s'$, avec $l_{\mathfrak{S}}(s_{i})=l_{\mathfrak{S}}(s_{i+1})-1$,
un cha\^ine d'éléments de $\mathfrak{S}_{n}$ aillent de $s'$ vers $s$. Alors pour $s,s'\in\mathfrak{I}_{n}$ avec 
$s'\leq_{B}s$ et $l_{\mathfrak{I}}(s')=l_{\mathfrak{I}}(s)-1$ nous allons définir,  
$$\phi_{s,s'}^{J}=\phi_{s,s_{m-1}}^{J}\circ\cdots\circ\phi_{s_{1},s'}^{J}.$$
On fait noter que par de propriétés de l'ordre de Bruhat sur $\mathfrak{S}_{n}$, la définition de $\phi_{s,s'}^{J},~s,s'\in\mathfrak{I}_{n}$
ne dépend pas du chemin choisit. 

Si maintenant nous considérons un couple $s,s'\in\mathfrak{I}_{n}$ v\'erifiant $s'\nleq_{B}s$ et $l_{\mathfrak{I}}(s')=l_{\mathfrak{I}}(s)-1$, alors 
nous allons définir $\phi_{s,s'}^{J}$ comme le morphisme nul.
 
Prenons pour tout $s\in\mathfrak{I}_{n}$ une fonctionelle de Whittaker $\Omega_{s}$ 
du module standard $X(s)$.  
Pour toute paire $s,s'\in\mathfrak{I}_{n}$ v\'erifiant $s'\leq_{B}s$ et $l_{\mathfrak{I}}(s')=l_{\mathfrak{I}}(s)-1$ la fonction
 $\phi_{s,s'}^{J}$ est surjective, la composition $\Omega_{s}\circ\phi_{s,s'}^{J}$ est en conséquence une fonctionelle de Whittaker non nulle pour 
$X(s')$, et de l'unicité multiplication par un scalaire près de la fonctionelle de Whittaker, il existe un scalaire $C\in\mathbb{C}^{\ast}$ de manière à ce que, 
$\Omega_{s'}=\Omega_{s}\circ C\phi_{s,s'}$.  Nous allons donc pour tout couple $s,s'\in\mathfrak{I}_{n}$ tel que 
$s'\leq_{B}s$ et $l_{\mathfrak{I}}(s')=l_{\mathfrak{I}}(s)-1$, définir,
\begin{align*}
\phi_{s,s'}^{wh}=C\phi_{s,s'}^{J}.
\end{align*}
Si au contraire $s,s'\in\mathfrak{I}_{N}$ est tel que $s'\nleq_{B}s$ et $l_{\mathfrak{I}}(s')=l_{\mathfrak{I}}(s)-1$, alors nous allons 
définir $\phi_{s,s'}^{J}$ comme le morphisme nul. 

Par de propriétés combinatoires de l'ordre de Bruhat sur $\mathfrak{S}_{n}$ et $\mathfrak{I}_{n}$ on peut montrer: 
\begin{prop}\label{prop:d2ltintroduction}
Soient $s,s'\in\mathfrak{I}_{n}$ v\'erifiant $s\leq_{B}s'$ et $l_{\mathfrak{I}}(s)=l_{\mathfrak{I}}(s')-2$. Alors
ils existent exactement deux \'el\'ements $s_{1}$ et $s_{2}$ dans $\mathfrak{I}_{n}$ de longueur $l_{\mathfrak{I}}(s)+1$
tels que,
\begin{align*}
s<_{B}s_{1},~s_{2}<_{B}s'. 
\end{align*}
De plus,
$$\phi_{s,s_{1}}^{{wh}}\circ \phi_{s_{1},s'}^{{wh}}=\phi_{s,s_{2}}^{{wh}}\circ \phi_{s_{2},s'}^{{wh}}.$$
\end{prop}
Par cons\'equent si pour toute paire 
$l,l'\in \mathfrak{I}_{n}$ v\'erifiant $l'<_{B}l$ et $l_{\mathfrak{I}}(l)-l_{\mathfrak{I}}(l')=1$
nous prennons une constante $\lambda_{l,l'}\in\mathbb{C}$ et notons,
\begin{align*}
\phi_{l,l'}^{\lambda}=\lambda_{l,l'}\phi^{{wh}}_{l,l'}, 
\end{align*}
alors pour tout couple $s,s'\in \mathfrak{I}_{N}$ avec $s'<_{B}s$ 
et $l_{\mathfrak{I}}(s)-l_{\mathfrak{I}}(s')=2$ on aura,
\begin{align}\label{eq:ccdintro}
\phi_{s,s_{1}}^{\lambda}\circ \phi_{s_{1},s'}^{\lambda}+\phi_{s,s_{2}}^{\lambda}\circ \phi_{s_{1},s}^{\lambda}
&=\lambda_{s,s_{1}}\phi_{s,s_{1}}^{{wh}}\circ \lambda_{s_{1},s'}\phi_{s_{1},s'}^{{wh}}+
\lambda_{s,s_{2}}\phi_{s,s_{2}}^{{wh}}\circ \lambda_{s_{2},s'}\phi_{s_{1},s'}^{{wh}}\nonumber\\
&=(\lambda_{s,s_{1}}\cdot\lambda_{s_{1},s'}+\lambda_{s,s_{2}}\cdot\lambda_{s_{2},s'})\left[\phi_{s,s_{1}}^{{wh}}\circ \phi_{s_{1},s'}^{{wh}}=\phi_{s,s_{2}}^{{wh}}\circ \phi_{s_{2},s'}^{{wh}}\right],
\end{align}
avec $s_{1}$ et $s_{2}$ la paire d'\'el\'ements donn\'ee par la proposition \textbf{(\ref{prop:d2ltintroduction})}.

La construction du complexe diff\'erentiel \textbf{(\ref{eq:complexetheointro})} 
se ram\`ene donc \`a trouver pour toute paire  
$l,l'\in \mathfrak{I}_{N}$ avec $l_{\mathfrak{I}}(s)-l_{\mathfrak{I}}(s')=1$, des scalaires $\lambda_{l,l'}\in\mathbb{C}$, avec $\lambda_{l,l'}\neq 0$ si
$l'<_{B}l$, de mani\`ere \`a ce que dans l'equation \textbf{(\ref{eq:ccdintro})}, la somme  
$\lambda_{s,s_{1}}\cdot\lambda_{s_{1},s'}+\lambda_{s,s_{2}}\cdot\lambda_{s_{2},s'}$
soit \'egal \`a z\'ero pour toute couple $s,s'\in \mathfrak{I}_{n},~l_{\mathfrak{I}}(s)-l_{\mathfrak{I}}(s')=2$.

Le point delicat donc de la contruction de \textbf{(\ref{eq:complexetheointro})} est de bien choisir ces scalaires. Pour faire ceci
on procède de la fa\c con suivant;

Soit, dans tout ce qui suit de cette introduction, $\mathbf{G}$ un groupe classique quasi-deploy\'e 
d\'efinit sur les r\'eels, c'est \`a dire,  $\textbf{Sp}(2n)$ ou $\textbf{SO}(p,q)$ avec $p-q=0,1$ ou 2. 
Faisons remarquer que selon la terminologie d'Arthur les groupes classiques que nous consid\'erons
sont ceux qui apparaissent dans les donn\'ees endoscopiques elliptiques simples des
groups tordus $\mathbf{GL}(N)\rtimes\theta_{N}$.
 
Associ\'e au module  $\textbf{Speh}(\delta,n)$ de $\mathbf{GL}(2n)$ on 
a un param\`etre de Arthur, 
$$\psi:W_{\mathbb{R}}\times\textbf{SL}(2,\mathbb{C})\longrightarrow \mathbf{GL}(2n,\mathbb{C}),$$ 
de caract\`ere infinit\'esimal r\'egulier et entier. 
Le module $\textbf{Speh}(\delta,n)$ \'etant $\theta_{2n}$-invariant, il existe un groupe classique $\mathbf{G}$ d\'efini par une des options suivants,
\begin{align*}
\mathbf{G}=\left\{\begin{array}{cl}
                   \textbf{SO}(n+1,n)   &\text{si }\frac{p+n-1}{2}\in2\mathbb{Z}+1,\\
                   \textbf{SO}(n,n)     &\text{si }n\in2\mathbb{N}\text{ et }\frac{p+n-1}{2}\in2\mathbb{Z},\\
                   \textbf{SO}(n+1,n-1) &\text{si }n\in2\mathbb{N}+1\text{ et }\frac{p+n-1}{2}\in2\mathbb{Z},
                  \end{array}\right.
\end{align*}
et un param\`etre d'Arthur $\psi_{\mathbf{G}}$ de $\mathbf{G}$ de fa\c con \`a ce que,
\begin{align*}
\psi:W_{\mathbb{R}}\xrightarrow{\psi_{\mathbf{G}}}~^{L}\mathbf{G}\xrightarrow{\iota}\mathbf{GL}(2n,\mathbb{C}), 
\end{align*}
o\`u $\iota$ d\'enote l'inclusion standard. 

Adams et Johnson \textbf{\cite{Adams-Johnson}} on attach\'e \`a $\psi_{\mathbf{G}}$ un paquet 
de repr\'esentations irr\'eductibles de $\mathbf{G}(\mathbb{R})$, not\'e $\Pi_{\psi_{\mathbf{G}}}^{\text{AJ}}$. 
Toute repr\'esentation dans $\Pi_{\psi_{\mathbf{G}}}^{\text{AJ}}$ est unitaire et avec de la 
$(\mathfrak{g},K)$-cohomologie non triviale. 

D'apr\`es la classification fait par Vogan-Zuckerman pour toute repr\'esentation 
$\pi\in\Pi_{\psi_{\mathbf{G}}}^{\text{AJ}}$ il existe un caract\`ere $\lambda_{\pi}$ d'un groupe
de Levi $\textbf{L}(\mathbb{R})$ de $\textbf{G}(\mathbb{R})$ tel que $\pi=\mathscr{R}_{\mathfrak{q}}^{S}$. 
Les groupe de Levi qui vont intervenir dans notre context sont  les groupes 
unitaires $\textbf{U}(p,q),~p+q=n$.  
 
Fixons $\pi\in\Pi_{\psi_{\mathbf{G}}}^{\text{AJ}}$ de fa\c con \`a ce que le groupe de Levi 
associ\'e, $\mathbf{L}=\textbf{U}(p,q)$, soit quasi-deployé. Pour $\mathbf{L}$
l'ensemble $\Pi_{\lambda_{\pi}}(\textbf{L}(\mathbb{R}))$ des repr\'esentations irr\'eductibles
dont le caract\`ere infinit\'esimal est \'egal \`a celui de $\lambda_{\pi}$ est param\`etre par l'ensemble
$\mathfrak{I}_{n}^{p,q,\pm}$ dont les \'el\'ements sont des couples, $(\eta,f_{\eta})$
o\`u $\eta\in\mathfrak{I}_{n}$ et $f_{\eta}$ est une application de l'ensemble
des points fixes de $\eta$ dans $\{\pm 1\}$ v\'eerifiant en plus 
la condition suplem\'entaire etablie par l'equation \textbf{(\ref{eq:fsupplementaire})}.
Dans cette description la partition de $\Pi_{\lambda}(\mathbf{L})$
en $L$-paquets est particulièrement simple, deux param\`etres $\eta$ et $\eta'\in\mathfrak{I}_{n}^{p,q,\pm}$ 
correspondent \`a des repr\'esentations dans le m\^eme $L$-paquet si et seulement si les involutions sous-jacentes $\eta$ et $\eta'$ dans $\mathfrak{I}_{n}$ sont \'egales.
 Ce param\'etrage de l'ensemble $\Pi_{\lambda_{\pi}}(\textbf{L}(\mathbb{R}))$ a par cons\'equence 
que pour $\lambda_{\pi}$ la r\'esolution de Johnson s'\'ecrit,
\begin{align}\label{eq:suiteJohnsonunitaire}
0\rightarrow\lambda_{\pi}\rightarrow\cdots\rightarrow X_{i}^{L}\rightarrow X_{i+1}^{L}
\cdots\rightarrow X_{l_{\mathfrak{I}}(s_{\text{max}})}^{L}\rightarrow 0. 
\end{align}
o\`u pour tout $i\in[1,l_{\mathfrak{I}}(s_{\text{max}})]$,
\begin{align*}
X_{i}^{L}&:=\oplus_{\eta\in\mathfrak{I}_{n},l_{\mathfrak{I}}(\eta)=i}\Theta_{\eta}\\ 
\Theta_{\eta}&:=\oplus_{\{\overline{\eta}=(\zeta,f_{\zeta})\in\mathfrak{I}_{n}^{p,q,\pm},\zeta=\eta\}}X(\overline{\zeta}).
\end{align*}
Soit $\varphi_{\eta,\mathbf{L}}$ le $L$-param\`etre de $\mathbf{L}$ 
associ\'e \`a l'\'el\'ement $\eta\in\mathfrak{I}_{n}$, c'est le $L$-param\`etre attach\'e au
pseudo-paquet des repr\'esentations intervenant dans la somme, 
$\Theta_{\eta}=\oplus_{\{\overline{\eta}=(\zeta,f_{\zeta})\in\mathfrak{I}_{n}^{p,q,\pm},\zeta=\eta\}}X(\overline{\zeta})$.
Notons par $\iota_{\mathbf{G},\mathbf{L}}$ l'inclusion$~^{L}\mathbf{L}\hookrightarrow~^{G}\mathbf{G}$
et consid\'erons le $L$-param\`etre de $\mathbf{GL}(N)$ donn\'e par la composition,
\begin{align*}
\varphi_{\eta}:W_{\mathbb{R}}\xrightarrow{\varphi_{\eta,\mathbf{L}}}~^{L}
\mathbf{L}\xrightarrow{\iota_{\mathbf{G},\mathbf{L}}}^{L}\mathbf{G}\xrightarrow{\iota}\mathbf{GL}(N).
\end{align*}
Soit $X(\eta)$ le module standard de $\mathbf{GL}(N,\mathbb{R})$ attach\'e \`a $\varphi_{\eta}$,
alors l'application,
\begin{align*}
\Theta_{\eta}\mapsto X(\eta) 
\end{align*}
d\'efinit une bijection entre l'ensemble des modules standard stables, $\Theta_{\eta}$, 
intervenant dans la r\'esolution de Johnson de $\lambda_{\pi}$ est l'ensemble
des modules standard $\theta_{N}$-invariants de $\mathbf{GL}(N,\mathbf{R})$ 
intervenant dans la suite \textbf{(\ref{eq:suite8})}. 

En cons\'equence, pour trouver les constants n\'ecessaires \`a la contruction du complexe  \textbf{(\ref{eq:complexetheointro})} on s'appuie sur la suite \textbf{(\ref{eq:suiteJohnsonunitaire})} et plus particulièrement sur de propriétés du morphisme de Johnson $\varphi_{\overline{\eta},\overline{\eta}'}^{J}$ 
entre $X(\overline{\eta})$ et $X(\overline{\eta}'),~l_{\mathfrak{I}}(\eta')=l_{\mathfrak{I}}(\eta)-1$ . 
On dispose des deux r\'esultats suivants:
\begin{lem}\label{lem:introduc}
Soient $\overline{\eta}=(\eta,f_{\eta}),\overline{\eta}'=(\eta',f_{\eta'})\in\mathfrak{I}_{n}^{p,q,\pm}$ 
tels que $\eta'<_{B}\eta$, $l_{\mathfrak{I}}(\eta')=l_{\mathfrak{I}}(\eta)-1$ 
et tels que les modules standards $X(\overline{\eta})$ et $X(\overline{\eta}')$ 
admettent de fonctionnelles de Whittaker non nulles. 
Alors $\varphi_{\overline{\eta},\overline{\eta}'}^{J}\neq 0$ et $\text{Im}(\varphi_{\overline{\eta},\overline{\eta}'})$ possède un quotient générique.
\end{lem}
\begin{prop}\label{prop:introduct}
Soit $\mathbf{G}$ un groupe classique quasi-deployé; linéaire, unitaire orthogonal ou symplectique, et
$X$ un module standard générique de $\mathbf{G}(\mathbb{R})$ de caractère infinitesimal entier et régulier.
 Alors $X$ admet un unique sous-quotient irréductible générique, lequel appara\^it avec multiplicité un.
\end{prop}
Considérons une famille $\{\overline{\eta},\overline{\eta}',\overline{\eta}_{1},\overline{\eta}_{2}\}$, vérifiant
$\eta'<\eta_{i}<\eta$, $l_{\mathfrak{I}}(\eta')=l_{\mathfrak{I}}(\eta_{i})-1=l_{\mathfrak{I}}(\eta)-2,~i=1,2$ et tel que pour tout $\overline{\mu}$ dans la famille
$X(\overline{\mu})$ admet une fonctionelle de Whittaker non nulle. 
Notons par $\Omega_{\overline{\eta}},\Omega_{\overline{\eta}'},\Omega_{\overline{\eta_{1}}}$ et $\Omega_{\overline{\eta}_{2}}$
les fonctionnelles de Whittaker respectifs.

D'après le lemme \textbf{(\ref{lem:introduc})}, pour toute paire $\zeta,\zeta'\in\mathfrak{I}_{N}$ v\'erifiant $\zeta'<_{B}\zeta$ et 
$l_{\mathfrak{I}}(\zeta')=l_{\mathfrak{I}}(\zeta)-1$, l'image $\text{Im}(\varphi_{\overline{\zeta},\overline{\zeta}'}^{J})$ contient un sous-quotient générique
si on peut dire de m\^eme de $X(\overline{\zeta})$ et $X(\overline{\zeta}')$. 
En conséquence, de la proposition \textbf{(\ref{prop:introduct})}, 
\begin{align*}
\Omega(\overline{\eta}_{i})\circ\varphi_{\overline{\eta}_{i},\overline{\eta}'}^{J}\quad\text{ et }\quad \Omega(\overline{\eta})
\circ\varphi_{\overline{\eta},\overline{\eta}_{i}}^{J},
\end{align*}
d\'efinisent des fonctionelles de Whittaker non nulles 
pour $X(\overline{\eta}_{i})~i=1,2,$ et $X(\overline{\eta})$, respectivement. Par l'unicit\'e de la fonctionelle de Whittaker 
il existent de constants $\beta_{\eta_{i},\eta'}$ et $\beta_{\eta,\eta_{i}},~i=1,2,$ 
de mani\`ere \`a ce que,
\begin{align*}
\Omega(\overline{\eta}')=\Omega(\overline{\eta}_{i})\circ\beta_{\eta_{i},\eta'}\varphi_{\overline{\eta}_{i},\overline{\eta}'}^{J}
\quad\text{ et }\quad\Omega(\overline{\eta}_{i})=\Omega(\overline{\eta})\circ\beta_{\eta,\eta_{i}}\varphi_{\overline{\eta},\overline{\eta}_{i}}^{J},
\end{align*}
 Notons,
\begin{align*}
\varphi_{\overline{\eta}_{i},\overline{\eta}'}^{{wh}}
=\beta_{\eta_{i},\eta'}\phi_{\overline{\eta}_{i},\overline{\eta}'}^{J}\quad\text{ et }
\quad\phi_{\overline{\eta},\overline{\eta}_{i}}^{{wh}}
=\beta_{\eta,\eta_{i}}\varphi_{\overline{\eta},\overline{\eta}_{i}}^{J}. 
\end{align*}
Comme,
\begin{align*}
\varphi_{\overline{\eta},\overline{\eta}_{1}}^{J}\circ \varphi_{\overline{\eta}_{1},\overline{\eta}'}^{J}+
\varphi_{\overline{\eta},\overline{\eta}_{2}}^{J}\circ \varphi_{\overline{\eta}_{2},\overline{\eta}'}^{J}=0, 
\end{align*}
il nous est possible de montrer que,
\begin{align*}
\varphi_{\overline{\eta},\overline{\eta}_{1}}^{{wh}}\circ \varphi_{\overline{\eta}_{1},\overline{\eta}'}^{{wh}}=
\varphi_{\overline{\eta},\overline{\eta}_{2}}^{{wh}}\circ \varphi_{\overline{\eta}_{2},\overline{\eta}'}^{{wh}}. 
\end{align*}
D'où on obtient, 
\begin{align*}
0&=\varphi_{\overline{\eta},\overline{\eta}_{1}}^{J}\circ \varphi_{\overline{\eta}_{1},\overline{\eta}'}^{J}+
\varphi_{\overline{\eta},\overline{\eta}_{2}}^{J}\circ \varphi_{\overline{\eta}_{2},\overline{\eta}'}^{J}\\
 &=\beta_{\eta,\eta_{1}}^{-1}\varphi_{\overline{\eta},\overline{\eta}_{1}}^{{wh}}\circ\beta_{\eta_{1},\eta'}^{-1}\varphi_{\overline{\eta}_{1},\overline{\eta}'}^{{wh}}
 +\beta_{\eta,\eta_{2}}^{-1}\varphi_{\overline{\eta},\overline{\eta}_{2}}^{{wh}}\circ\beta_{\eta_{2},\eta'}^{-1}\varphi_{\overline{\eta}_{2},\overline{\eta}'}^{{wh}}\\
 &=(\beta_{\eta,\eta_{1}}^{-1}\cdot\beta_{\eta_{1},\eta'}^{-1}+\beta_{\eta,\eta_{2}}^{-1}\cdot\beta_{\eta_{2},\eta'}^{-1})
 (\varphi_{\overline{\eta},\overline{\eta}_{1}}^{{wh}}\circ \varphi_{\overline{\eta}_{1},\overline{\eta}'}^{{wh}}+
\varphi_{\overline{\eta},\overline{\eta}_{2}}^{{wh}}\circ \varphi_{\overline{\eta}_{2},\overline{\eta}'}^{wh}),
\end{align*}
ce qui nous permet de conclure l'\'egalit\'e,
\begin{align}
\beta_{\eta,\eta_{1}}^{-1}\cdot\beta_{\eta_{1},\eta'}^{-1}+\beta_{\eta,\eta_{2}}^{-1}\cdot\beta_{\eta_{2},\eta'}^{-1}=0. 
\end{align}
Par cons\'equent si pour toute paire $s,s'\in\mathfrak{I}_{n}$, v\'erifiant $s'<_{B}s$ 
et $l_{\mathfrak{I}}(s')=l_{\mathfrak{I}}(s)-1$, nous définissons,
\begin{align*}
\phi_{s,s'}=\lambda_{s,s'}\phi_{s,s'}^{\text{wh}},\quad\text{où }\lambda_{s,s'}=\beta_{s,s'}^{-1}.
\end{align*} 
Alors la suite \textbf{(\ref{eq:complexetheointro})} construit à partir des morphismes 
\begin{align*}
\phi_{i}&:X_{i-1,\theta}\rightarrow X_{i,\theta},\\
\phi_{i}&=\oplus_{s\in\mathfrak{I}_{n},l_{\mathfrak{I}}(s)=i}\oplus_{s'\in\mathfrak{I}_{n},l_{\mathfrak{I}}(s')=i-1}\phi_{s,s'},
\end{align*}
est un complexe différentiel.

\subsection{Plan du papier}
Le plan du papier est le suivant. Dans la section trois nous introduisons l'ordre de Bruhat $\leq_{B}$
sur le groupe symetrique $\mathfrak{S}_{N}$.
Nous donnons une définition de longueur sur $\mathfrak{S}_{N}$ et sur le sous ensemble $\mathfrak{I}_{N}$ 
des involutions de celui-ci, en plus d'\'enoncer les principales propriétés satisfait par
cette fonction longueur. La section quatre concerne, la théorie de représentation 
de $\mathbf{GL}(N,\mathbb{R})$, elle est divisée en trois parties.
Dans un premier temp nous rappelons une paramétrisation de l'ensemble de représentations irréductibles dont 
le caractère infinitésimal est celui d'une représentation de dimension
fini de $\mathbf{GL}(N,\mathbb{R})$, par l'ensemble $\mathfrak{I}_{N}^{\bullet,\pm}$,  des involutions avec points fixes 
En suite nous introduisons l'ensemble des représentations $\theta_{N}$ invariantes. La  
section quatre s'achève
avec la définition du module de Speh basé sur une série discrète. 
La section cinq s'intéresse à la théorie
de represéntation du groupe unitaire $\mathbf{U}(p,q,\mathbb{R}),~p+q=N$, comme pour 
$\mathbf{GL}(N,\mathbb{R})$ on rappelle un paramétrage combinatoire 
des représentations irréductibles dont le caractère infinitésimal est celui d'une représentation de dimension
fini de $\mathbf{GL}(N,\mathbb{R})$. 
La preuve commence dans la section six. Celle-ci s'ouvre 
avec un rappel des résultats de la thèse de Johnson qui donnent
des résolutions pour les représentations qui nous intéressent. 
En suite nous contruissons \textbf{(\ref{eq:complexetheointro})}, ramenons la preuve de \textbf{(\ref{theo:intro})} à la solution d'un système d'équations et 
montrons le lemme \textbf{(\ref{lem:introduc})} et la proposition \textbf{(\ref{prop:introduct})}, ce qui a comme conséquence la veracité du théorème \textbf{(\ref{theo:intro})}.
Nous finissons avec la preuve de la $\theta$-exactitude de \textbf{(\ref{eq:complexetheointro})}
pour $\textbf{Speh}(\delta,n)$ avec $n$ plus petit ou égal à 4. 
La section \textbf{(5.3)} traite le cas de $\textbf{Speh}(\delta,6)$  
et la section \textbf{(5.4)} celui de $\textbf{Speh}(\delta,8)$ .

\section{Notations}
Dans cet article, nous nous int\'eressons \`a la th\'eorie 
de repr\'esentation des groupes alg\'ebriques suivants:
\begin{enumerate}[i.]
 \item Le groupe lin\'enaire $\mathbf{GL}(N),~N\in\mathbb{N}$.
 \item Le groupe unitaire $\mathbf{U}(p,q),~p,q\in\mathbb{N}$.
 \item Les groupes classiques $\textbf{Sp}(2n),~n\in\mathbb{N}$ 
 et $\mathbf{SO}(p,q),p,q\in\mathbb{N}$ avec $p-q=0,1$ ou 2. 
\end{enumerate}

Soit $\mathbf{G}$ un groupe alg\'ebrique d\'efini par une des trois possibilit\'es pr\'ec\'edents.
Dans tout ce qui suit, nous allons supposer fix\'e une involution de Cartan $\tau$ de $\mathbf{G}$, de sorte que le sous-groupe
$K=\mathbf{G}(\mathbb{R})^{\tau}$ des points fixes de $\tau$ dans $\mathbf{G}(\mathbb{R})$ soit un sous-groupe compact maximal de $\mathbf{G}(\mathbb{R})$.
Par exemple, si $\mathbf{G}=\mathbf{GL}(N)$ alors,
\begin{align*}
\tau:g\mapsto~^{t}g^{-1} 
\end{align*}
et $K=\mathbf{O}(N,\mathbb{R})$. Maintenant, si $\mathbf{G}=\mathbf{U}(p,q)$ alors
\begin{align*}
\tau:g\mapsto~^{t}\bar{g}^{-1} 
\end{align*}
et $K=\mathbf{U}(p,\mathbb{R})\times \mathbf{U}(q,\mathbb{R})$.

Soit $\mathcal{M}=\mathcal{M}(\mathbf{G}(\mathbb{R}),\tau)$ la categorie des modules de Harish-Chandra pour la paire
$(\mathfrak{g},K)$, $\mathfrak{g}$ la complexification de l'alg\`ebre de Lie du groupe de 
Lie $\mathbf{G}(\mathbb{R})$. Notons $\Pi(\mathbf{G}(\mathbb{R}))$ l'ensemble des classes d'\'equivalence de modules 
irr\'eductibles dans $\mathcal{M}$. 
Si $F$ est une repr\'esentation irr\'eductible de dimension
finie de $\mathbf{G}(\mathbb{R})$, notons $\Pi_{F}(\mathbf{G}(\mathbb{R}))$ l'ensemble des \'el\'ements de $\Pi(\mathbf{G}(\mathbb{R}))$ 
dont le caract\`ere infinit\'esimal est \'egal \`a celui de $F$.

Dans \textbf{\cite{BB}} Beilinson et Bernstein 
proposent une param\'etrization de $\Pi_{F}(\mathbf{G}(\mathbb{R}))$ par un ensemble de param\`etres, not\'e 
$\mathscr{P}_{F}$. Dans cette parametrization \`a chaque param\`etre $\gamma\in \mathscr{P}_{F}$, Beilinson et Bernstein 
associent un module standard $X(\gamma)$. Ce module standard poss\`ede la propri\'et\'e de ne contenir
qu'un seul sous-module irr\'eductible que l'on note $\overline{X(\gamma)}$. La bijection 
\'etablissant la classification est,
\begin{align*}
\gamma\in\mathscr{P}_{F}\mapsto \overline{X(\gamma)}\in\Pi_{F}(\mathbf{G}(\mathbb{R})). 
\end{align*}
Dans \textbf{\cite{VoganIII}}, David Vogan munit l'ensemble $\Pi_{F}(\mathbf{G}(\mathbb{R}))$
d'un ordre partiel $\leq$ appel\'e $\mathcal{G}$-ordre de Bruhat ainsi que d'une fonction longueur,
\begin{align*}
l:\mathscr{P}_{\chi}\rightarrow \mathbb{Z}. 
\end{align*}
L'ordre de Bruhat et la longueur \'etant d\'efinis sur l'ensembles
des param\`etres $\mathscr{P}_{F}$ on peut par transport de structure,
les d\'efinir sur les $\Pi_{F}(\mathbf{G}(\mathbb{R}))$ ou bien sur tout
autre ensemble de param\`etres en bijection avec $\Pi_{\mathbf{G}(\mathbb{R})}(F)$.

L'ordre de Bruhat et la longueur poss\`edent les propri\'et\'es fondamentales
suivantes,
\begin{prop}
\textbf{i.} Soient $\gamma,\gamma'$ dans $\mathscr{P}_{F}$ avec $\gamma\leq\gamma'$.
Alors $l(\gamma)\leq l(\gamma')$. Si $\gamma\leq\gamma'$ et $l(\gamma)\leq l(\gamma')$, 
alors $\gamma=\gamma'$. D'autre part, si $\gamma\leq\gamma'$ et il n'existe pas
d'\'el\'ement strictement compris entre $\gamma$ et $\gamma'$ dans le $\mathcal{G}$-ordre
de Bruhat, alors $\gamma=\gamma'$ ou bien $l(\gamma')=l(\gamma)+1$.\\
\textbf{ii.} Soient $\gamma,\gamma'\in\mathscr{P}_{F}$ tels que 
$\overline{X(\gamma)}\in\mathrm{J.H}(X(\gamma))$. Alors $\gamma\leq\gamma'$. Si 
de plus $l(\gamma')=l(\gamma)$ ou $l(\gamma')=l(\gamma)+1$ alors $\overline{X(\gamma)}$ appara\^it dans $\mathrm{J.H}(X(\gamma))$ 
avec multiplicit\'e 1.
\end{prop}
Parlons un peut sur l'ensemble $\mathscr{P}_{F}$. Notons 
pour cela par $\mathscr{B}$ la vari\'et\'e des drapeaux de $\mathfrak{g}$ et 
$K_{\mathbb{C}}$ la complexification du sous-groupe compact maximal $K$.
Alors $\mathscr{P}_{F}$ est constitu\'e par des couples $\gamma=(Q,\chi)$
o\`u $Q$ est une $K_{\mathbb{C}}$-orbite dans $\mathscr{B}$ et $\chi$ est un fibr\'e
en droite holomorphe plat $K_{\mathbb{C}}$-homog\`ene sur $Q$. Les
modules ${X(\gamma)}$ et $\overline{X(\gamma)}$ sont obtenus \`a partir
de $\gamma$ par la th\'eorie de la localisation de Beilinson-Bernstein, \textbf{\cite{BB}}.\\

Si maintenant nous consid\'erons le cas du groupe lin\'eaire $\mathbf{G}=\mathbf{GL}(N)$, alors $\mathbf{K}=\mathbf{O}(N)$
et les $\mathbf{O}(N,\mathbb{C})$-orbites sur la vari\'et\'e des drapeaux $\mathscr{B}$
de $\mathbf{GL}(N,\mathbb{C})$ sont param\`etr\'ees par les involutions dans le groupe
sym\'etrique $\mathfrak{S}_{N}$. Notons $\mathfrak{I}_{N}$ l'ensemble des involutions
dans $\mathfrak{S}_{N}$ on a donc une bijection entre $\mathfrak{I}_{N}$ 
et $\mathbf{O}(N,\mathbb{C})/\mathscr{B}$. Danc la section \textbf{(4.1)}
nous introduisons l'ensemble $\mathfrak{I}_{N}^{\bullet,\pm}$
des involutions avec points fixes sign\'es form\'es des couples, $(\eta,f_{\eta})$
o\`u $\eta\in\mathfrak{I}_{N}$ et $f_{\eta}$ est une application de l'ensemble
des points fixes de $\eta$ dans $\{\pm 1\}$. Ce donn\'ee suppl\'ementaire
permet de construir un \'el\'ement de $\mathscr{P}_{\chi}$ et donc d'obtenir 
une bijection entre $\mathfrak{I}_{N}^{\bullet,\pm}$ et $\Pi_{F}(\mathbf{G}(\mathbb{R}))$.\\

Passons au cas du groupe unitaire $\mathbf{G}=\mathbf{U}(p,q)$. Posons $p+q=N$. On a 
$K\cong \mathbf{U}(p)\times \mathbf{U}(q)$ et $K_{\mathbb{C}}=\mathbf{GL}(p,\mathbb{C})\times \mathbf{GL}(p,\mathbb{C})$. 
 Dans la section \textbf{(5.1)} nous donnons un param\'etrage 
combinatoire de $K_{\mathbb{C}}/\mathscr{B}$ et donc de $\Pi_{F}(\mathbf{G}(\mathbb{R}))$,  car pour
ce groupe $K_{\mathbb{C}}/\mathscr{B}\cong\Pi_{F}(\mathbf{G}(\mathbb{R}))$. Cet param\'etrage est fait \`a
partir d'un ensemble not\'e $\mathfrak{I}_{N}^{p,q,\pm}$ qui, tel 
tel comme pour $\mathbf{GL}(N)$, est constitu\'e par des couples, $(\eta,f_{\eta})$
o\`u $\eta\in\mathfrak{I}_{N}$ et $f_{\eta}$ est une une application de l'ensemble
des points fixes de $\eta$ dans $\{\pm 1\}$, v\'eerifiant en plus 
la condition suplem\'entaire etablie par l'equation \textbf{(\ref{eq:fsupplementaire})}.\\
~\\
Maintenant qu'on a pu constater l'importance de l'ensemble des involutions $\mathfrak{I}_{N}$ 
dans la param\'etrisation de $\Pi_{F}(\mathbf{G}(\mathbb{R}))$ passons \`a donner une description plus
en d\'etail de cet ensemble. 

\section{L'ensemble des involutions du groupe sym\'etrique}
\subsection{Ensemble partiellement ordonné}
Soit $P$ un ensemble fini partiellement ordonné. Nous allons dire que 
$P$ est borné s'il existe un élément $\hat{1}\in P$ et un élémént $\hat{0}\in P$
tel que pour tout $x\in P$ on a $\hat{0}\leq x\leq \hat{1}$. Pour tout $x, y\in P$
avec $x\leq y$ notons,
\begin{align*}
[x,y]=\{z\in P:x\leq z\leq y\}.
\end{align*}  
L'ensemble $[x,y]$ on l'appelle un intervalle de $P$. Si $x, y\in P$
avec $x<y$, une chaîne de $x$ à $y$ de longueur $k$
est un $k$-uplet $(x_{0},x_{1},\cdots,x_{k})$ tel que $x=x_{0}<x_{1}<\cdots<x_{k}=y$. 
Pour tout couple $x,y\in P$ on dit que $y$ couvre $x$ si $x<y$ et il n'existe pas $z\in P$ 
tel que $x<z<y$. Une chaîne $x_{0}<x_{1}<\cdots<x_{k}$ de $P$ se dit alors saturée si 
pour tout  $0\leq i\leq k-1$, $x_{i+1}$ couvre $x_{i}$.
Finalement on dit que $P$ est gradué de rang $n$ si $P$ est borné avec 
tous les chaînes maximales de même longueur $n$.
\begin{deftn}
Soit $P$ un ensemble fini partiellement ordonné, pure et gradué. Alors on dit
que $P$ est shellable s'il existe une relation d'ordre totale sur l'ensemble 
des chaînes maximales, $C_{1},C_{2},\cdots,C_{k}$, de $P$ tel que, pour tout $i<j$
il existe $t<j$ tel que,
$$C_{i}\cap C_{j}\subset C_{t}\cap C_{j},\text{ et } |C_{t}\cap C_{j}|=|C_{j}|-1.$$ 
\end{deftn}
\begin{deftn}
Pour $P$ un ensemble fini partiellement ordonné, notons, 
$$C(P)=\{(x,y)\in P^{2},~x\text{ est couvert par y}\}.$$    
Soit $Q$ un ensemble fini totalement ordonné. Un edge labelling de $P$
avec valeurs dans $Q$ est une application $\lambda:C(P)\rightarrow Q$. 
Si lambda est un edge labelling de $P$, pour toute chaîne saturée $x_{0}<x_{1}<\cdots<x_{k}$ notons,
$$\lambda(x_{0},x_{1},\cdots,x_{k})=(\lambda(x_{0},x_{1}),\lambda(x_{1},x_{2}),\cdots,\lambda(x_{k-1},x_{k}))$$  
Un edge labelling est dit un EL-labelling, si pour tout $x,y\in P$ avec $x<y$ on a le deux propriétés suivants, 
\begin{enumerate}[i.]
\item Il existe une unique chaîne saturée $x_{0}<x_{1}<\cdots<x_{k}$ tel que  
la suite $\lambda(x_{0},x_{1},\cdots,x_{k})$ est non décroissant, 
c'est-à-dire tel que $\lambda(x_{0},x_{1})\leq\lambda(x_{1},x_{2})\leq\cdots\leq\lambda(x_{k-1},x_{k})$.
\item Toute autre chaîne saturée $y_{0}<y_{1}<\cdots<y_{k}$ différent de 
$x_{0}<x_{1}<\cdots<x_{k}$, est tel que, 
$\lambda(x_{0},x_{1},\cdots,x_{k})<_{L}\lambda(y_{0},y_{1},\cdots,y_{k})$, où $<_{L}$
 désigne l'ordre lexicographique:  $(a_{0},a_{1},\cdots,a_{k})<_{L}(b_{0},b_{1},\cdots,b_{k})$ si et seulement si 
$a_{i}<b_{i}$, $i=\text{min}\{j\leq k:a_{j}\neq b_{j}\}$.
\end{enumerate}
Alors nous allons dire que $P$ est lexicographique shellable ou El-shellable s'il possède un El-labelling.
\end{deftn}
\begin{theo}\label{theo:EL-SHE}
Soit $P$ un ensemble fini partiellement ordonné, gradué. Si $P$ est El-shellable alors $P$ est shellable. 
\end{theo}
\begin{proof}[\textbf{Preuve}.]
Voir théorème \textbf{(2.3)} de \textbf{\cite{Bjorner}}.
\end{proof}
\subsection{L'ordre de Bruhat sur les involutions du groupe sym\'etrique}
Cette section est issus-presque tel quel de l'annexe \textbf{A} de \textbf{\cite{AMR}}.\\

Soit $\mathfrak{S}_{N}$ le groupe sym\'etrique, c'est-\`a-dire l'ensemble des 
bijections de l'ensemble $\{1,\cdots,N\}$ dans lui m\^eme. C'est un groupe de
Coxeter, qui est donc muni de sa fonction longueur $l_{\mathfrak{S}}$ 
et de son ordre de Bruhat $\leq_{B}$ associ\'e.

Si $s\in\mathfrak{S}_{N}$, on d\'efinit l'ensemble des inversions de $s$,
\begin{align*}
\text{Inv}(s)=\{(i,j)\in\{1,\cdots,N\}^{2}:i<j\text{ et }s(i)>s(j)\} 
\end{align*}
et l'on note inv$(s)$ le cardinal de cet ensemble. La longueur dans $\mathfrak{S}_{N}$ est alors
définie par,
\begin{align*}
l_{\mathfrak{S}}(s)=\text{inv}(s), \quad(s\in\mathfrak{S}_{N}). 
\end{align*}
L'ordre de Bruhat $\leq_{B}$ de $\mathfrak{S}_{N}$, est la relation d'ordre partielle dans $\mathfrak{S}_{N}$ 
obtenue après fermeture transitive de la relation $\rightarrow$ définie par, 
\begin{align*}
\sigma\rightarrow \tau~\text{ ssi il existe une transposition } t_{i,j} \text{ tel que } 
\tau=\sigma t_{i,j}\text{ et }\text{inv}(\sigma)\leq\text{inv}(\tau).
\end{align*}
L'ordre de Bruhat est gradu\'e par la fonction longueur $l_{\mathfrak{S}}$, c'est-\`a-dire que,
\begin{enumerate}[i.]
 \item Si $s\leq_{B} s'$, alors $l_{\mathfrak{S}}(s)\leq l_{\mathfrak{S}}(s')$.
 \item Si $s\leq_{B} s'$, et $l_{\mathfrak{S}}(s)=l_{\mathfrak{S}}(s')$ alors $s=s'$.
 \item Si $s<_{B} s'$, et s'il n'existe pas $t\in\mathfrak{S}_{N}$ tel que 
 $s<_{B}<_{B}t<_{B} s'$, alors $l_{\mathfrak{S}}(s)= l_{\mathfrak{S}}(s')-1$.
\end{enumerate}
 Soit $s\in\mathfrak{S}_{N}$. On appelle mont\'ee de $s$ un couple $(i,j)\in\{1,\cdots,N\}^{2}$ tel
que $i<j$ et $s(i)<s(j)$. On dit qu'une mont\'ee $(i,j)$ de $s$ est libre s'il n'existe 
pas de $k$ tel que $i<k<j$ et $s(i)<s(k)<s(j)$. L'ordre de Bruhat est alors totalement d\'etermin\'e
par les propri\'et\'es \textbf{i.}, \textbf{ii.}, \textbf{iii.} et,
\begin{enumerate}
 \item[iv.] $s\leq_{B}s'$ avec $l_{\mathfrak{S}}(s)=l_{\mathfrak{S}}(s')-1$ si et seulements s'il 
 existe une mont\'ee libre $(i,j)$ de $s$ et $s'=st_{i,j}$, o\`u $t_{i,j}$ d\'esigne la transposition
 de $\mathfrak{S}_{N}$ \'echangeant $i$ et $j$.
 \end{enumerate}
En particulier $\mathfrak{S}_{N}$ poss\`ede un \'el\'ement minimal, de longueur nulle, \`a savoir
l'identit\'e, que nous notons Id, et un \'el\'ement maximal, $s_{\text{max}}$ de longueur $\frac{N(N-1)}{2}$,
d\'efinit par $t_{1,N}t_{2,N-1}\cdots t_{\frac{N}{2},\frac{N}{2}+1}$ si $N$ est pair et 
$t_{1,N}t_{2,N-1}\cdots t_{\frac{N-1}{2},\frac{N-1}{2}+1}$ si $N$ est impair, 
o\`u $t_{i,j}$ d\'esigne la transposition de $\mathfrak{S}_{N}$ \'echangeant $i$
et $j$.\\

Consid\'erons maintenant l'ensemble $\mathfrak{I}_{N}$ des involutions dans $\mathfrak{S}_{N}$, et munissons-le de l'ordre induit
par $\leq_{B}$. Il est lui-aussi munit d'une fonction longueur. Pour la d\'efinir, pour tout
$s\in \mathfrak{I}_{N}$ notons,
\begin{align*}
\text{Exc}(s)=\{i\in\{1,\cdots,N\}:s(i)>i\} 
\end{align*}
et exc$(s)$ le cardinal de cet ensemble. Posons alors,
\begin{align*}
l_{\mathfrak{I}}(s)=\frac{\text{inv}(s)+\text{exc}(s)}{2}. 
\end{align*}
L'ordre de Bruhat sur $\mathfrak{I}_{N}$ est alors gradu\'e par $l_{\mathfrak{I}}$, c'est-\`a-dire
que les propr\'et\'es \textbf{i.}, \textbf{ii.}, \textbf{iii.} et une propri\'et\'e 
analogue \`a \textbf{iv}, que nous n'avons pas besoin de expliciter,
sont v\'erifi\'es pour $\mathfrak{I}_{N}$ en rempla\c cant $l_{\mathfrak{S}}$
par $l_{\mathfrak{I}}$. Remarquons aussi que Id et $s_{\text{max}}$ sont dans $\mathfrak{I}_{N}$ et sont donc
respectivement les \'el\'ements minimaux et maximaux de $\mathfrak{I}_{N}$.\\

La longuer $l_{\mathfrak{S}}$ et $l_{\mathfrak{I}}$ satisfont le r\'esultat suivant;
\begin{lem}\label{lem:long}
Soient $s,s'\in\mathfrak{I}_{N}$ tels que $s\leq_{B}s'$ et $l_{\mathfrak{I}}(s)=l_{\mathfrak{I}}(s')-1$. Alors
\begin{align*}
l_{\mathfrak{S}}(s)-l_{\mathfrak{S}}(s')=,1,2\text{ ou }3. 
\end{align*}
\end{lem}
\begin{proof}[\textbf{Preuve}.]
Voir point \textbf{i} du lemme \textbf{(A.1)} de l'annexe \textbf{A} de \textbf{\cite{AMR}}.
\end{proof}
Notons à présent $Z=\{(i,j)\in \{1,\cdots,N\}^{2}:i<j\}$ et soit, $\lambda_{\mathfrak{I}}:C(\mathfrak{S}_{N})\rightarrow Z$
le labelling défini pour $(\sigma_{1},\sigma_{2})\in C(\mathfrak{S}_{N})$ avec $\sigma_{2}=\sigma_{1}t_{i,j}$ par,
$$\lambda_{\mathfrak{S}}(\sigma_{1},\sigma_{2})=(i,j).$$
P.H. Edelman démontre dans \textbf{\cite{Edelman}} (théorème à la fin de la page  356): 
\begin{theo}\label{theo:ELSHEsim}
L'ensemble $\mathfrak{S}_{N}$ est El-shellable (voir définition \textbf{(3.2)}) avec $\lambda_{\mathfrak{S}}$ comme El-labelling.
\end{theo}
De manière analogue, sur $\mathfrak{I}_{N}$ on peut en suivant le travaux de F.Incitti de définir, 
\begin{align*}
\lambda_{\mathfrak{I}}&:C(\mathfrak{I}_{N})\rightarrow Z\\
\lambda_{\mathfrak{I}}(\tau_{1},\tau_{2})&=(i,j),~\text{ où }\tau_{2}=ct_{i,j}\tau_{1}.
\end{align*}
avec l'élément $\tau_{2}=ct_{i,j}\tau_{1}$ donné comme dans la définition \textbf{(3.2)} de \textbf{\cite{Incitti}}.
Cette défintion permet ensuite à F.Incitti de demontrer (théorème \textbf{(6.2)} de \textbf{\cite{Incitti}}):  
\begin{theo}\label{theo:ELSHEinv}
L'ensemble $\mathfrak{I}_{N}$ est El-shellable (voir définition \textbf{(3.2)}) avec $\lambda_{\mathfrak{I}}$ comme El-labelling.
\end{theo} 
Le résultat qui suit est une consequ\'ence du théorème \textbf{\ref{theo:ELSHEsim}} et du théorème \textbf{\ref{theo:ELSHEinv}}:
\begin{lem}\label{lem:d2lt}
Soient $s,s'\in\mathfrak{S}_{N}$ (respectivement $s,s'\in\mathfrak{I}_{N}$) v\'erifiant $s\leq_{B}s'$ et $l_{\mathfrak{I}}(s)=l_{\mathfrak{I}}(s')-2$ (respectivement
$l_{\mathfrak{S}}(s)=l_{\mathfrak{S}}(s')-2$). 
Alors ils existent exactement deux \'el\'ements $s_{1}$ et $s_{2}$ dans $\mathfrak{S}_{N}$ (respectivement dans $\mathfrak{I}_{N}$) de longueur 
$l_{\mathfrak{S}}(s)+1$ (respectivement $l_{\mathfrak{I}}(s)+1$) tels que,
\begin{align*}
s<_{B}s_{1},~s_{2}<_{B}s'. 
\end{align*}
\end{lem}
\begin{proof}[\textbf{Preuve}.]
Voir lemme \textbf{(A.2)} et proposition \textbf{(A.3)} de l'annexe \textbf{A} de \textbf{\cite{AMR}}.
\end{proof}
Pour finir avec cette section, supposons $N$ pair. Alors dans $\mathfrak{I}_{N}$ consid\'erons le sous-ensemble 
$\mathfrak{I}_{N}^{s}$ des involutions sans points
fixes, envoyant $\{1,\cdots,N/2\}$ sur $\{N/2+1,\cdots,N\}$. Munissons le de l'ordre induit. 
Pour $\mathfrak{I}_{N}^{s}$ on a le r\'esultat suivant:
\begin{prop}\label{prop:intervalle}
L'ensemble $\mathfrak{I}_{N}^{s}$ est un intervalle dans $\mathfrak{I}_{N}$
avec $t_{1,N/2+1}t_{2,N/2+2}\cdots t_{N/2-1,N}$ comme \'el\'ement minimal et $s_{\text{max}}$ comme \'el\'ement maximal.
De plus le morphisme, 
\begin{align*}
\sigma:\mathfrak{I}_{N}&\rightarrow\mathfrak{S}_{N}\\
\sigma(s)(i)&=s(i)-\frac{N}{2}, \quad i\in\{1,\cdots,N\}, 
\end{align*}
d\'efinit un isomorphisme entre $\mathfrak{I}_{N}^{s}$ et le groupe sym\'etrique $\mathfrak{S}_{N/2}$.
\end{prop}
\begin{proof}[\textbf{Preuve}.]
Voir proposition \textbf{(A.4)} de l'annexe \textbf{A} de \textbf{\cite{AMR}}.
\end{proof}
\section{L'ensembles des représentations irréductibles du groupe lineaire}
\subsection{Paramétrisation de Beilinson-Bernstein pour $\textbf{GL}(N,\mathbb{R})$}
Fixons une repr\'esentation de dimension finie $F$ de $\mathbf{GL}(N,\mathbb{R})$
et consid\'erons l'ensemble $\Pi_{F}(\mathbf{GL}(N,\mathbb{R}))$ des
repr\'esentations de $\mathbf{GL}(N,\mathbb{R})$ dont le caract\`ere infinit\'esimal
est \'egal \`a celui de $F$. Le  caract\`ere infinit\'esimal de $F$ est entier et r\'egulier, 
il est donc donn\'e par $(\lambda_{1},\cdots,\lambda_{N})\in\mathbb{C}^{N}$ tel que
pour tout $i$ et $j$ distincts, $\lambda_{i}-\lambda_{j}\in\mathbb{N}^{\times}$.
On peut de plus permuter les $\lambda_{i}$ de sorte que pour tout $i\in\{1,\cdots,N\}$,
$\lambda_{i+1}-\lambda_{i}\in\mathbb{N}$. 

Soit $\mathscr{P}_{F}$ l'ensemble des param\`etres de Beilinson-Bernstein introduit dans \textbf{\cite{BB}}. Rappelons que $\mathscr{P}_{F}$ 
est constitu\'e par des couples $\gamma=(Q,\chi)$ o\`u $Q$ est une $\mathbf{O}(N,\mathbb{C})$-orbite dans $\mathscr{B}$, la variet\'e
de drapeaux de $\mathfrak{gl}_{n}$, et $\chi$ est un fibr\'e
en droite holomorphe plat $\mathbf{O}(N,\mathbb{C})$-homog\`ene sur $Q$. L'ensemble
$\mathscr{P}_{F}$ est, d'apr\`es le travaux de Beilinson-Bernstein, en bijection 
avec $\Pi_{F}(\mathbf{GL}(N,\mathbb{R}))$. 

On voudrait donner un param\`etrage combinatoire de $\mathscr{P}_{F}$. On sait que 
les $\mathbf{O}(N,\mathbb{C})$-orbites sur la vari\'et\'e des drapeaux $\mathscr{B}$
de $\mathbf{GL}(N,\mathbb{C})$ sont param\`etr\'ees par les involutions dans le groupe
sym\'etrique $\mathfrak{S}_{N}$. Pour \'etendre cette 
bijection en une param\`etrization de $\mathscr{P}_{F}$, introduisons 
l'ensemble $\mathfrak{I}_{N}^{\bullet,\pm}$ des involutions
avec point fixes sign\'es form\'es des couples $(\eta,f_{\eta})$ 
o\`u $\eta\in\mathfrak{I}_{N}$ est une involution dans $\mathfrak{S}_{N}$
et $f_{\eta}$ est une application de l'ensemble des pointes fixes de $\eta$ dans $\{\pm 1\}$. 
Cette donn\'ee supl\'ementaire conc\'ed\'ee 
par l'attribution d'un signe \`a chaque point fixe d'un \'el\'ement $\eta\in\mathfrak{I}_{N}$ 
est justement l'information qu'il faut 
pour construire une bijection entre $\mathfrak{I}_{N}^{\bullet,\pm}$ 
et $\Pi_{F}(\mathbf{GL}(N,\mathbb{R}))$, voir \textbf{\cite{AMR}}, section \textbf{(4.2)}, 
pour la construction de cet isomorphisme. 

Par cons\'equent entre $\mathfrak{I}_{N}^{\bullet,\pm}$ et $\Pi_{F}(\mathbf{GL}(N,\mathbb{R}))$ on obtient,
\begin{align*}
\mathfrak{I}_{N}^{\bullet,\pm}\cong\mathscr{P}_{F}\cong \Pi_{F}(\mathbf{GL}(N,\mathbb{R})). 
\end{align*}
Donnons une d\'escripton plus d\'etaill\'e de cette bijection entre
$\mathfrak{I}_{N}^{\bullet,\pm}$ et l'ensemble de repr\'esentations $\Pi_{F}(\mathbf{GL}(N,\mathbb{R}))$.

Sur $\mathbf{GL}(N,\mathbb{R})$ d\'efinissons,
\begin{align*}
\nu_{N}:\mathbf{GL}(N,\mathbb{R})&\rightarrow\mathbb{R}^{\ast}\\
g&\mapsto |\det g|, 
\end{align*}
et pour toute paire $t_{1}, t_{2}\in\mathbb{C}$ avec $t_{1}-t_{2}\in\mathbb{Z}$, soit,
\begin{align*}
\delta(t_{1},t_{2}), 
\end{align*}
la s\'erie discr\`ete de $GL(2,\mathbb{R})$ de caract\`ere infinit\'esimal $(s_{1},s_{2})$.

Prennons un \'el\'ement $\overline{\eta}=(\eta,f_{\eta})\in \mathfrak{I}_{N}^{\bullet,\pm}$. 
La repr\'esentation de $\Pi_{F}(G)$ associ\'e \`a $\overline{\eta}$ est obtenue
de la mani\`ere suivant; on d\'ecompose l'involution $\eta$ en cycles $\tau$: ceux-ci
sont des point fixes ou bien des transpositions. Chaque point fixe $i=\eta(i)$ donne une
repr\'esentation,
\begin{align*}
\delta_{\eta,(i)}=\epsilon_{i}\nu^{\lambda_{i}} 
\end{align*}
de $GL(1,\mathbb{R})$, o\`u $\epsilon_{i}\in\{\mathbf{Triv},\mathbf{sgn}\}$ d\'epend du signe $f_{\eta}(i)$ 
attach\'e au point fixe $i$ d'une manie\`ere que nous n'avons pas besoin d'expliciter, mais qui depend de la repr\'esentation
de dimension finie $F$ fix\'ee au d\'epart. Chaque transposition $(ij)$ avec $i<j$ donne une s\'erie discr\`ete
de $GL(2,\mathbb{R})$, 
\begin{align*}
\delta_{\eta,(ij)}=\delta(\lambda_{i},\lambda_{j}). 
\end{align*}
La repr\'esentation de $\Pi_{F}(G)$ associ\'ee \`a $\overline{\eta}$ est l'unique sous-module
irr\'eductible, not\'e $\overline{X}(\overline{\eta})$, du module standard,
\begin{align}\label{eq:modulestandardassocie}
X(\overline{\eta})=\times_{\tau\text{ cycle de }\eta}^{\rightarrow}\delta_{s,\tau}. 
\end{align} 
La fl\`eche au dessus du produit est pour indiquer que le produit est \'ecrit dans un ordre standard (Voir definition \textbf{(4.3)} de \textbf{\cite{AMR}}).

\subsection{Repr\'esentations $\theta$-invariants de $\mathbf{GL}(N,\mathbb{R})$}
Soit $J_{N}\in\mathbf{GL}(N,\mathbb{R})$ la matrice antidiagonale,
\begin{equation}\label{eq:mat}
J_{N}=\left(\begin{array}{cccc}
     ~ &~ &~ &1\\
      ~ &~   &-1  &~\\
  ~&{\mathinner{\mkern2mu\raise1pt\hbox{.}\mkern2mu
\newline \raise4pt\hbox{.}\mkern2mu\raise7pt\hbox{.}\mkern1mu}}&~ &~\\
(-1)^{N+1}&~&~&~
    \end{array}\right).
\end{equation}
Sur le groupe lin\'eaire $\mathbf{GL}(N)$ d\'efinissons l'automorphisme involutif,
\begin{align}\label{eq:theta}
\theta_{N}:\mathbf{GL}(N)&\rightarrow \mathbf{GL}(N),\\
                         g&\mapsto J_{N}^{-1}(^{t}g^{-1})J_{N}.
\end{align}
L'automorphisme $\theta_{N}$ nous am\`ene \`a d\'efinir;
\begin{deftn}
On dit d'une repr\'esentation $\pi$ de $\mathbf{GL}(N,\mathbb{R})$ qu'elle est $\theta_{N}$-invariante si $\pi$ est \'equivalente \`a 
$\pi^{\theta_{N}}:=\pi\circ\theta_{N}$.
\end{deftn}
Cette section a comme objectif, la description l'ensemble de repr\'esentation irr\'eductibles
et $\theta_{N}$-invariants de $\mathbf{GL}(N,\mathbb{R})$.

Fixons une repr\'esentation de dimension finie $F$ de $\mathbf{GL}(N,\mathbb{R})$ que l'on suppose
$\theta_{N}$ invariant et consid\'erons l'ensemble de repr\'esentations $\pi\in\Pi_{F}(\mathbf{GL}(N,\mathbb{R}))$. 
 D'apr\`es la caract\`erisation fait par l'equation \textbf{(\ref{eq:modulestandardassocie})} 
il existe un \'el\'ement $\overline{\eta}=(\eta,f_{\eta})\in\mathfrak{I}_{N}^{\bullet,\pm}$ 
de mani\`ere \`a ce que $\pi$ soit l'unique sous-module irr\'eductible du module standard,
\begin{align*}
X(\overline{\eta})&=\times_{\tau\text{ cycle de }\eta}^{\rightarrow}\delta_{\eta,\tau}\\ 
&=\text{Ind}_{P}(\otimes_{\tau\text{ cycle de }\eta}^{\rightarrow}\delta_{\eta,\tau}),
\end{align*}
o\`u $P$ est un sous-groupe parabolique de levi,
\begin{align*}
M=\Pi_{\tau\text{ cycle de }\eta}^{\rightarrow}\mathbf{GL}(n_{\tau},\mathbb{R}),\quad n_{\tau}=\left\{\begin{array}{cc}
            1, &\text{si }\tau\text{ est un point fixe},\\
            2, &\text{si }\tau\text{ est une transposition}
           \end{array}\right. 
\end{align*}
contenant l'ensemble des matrices triangulaires sup\'erieures de $\mathbf{GL}(N,\mathbb{R})$.

Consid\'erons l'op\'erateur d\'efini pour tout $f\in X(\eta),~g\in \mathbf{GL}(N,\mathbb{R})$ par,
\begin{align*}
\vartheta(f)(g)=f(\theta_{N}(g)).
\end{align*}
Il est facile de voir que $\vartheta$ est un isomorphisme qui entrelace, 
\begin{align*}
X(\overline{\eta})^{\theta_{N}}\quad\text{et}\quad\text{Ind}_{\theta_{N}(P)}((\otimes_{\tau\text{ cycle de }\eta}^{\rightarrow}\delta_{\eta,\tau})^{\theta_{N}}).
\end{align*}
Comme en plus $\pi^{\theta_{N}}\cong\overline{\text{Ind}_{\theta_{N}(P)}(\pi_{M}^{\theta_{N}})}$ on obtient que $\pi$ est isomorphe \`a
$\pi^{\theta_{N}}$ si et seulement si,
\begin{align*}
\theta_{N}(P)=P\quad\text{et}\quad\otimes_{\tau\text{ cycle de }\eta}^{\rightarrow}\delta_{\eta,\tau}\cong
(\otimes_{\tau\text{ cycle de }\eta}^{\rightarrow}\delta_{\eta,\tau})^{\theta_{N}}. 
\end{align*}
Consid\'erons maintenant l'\'el\'ement $\overline{\eta}^\theta=(\eta^{\theta},f_{\eta^{\theta}})\in\mathfrak{I}_{N}^{\bullet,\pm}$ 
d\'efini pour tout $i\in\{1,\cdots,N\}$ par,
\begin{align*}
\eta^{\theta}(i)&=j~\text{ si et seulement si }~\eta(n+1-i)=n+1-j,\\
f_{\eta^{\theta}}(i)&=f_{\eta^{\theta}}(n+1-i),~\text{ si }i\text{ est un point fixe}.
\end{align*}
Alors on peut facilement v\'erifier que,
\begin{align*} 
(\otimes_{\tau\text{ cycle de }\eta}^{\rightarrow}\delta_{\eta,\tau})^{\theta_{N}}=
\otimes_{\tau'\text{ cycle de }\eta^{\theta}}^{\rightarrow}\delta_{\eta^{\theta},\tau'}
\end{align*}
et que $\theta_{N}(P)=P'$ avec $P'$ le un sous-groupe parabolique de levi,
\begin{align*}
M'=\Pi_{\tau'\text{ cycle de }\eta^{\theta}}^{\rightarrow}\mathbf{GL}(n_{\tau'},\mathbb{R}),\quad n_{\tau'}=\left\{\begin{array}{cc}
            1, &\text{si }\tau'\text{ est un point fixe},\\
            2, &\text{si }\tau'\text{ est une transposition}
           \end{array}\right. 
\end{align*}
contenant l'ensemble des matrices triangulaires sup\'erieures de $\mathbf{GL}(N,\mathbb{R})$.
En cons\'equence, 
\begin{align*}
X(\overline{\eta})^{\theta_{N}}\cong X(\overline{\eta}^{\theta}), 
\end{align*}
et d'apr\`es la bijection $\mathfrak{I}_{N}^{\bullet,\pm}\rightarrow\Pi_{F}(GL(N,\mathbb{R}))$, 
on obtient $\pi^{\theta}\cong \overline{X}(\overline{\eta}^{\theta})$. 
L'\'effet donc de $\theta_{N}$ sur $\mathfrak{I}_{N}^{\bullet,\pm}$ 
via la bijection $\mathfrak{I}_{N}^{\bullet,\pm}\cong \Pi_{F}(GL(N,\mathbb{R}))$
est donn\'e pour tout $\overline{\eta}=(\eta,f_{\eta})\in\mathfrak{I}_{N}^{\bullet,\pm}$ par, 
\begin{align}\label{eq:actionthetasurInvo}
\overline{\eta}\mapsto\theta_{N}(\overline{\eta})=\overline{\eta}^\theta=(\eta^{\theta},f_{\eta^{\theta}}).
\end{align}
D'\`ou $\pi=\overline{X}(\overline{\eta})\cong \overline{X}(\overline{\eta}^{\theta})=\pi^{\theta_{N}}$ 
si et seulement si $\overline{\eta}=\theta_{N}(\overline{\eta})=\overline{\eta}^{\theta}$.
\subsection{Le module de Speh}
Soit $p,n\in\mathbb{N}^{\times}$ tels que $p>n-1$. Notons par
$\delta=\delta(p/2,-p/2)$ la s\'erie discr\`ete de $\mathbf{GL}(2,\mathbb{R})$ 
de caract\`ere infinit\'esimal $(p/2,-p/2)$ et consid\'erons le module standard de $\mathbf{GL}(2n,\mathbb{R})$,
\begin{align}\label{eq:modulemaximal}
\delta\nu^{-\frac{n-1}{2}}\times\delta\nu^{-\frac{n-3}{2}}\times\cdots\times\delta\nu^{\frac{n-1}{2}}. 
\end{align}
L'unique sous module irr\'eductible de \textbf{(\ref{eq:modulemaximal})} on le note $\textbf{Speh}(\delta,n)$. 
Le caract\`ere infinit\'esimal de $\textbf{Speh}(\delta,n)$ est donn\'e par,
\begin{align}\label{eq:c.i.speh}
\left(\frac{-p-(n-1)}{2},\frac{-p-(n-3)}{2},\cdots,\frac{p+(n-1)}{2}\right) 
\end{align}
et la condition $p>n-1$ assure que ce caract\`ere infinit\'esimal est bien entier et r\'egulier.
C'est donc le caract\`ere infinit\'esimal d'une repr\'esentation de dimension finie $F$ de $\mathbf{GL}(2n,\mathbb{R})$.
Rappelons la bijection entre l'ensemble des involutions avec points fixes sign\'es 
$\mathfrak{I}_{N}^{\bullet,\pm}$ et l'ensemble 
de repr\'esentations $\Pi_{F}(\mathbf{GL}(N,\mathbb{R}))$.
 Dans cette bijections, $\textbf{Speh}(\delta,n)$ correspond \`a l'involution 
sans points fixes $s_{0}$ d\'efinie par,
\begin{align*}
s_{0}=n+i,\quad i=1,2,\cdots,n-1. 
\end{align*}
L'involution $s_{0}$ est l'\'el\'ement minimal dans $\mathfrak{I}_{N}^{s}$, 
l'ensemble des involutions sans point fixes envoyant $\{1,\cdots,N/2\}$ sur $\{N/2+1,\cdots,N\}$.
D'apr\`es la proposition \textbf{(\ref{prop:intervalle})}, $\mathfrak{I}_{N}^{s}$ est un intervalle dans $\mathfrak{I}_{N}$ 
avec $s_{\text{max}}$ comme \'el\'ement maximal. 
De plus $\mathfrak{I}_{N}^{s}\cong\mathfrak{S}_{n}$ et comme $\mathfrak{I}_{N}^{s}$ est 
naturellement un sous-ensemble de $\mathfrak{I}_{N}^{\bullet,\pm}$, 
la bijection $\mathfrak{I}_{N}^{\bullet,\pm}\rightarrow \Pi_{F}(\mathbf{GL}(N,\mathbb{R}))$ 
associe, par la formule \textbf{(\ref{eq:modulestandardassocie})}, \`a tout \'el\'ement 
$s\in\mathfrak{S}_{n}\cong \mathfrak{I}_{N}^{s}$ l'unique sous-module irr\'eductible $\overline{X}(s)$ 
de la repr\'esentation standard de $\mathbf{GL}(2n,\mathbb{R})$ don\'ee par,
\begin{align}\label{eq:modulestandardspeh}
X(s)=\times_{i=1}^{n}\delta\left(\frac{-p-(n-1)}{2}+(i-1),\frac{p-(n-1)}{2}+(s(i)-1)\right). 
\end{align}
Parlons un peut sur l'ensemble des sous-quotients irr\'eductibles de 
$X(s_{0})$.
Pour pouvoir faire ceci notons que l'ordre de Bruhat sur
$\mathfrak{I}_{N}$ est lie au 
ordre naturel sur $\mathbf{O}(N,\mathbb{C})\setminus\mathscr{B}$ (Voir \textbf{\cite{AMR}} equation \textbf{(2.1)} 
pour une description de cette ordre) 
par le r\'esulat suivant;
\begin{prop}
L'ordre naturel sur $\mathbf{O}(N,\mathbb{C})\setminus\mathscr{B}$ co\"incide 
 avec l'ordre de Bruhat invers\'e sur $\mathfrak{I}_{N}$.  
\end{prop}
\begin{proof}[\textbf{Preuve}.]
Voir \textbf{\cite{Richardson-Springer}} section \textbf{(10.2)}.
\end{proof}
Cette derni\`ere proposition unie \`a la proposition \textbf{(\ref{prop:intervalle})} a 
comme cons\'equence le r\'esultat suivant;
\begin{prop}
La repr\'esentation $\pi$ de $\mathbf{GL}(2n,\mathbb{R})$ est un sous-quotient irr\'eductible de $I(\delta,n)$ si et seulement si 
il existe une permutation $s\in\mathfrak{S}_{n}$ tel que $\pi\cong\overline{X}(s)$.
\end{prop}

Dans la section pr\'ecédent on a vu que l'action de $\theta_{N}$ sur l'ensemble
$\mathfrak{I}_{N}^{\bullet,\pm}$ est donn\'ee par 
l'\'equation \textbf{(\ref{eq:actionthetasurInvo})}. 
 En particulier, $\mathfrak{I}_{N}^{s}$ est stabilis\'e par $\theta_{N}$ et comme  
$\mathfrak{I}_{N}^{s}\cong\mathfrak{S}_{n}$, il est facil de voir que l'action induite par $\theta_{N}$ sur ce dernier ensemble est,
pour tout $s\in\mathfrak{S}_{N}$, donn\'ee par,
\begin{align*}
s\mapsto s^{-1}. 
\end{align*}
Par cons\'equent $X(s)\cong X(s)^{\theta_{N}}\cong X(s^{-1})$ si et seulement si 
$s^{2}=$Id, c'est-\`a-dire si et seulement si $s\in\mathfrak{I}_{n}$.\\

Pour finir avec cette section traduisson l'effet du lemme \textbf{(\ref{lem:d2lt})}  
sur l'ensemble des modules standards $X(s),~s\in\mathfrak{I}_{n}$. 
Faisons d'abord noter que, si $s\geq s'$ et $l_{\mathfrak{S}}(s)=l_{\mathfrak{S}}(s')+1$,
alors $s=s't_{i,j}$ pour une certaine mont\'ee libre $(i,j)$, voir point \textbf{iv.} à la fin de la page 12. 
Par cons\'equent si nous ecrivons, 
\begin{align*}
X(s')=\delta_{1}\times\cdots\times\delta_{i}\times\cdots\times\delta_{j}\times\cdots\delta_{n},
\end{align*}
avec,
\begin{align*}
\delta_{i}&=\delta\left(\frac{-p-(n+1)+2i}{2},\frac{-p-(n+1)+2s'(i)}{2}\right),\\ 
\delta_{j}&=\delta\left(\frac{-p-(n+1)+2j}{2},\frac{-p-(n+1)+2s'(j)}{2}\right). 
\end{align*}
Pour $X(s)$ on a, 
\begin{align*}
X(s)=\delta_{1}\times\cdots\times\delta_{i}\times\cdots\times\delta_{j}\times\cdots\delta_{n} 
\end{align*}
avec,
\begin{align*}
\delta_{i}&=\delta\left(\frac{-p-(n+1)+2i}{2},\frac{-p-(n+1)+2s'(j)}{2}\right),\\ 
\delta_{j}&=\delta\left(\frac{-p-(n+1)+2j}{2},\frac{-p-(n+1)+2s'(i)}{2}\right). 
\end{align*}
Consid\'erons maintenant $s,s'\in\mathfrak{I}_{n}$ avec 
$s\geq s'$ et $l_{\mathfrak{I}}(s)=l_{\mathfrak{I}}(s')+1$. 
Alors pour $X(s)$ et $X(s')$ on a trois posibilit\'es,\\

\textbf{i.} Si $l_{\mathfrak{S}}(s')=l_{\mathfrak{S}}(s)-1$ 
il existe une mont\'ee libre $(i,j)$ avec $s'(i)=j$ de fa\c con \`a ce que, 
$$s=s't_{i,j}.$$ 
Par cons\'equent si $X(s')=\times_{k=1}^{n}\delta_{k}'$ pour $X(s)$ on obtient 
$X(s)=\times_{k=1}^{n}\delta_{k}$ avec $\delta_{k}'=\delta_{k}$ si $k\neq i,j$ et, 
\begin{align*}
\delta_{i}'&=\delta\left(\frac{-p-(n+1)+2i}{2},-\frac{-p-(n+1)+2i}{2}\right),\\ 
\delta_{j}'&=\delta\left(\frac{-p-(n+1)+2j}{2},-\frac{-p-(n+1)+2j}{2}\right). 
\end{align*}

\textbf{ii.} Si $l_{\mathfrak{S}}(s')=l_{\mathfrak{S}}(s)-2$ alors pour $X(s)$ et $X(s')$ on a une des deux posibilit\'es suivants.\\
\textbf{a.)} Il existe une mont\'ee libre $(i,j)$ avec $s'(i)\neq i$ et $s'(j)\neq j$
de fa\c con \`a ce que, 
$$s=s't_{i,j}t_{s'(j),s'(i)}.$$ 
Par cons\'equent si  $X(s')=\times_{k=1}^{n}\delta_{k}'$ pour $X(s)$ on obtient 
$X(s)=\times_{k=1}^{n}\delta_{k}$ avec $\delta_{k}'=\delta_{k}$ si $k\neq i,j,s'(i),s'(j)$ et, 
\begin{align*}
\delta_{i}&=\delta\left(\frac{-p-(n+1)+2i}{2},-\frac{-p-(n+1)+2s'(j)}{2}\right), \\
\delta_{s'(i)}&=\delta\left(\frac{-p-(n+1)+2s'(i)}{2},-\frac{-p-(n+1)+2j}{2}\right),\\ 
\delta_{j}&=\delta\left(\frac{-p-(n+1)+2j}{2},-\frac{-p-(n+1)+2s'(i)}{2}\right), \\
\delta_{s'(j)}&=\delta\left(\frac{-p-(n+1)+2s'(j)}{2},-\frac{-p-(n+1)+2i}{2}\right). 
\end{align*}
\textbf{b.} Il existe une mont\'ee libre $(i,j)$ avec $s'(i)\neq i$ et $s'(j)=j$ de fa\c con \`a ce que,
\begin{align*} 
s=s't_{i,j}t_{j,s'(i)}\text{ ou }s=s't_{i,j}t_{i,s'(i)}. 
\end{align*}
Par cons\'equent si $X(s')=\times_{k=1}^{n}\delta_{k}$, pour $X(s)$ on obtient 
$X(s)=\times_{k=1}^{n}\delta_{k}$ avec $\delta_{k}'=\delta_{k}$ si $k\neq i,j,s'(i)$ et, 
\begin{align*}
\delta_{i}&=\delta\left(\frac{-p-(n+1)+2i}{2},\frac{-p-(n+1)+2j}{2}\right), \\
\delta_{s'(i)}&=\delta\left(\frac{-p-(n+1)+2s'(i)}{2},\frac{-p-(n+1)+2s'(i)}{2}\right),\\ 
\delta_{j}&=\delta\left(\frac{-p-(n+1)+2j}{2},\frac{-p-(n+1)+2i}{2}\right),\text{ si }s=s't_{i,j}t_{j,s'(i)}, 
\end{align*}
ou,
\begin{align*}
\delta_{i}&=\delta\left(\frac{-p-(n+1)+2i}{2},\frac{-p-(n+1)+2i}{2}\right), \\
\delta_{s'(i)}&=\delta\left(\frac{-p-(n+1)+2s'(i)}{2},\frac{-p-(n+1)+2j}{2}\right),\\ 
\delta_{j}&=\delta\left(\frac{-p-(n+1)+2j}{2},\frac{-p-(n+1)+2s'(i)}{2}\right), \text{ si }s=s't_{i,j}t_{i,s'(i)}. 
\end{align*}
\textbf{iii.} Si $l_{\mathfrak{S}}(s')=l_{\mathfrak{S}}(s)-3$ il existe une mont\'ee libre $(i,j)$ avec $s'(i)\neq i$,  $s'(j)=j$ 
et un entier $l\in\{1,\cdots,n\}$ tel que $s'(l)=l$ et $(l,s(i))$ est une mont\'ee libre,
de fa\c con \`a ce que, 
$$s=s't_{i,j}t_{l,s(i)}t_{j,l}.$$ 
Par cons\'equent si $X(s')=\times_{k=1}^{n}\delta_{k}'$ pour $X(s)$ on obtient 
$X(s)=\times_{k=1}^{n}\delta_{k}$ avec $\delta_{k}'=\delta_{k}$ si $k\neq i,s(i),j,l$ et, 
\begin{align*}
\delta_{i}&=\delta\left(\frac{-p-(n+1)+2i}{2},-\frac{-p-(n+1)+2j}{2}\right), \\
\delta_{j}&=\delta\left(\frac{-p-(n+1)+2j}{2},-\frac{-p-(n+1)+2i}{2}\right),\\ 
\delta_{l}&=\delta\left(\frac{-p-(n+1)+2l}{2},-\frac{-p-(n+1)+2s'(i)}{2}\right), \\
\delta_{s'(i)}&=\delta\left(\frac{-p-(n+1)+2s'(i)}{2},\frac{-p-(n+1)+2l}{2}\right). 
\end{align*}
\section{L'ensembles des représentations irréductibles du groupe unitaire}
\subsection{Paramétrisation de Beilinson-Bernstein pour $\textbf{U}(p,q,\mathbb{R})$}
Int\'eressons-nous \`a pr\'esent au groupe unitaire $\mathbf{U}(p,q)$. Posons $p+q=N$
et fixons une repr\'esentation $F$ de dimension finie de $\mathbf{U}(p,q,\mathbb{R})$.
Rappelons que l'ensemble de repr\'esentations 
$\Pi_{F}(\mathbf{U}(p,q,\mathbb{R}))$ est param\`etre par l'ensemble $\mathscr{P}_{F}$
des param\`etres de Beilinson-Bernstein. 
Celui-ci est constitu\'e par des couples $(Q,\chi)$ avec $Q$ une
$\mathbf{GL}(p,\mathbb{C})\times \mathbf{GL}(q,\mathbb{C})$-orbite dans 
$\mathscr{B}$, la variete de drapeaux de $\mathfrak{u}(p,q)$, et $\chi$ un fibr\'e en droite
holomorphe plat $\mathbf{GL}(p,\mathbb{C})\times \mathbf{GL}(q,\mathbb{C})$-homog\`ene sur 
$Q$. Pour $\mathbf{U}(p,q)$ un ph\'enom\`ene assez remarquable se produit: Tout fibr\'e en droite 
est conjugu\'e au fibr\'e trivial, d'o\`u $\mathscr{P}_{F}\cong\mathbf{GL}(p,\mathbb{C})\times \mathbf{GL}(q,\mathbb{C})\setminus\mathscr{B}$
et donc,
\begin{align}\label{eq:premierparamunitaire}
\Pi_{F}(\mathbf{U}(p,q,\mathbb{R}))\cong\mathbf{GL}(p,\mathbb{C})\times \mathbf{GL}(q,\mathbb{C})\setminus\mathscr{B}. 
\end{align}
Donnons comme pour $\mathbf{GL}(N)$ un param\'etrage combinatoire de 
$\mathbf{GL}(p,\mathbb{C})\times \mathbf{GL}(q,\mathbb{C})\setminus\mathscr{B}$.
Introduisons pour cela l'ensemble $\mathfrak{I}_{N}^{p,q,\pm}$
dont les \'el\'ements sont des couples $(\eta, f_{\eta})$, o\`u $\eta\in\mathfrak{I}_{N}$
et $f_{\eta}$ est une application de l'ensemble des points fixes de $\eta$ \`a valeurs
dans $\{\pm1\}$ qui v\'erifie, si l'on note $m$ le nombre de $2$-cycles dans la d\'ecomposition en cycles de $\eta$
\begin{align}\label{eq:fsupplementaire}
m+|f_{\eta}^{-1}(\{1\})|=p,\quad m+|f_{\eta}^{-1}(\{-1\})|=q. 
\end{align}
Adoptons une notation symbolique pour les éléments de 
$\mathfrak{I}_{n}^{p,q,\pm}$. Soit $\overline{\eta}=(\eta,f_{\eta})\in\mathfrak{I}_{n}^{p,q,\pm}$, 
alors $\overline{\eta}$ on l'identifie à la $n$-tuple $(c_1,\cdots,c_{n})$, où $c_{i}=f_{\eta}(i)=\pm$ si $\eta(i)=i$, et 
$c_{i}=c_{j}\in\mathbb{N}$ si $\eta(i)=j$, avec tout nombre naturel apparaissent exactament deux fois. 

Par exemple, pour $(p,q)=(2,1)$, on a, 
$$\mathfrak{I}_{3}^{2,1,\pm}=\{(1+1),(11+),(+11),(++-),(+-+),(++-)\}.$$
A.Yamamoto, dans \textbf{\cite{Yamamoto}}, th\'eor\`eme \textbf{(2.2.8)}, donne une param\'etrisation  
de $\mathbf{GL}(p,\mathbb{C})\times \mathbf{GL}(q,\mathbb{C})\setminus\mathscr{B}$ 
par $\mathfrak{I}_{N}^{p,q,\pm}$. Ce qui a comme cons\'equence la suite de bijections suivants,
\begin{align}\label{eq:prgu}
\mathfrak{I}_{N}^{p,q,\pm}\cong \mathbf{GL}(p,\mathbb{C})\times \mathbf{GL}(q,\mathbb{C})\setminus\mathscr{B} \cong \Pi_{F}(\mathbf{U}(p,q,\mathbb{R})). 
\end{align}
Dans cette description, la partition de $\Pi_{F}(\mathbf{U}(p,q,\mathbb{R}))$
en $L$-paquets est particulièrement simple, deux param\`etres $\overline{\eta}$ et $\overline{\eta}'\in\mathfrak{I}_{N}^{p,q,\pm}$ 
correspondent \`a des repr\'esentations dans le m\^eme $L$-paquet si et seulement si les involutions sous-jacentes $\eta$ et $\eta'$ dans $\mathfrak{I}_{N}$ sont \'egales. 

De plus l'isomorphisme $\mathfrak{I}_{N}^{p,q,\pm}\cong \mathbf{GL}(p,\mathbb{C})\times \mathbf{GL}(q,\mathbb{C})\setminus\mathscr{B}$
permet par transport de structure, d\'efinir \`a partir de l'ordre naturel sur 
$\mathbf{GL}(p,\mathbb{C})\times \mathbf{GL}(q,\mathbb{C})\setminus\mathscr{B}$ un ordre sur $\mathfrak{I}_{N}^{p,q,\pm}$. 
Cet ordre sur $\mathfrak{I}_{N}^{p,q,\pm}$ induit un ordre sur $\mathfrak{I}_{N}$ via la projection naturelle
$\mathfrak{I}_{N}^{p,q,\pm}\rightarrow\mathfrak{I}_{N}$ et il est possible de montrer 
qu'il co\"incide avec le ordre de Bruhat, $\leq_{B}$, sur $\mathfrak{I}_{N}$. Notons, 
comme expliqué dans \textbf{\cite{Yamamoto}}, que les ar\^etes du diagramme de Hasse de $\mathfrak{I}_{N}^{p,q,\pm}$ 
sont obtenues lorsque l'écriture symbolique d'un élément $\overline{\eta}'\in\mathfrak{I}_{N}^{p,q,\pm}$ 
est obtenue à partir de celle pour $\overline{\eta}\in\mathfrak{I}_{N}^{p,q,\pm}$ soit en remplaçant
deux symboles consécutifs de la forme $aa$, $a\in\mathbb{N}$ par $(+-)$ ou $(-+)$, soit en permutant
deux symboles consécutifs qui ne sont pas tout deux des signes, il faut ensuite ajouter toutes les ar\^etes
obtenues par la condition d'echange de \textbf{\cite{Trapa}} (figure \textbf{(3)}). 
Par exemple, pour $(p,q)=(2,1)$, le diagramme de Hasse de 
$\mathfrak{I}_{3}^{2,1,\pm}$ est donné par,
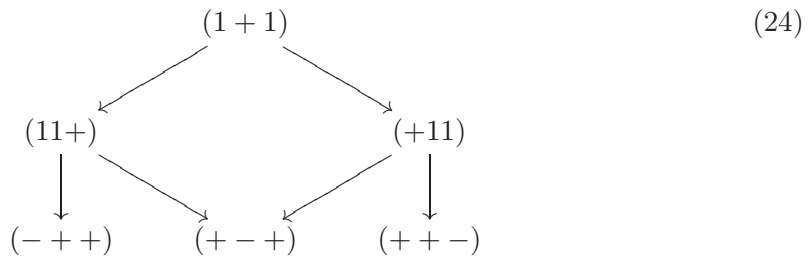
\begin{align}\label{eq:Hasse}
 \xymatrix {\relax &(1+1)\ar[dl] \ar[dr]&\\ 
(11+)\ar[d]\ar[dr]& \empty &(+11) \ar[dl]\ar[d]\\
(-++)& (+-+)& (++-)}
 \end{align}  
Finissons cette section avec le résultat suivant:
\begin{prop}\label{prop:multipliciteinvo}
Soient $\overline{\eta}=(\eta,f_{\eta}),~
\overline{\eta}'=(\eta',f_{\eta'})\in\mathfrak{I}_{N}^{p,q,\pm}$ 
tels que $l_{\mathfrak{I}}(\eta')=l_{\mathfrak{I}}(\eta)-1$. Alors,
la multiplicité $m(\overline{X(\overline{\eta})},X(\overline{\eta}'))$ 
de $\overline{X(\overline{\eta})}$ dans la suite de Jordan-Holder de ${X(\overline{\eta}')}$, est différent de zéro,
si et seulement si $\overline{\eta'}<_{B}\overline{\eta}$.
\end{prop}
\begin{proof}[\textbf{Preuve}.]
D'après le corollaire \textbf{(7.2)} de \textbf{\cite{VoganIII}},  pour $\overline{X(\overline{\eta}')}$ on a dans le groupe de Grothendieck
des représentations de $\mathbf{U}(p,q)(\mathbb{R})$, l'égalité,
\begin{align*}
\overline{X(\overline{\eta}')}=\sum_{\substack{\overline{\eta}\in\mathfrak{I}_{N}^{p,q,\pm}\\ l_{\mathfrak{I}}(\eta')<l_{\mathfrak{I}}(\eta)}}(-1)^{l(\eta)-l(\eta')}P_{\overline{\eta},\overline{\eta}'}(1)X(\overline{\eta}),
\end{align*}
où $P_{\overline{\eta},\overline{\eta}'}$ est le polyn\^ome de Kazhdan-Lusztig-Vogan associé au couple $(\overline{\eta},\overline{\eta}')$.
Comme $P_{\overline{\eta}',\overline{\eta}'}=1$, pour $X(\overline{\eta}')$ on obtient,
 \begin{align*}
X(\overline{\eta}')&=\overline{X(\overline{\eta}')}+\sum_{\substack{\overline{\eta}\in\mathfrak{I}_{N}^{p,q,\pm}\\l_{\mathfrak{I}}(\eta')=l_{\mathfrak{I}}(\eta)-1}}
P_{\overline{\eta},\overline{\eta}'}(1)X(\overline{\eta})+
\sum\limits_{\substack{\overline{\eta}\in\mathfrak{I}_{N}^{p,q,\pm}\\ l_{\mathfrak{I}}(\eta')<l_{\mathfrak{I}}(\eta)-1}}(-1)^{l_{\mathfrak{I}}(\eta)-l_{\mathfrak{I}}(\eta)'+1}P_{\overline{\eta},\overline{\eta}'}(1)X(\overline{\eta})\\
&=\overline{X(\overline{\eta}')}+\sum_{\substack{\overline{\eta}\in\mathfrak{I}_{N}^{p,q,\pm}\\l_{\mathfrak{I}}(\eta')=l_{\mathfrak{I}}(\eta)-1}}P_{\overline{\eta},\overline{\eta}'}(1)\overline{X(\overline{\eta})}\\
&\qquad\qquad\qquad\qquad\qquad+\sum_{\substack{\overline{\eta}\in\mathfrak{I}_{N}^{p,q,\pm}\\l_{\mathfrak{I}}(\eta')=l_{\mathfrak{I}}(\eta)-1}}\sum\limits_{\substack{\overline{\zeta}\in\mathfrak{I}_{N}^{p,q,\pm}\\ l_{\mathfrak{I}}(\eta)<l_{\mathfrak{I}}(\zeta)-1}}(-1)^{l_{\mathfrak{I}}(\zeta)-l_{\mathfrak{I}}(\eta)+1}P_{\overline{\zeta},\overline{\eta}}(1)X(\overline{\zeta})\\
&\qquad\quad~~~+\sum\limits_{\substack{\overline{\eta}\in\mathfrak{I}_{N}^{p,q,\pm}\\ l_{\mathfrak{I}}(\eta')<l_{\mathfrak{I}}(\eta)-1}}(-1)^{l_{\mathfrak{I}}(\eta)-l_{\mathfrak{I}}(\eta)'+1}P_{\overline{\eta},\overline{\eta}'}(1)X(\overline{\eta}).
\end{align*}
D'où, $m(\overline{X(\overline{\eta})},X(\overline{\eta}'))=P_{\overline{\eta},\overline{\eta}'}(1)$ pour tout $\overline{\eta}$ 
avec $l_{\mathfrak{I}}(\eta')=l_{\mathfrak{I}}(\eta)-1$. 
Mais, de la proposition \textbf{(6.14)} de \textbf{\cite{VoganIII}} on peut en déduire pour tout $\overline{\eta}$ 
avec $l_{\mathfrak{I}}(\eta')=l_{\mathfrak{I}}(\eta)-1$ que, $P_{\overline{\eta},\overline{\eta}'}(1)\neq 0$ si et seulement si 
$\overline{\eta}'<_{B}\overline{\eta}$. On peut dire m\^eme plus,
car d'après la formule recursive donne par cette proposition on conclut pour $\overline{\eta}'<_{B}\overline{\eta}$ avec $l_{\mathfrak{I}}(\eta')=l_{\mathfrak{I}}(\eta)-1$
que $P_{\overline{\eta},\overline{\eta}'}=P_{\overline{\eta}',\overline{\eta}'}=1$. 
\end{proof}
\section{Un complexe diff\'erentiel pour les modules de Speh \text{$\theta$-invariants}}
\subsection{Construction du complexe diff\'erentiel}
Soit $p,n\in\mathbb{N}^{\times}$ tels que $p>n-1$ et soit 
$\textbf{Speh}(\delta,n)$ la repr\'esentation de Speh de $\mathbf{GL}(2n,\mathbb{R})$ d\'efinie dans la section \textbf{(4.3)}.

Dans la section \textbf{(4.3)}, nous avons vu que les repr\'esentations plus
petits que $\textbf{Speh}(\delta,n)$ correspondent, 
via la bijection $\mathfrak{I}_{2n}^{\bullet,\pm}\longrightarrow\mathbf{GL}(2n,\mathbb{R})$, 
au sous-ensemble $\mathfrak{I}_{2n}^{s}\cong\mathfrak{S}_{n}$ de $\mathfrak{I}_{2n}^{\bullet,\pm}$.
D'apr\`es le th\'eor\`eme \textbf{(1)} de la section \textbf{(6)} de \textbf{\cite{Johnson}}, 
ces param\`etres coincident avec ceux des repr\'esentations standard intervenant dans
la r\'esolution de Johnson de $\textbf{Speh}(\delta,n)$. 
 Par cons\'equent pour $\textbf{Speh}(\delta,n)$ 
 la r\'esolution de Johnson s'\'ecrit,
\begin{align}
0\rightarrow \textbf{Speh}(\delta,n)\rightarrow I(\delta,n)\rightarrow X_{1}\rightarrow\cdots\rightarrow
X_{i-1}\xrightarrow{\phi_{i}^{J}} X_{i}\cdots\rightarrow X_{\frac{n(n-1)}{2}}\rightarrow 0  
\end{align}
o\`u,
\begin{align*}
X_{i}=\oplus_{s\in\mathfrak{S}_{n},l_{\mathfrak{S}}(s)=i}X(s) 
\end{align*}
avec pour tout $s\in\mathfrak{S}_{n}$, $X(s)$ d\'efini par \textbf{(\ref{eq:modulestandardspeh})} 
et pour chaque $i=1,\cdots,\frac{n(n-1)}{2}$, $\phi_{i}^{J}:X_{i-1}\longrightarrow X_{i}$ d\'efini
par,
\begin{align}\label{eq:morJ}
\phi_{i}^{J}=\oplus_{s\in\mathfrak{S}_{n},~l_{\mathfrak{S}}(s)=j}\oplus_{s'\in\mathfrak{S}_{n},~l_{\mathfrak{S}}(s)=j-1}\phi_{s,s'}^{J},\quad\phi_{s,s'}^{J}:X(s')\longrightarrow X(s).
\end{align}
Le morphismme $\phi_{s,s'}^{J}$ on l'appelle le morphisme de Johnson entre $X(s')$ et $X(s)$. 
D'apr\`es le r\'esulat suivant tout morphisme entre $X(s')$ et $X(s)$ est 
proportionnel au morphisme de Johnson.
\begin{lem}\label{lem:defmorsection2.5}
Soient $s$ et $s'$ une pair d'\'el\'ements de $\mathfrak{S}_{n}$, verifiant $l_{\mathfrak{S}}(s')<l_{\mathfrak{S}}(s)$.
Si $s'\nleq_{B} s$ alors $\dim\mathrm{Hom}(X(s'),X(s))=m(\overline{X}(s),X(s'))=0$. 
Si $s' \leq_{B} s$ et $l(s)<l(s')-1$ alors $\dim \mathrm{Hom}(X(s'),X(s))=m(\overline{X}(s),X(s'))=1$ 
et tout morphisme non nul de cet espace est surjectif.
\end{lem}
Consid\'erons maintenant l'ensemble $\mathfrak{I}_{n}$ des involutions de $\mathfrak{S}_{n}$.
Rappelons que $\mathfrak{I}_{n}$, muni de l'ordre induit par celui de $\mathfrak{S}_{n}$,
est lui aussi gradu\'e par une fonction longueur que nous avons not\'e $l_{\mathfrak{I}}$.

Nous avons vu dans la section \textbf{(4.2)} que pour tout $s\in\mathfrak{I}_{n}$ 
le modules standard  $X(s)$ associ\'e, est $\theta_{N}$-invariant. 
D\'efinissons donc pour tout $i\in[1,l_{\mathfrak{I}}(s_{\text{max}})]$ l'espace,
\begin{align*}
X_{i,\theta}:=\oplus_{s\in\mathfrak{I}_{n},l_{\mathfrak{I}}(s)=i}X(s). 
\end{align*}
L'objectif de cette section est de construire pour tout $i\in[1,l_{\mathfrak{I}}(s_{\text{max}})]$ un morphisme, 
\begin{align}\label{eq:morphcasserie}
\phi_{j}&:X_{j-1,\theta}\longrightarrow X_{j,\theta},\\
\phi_{j}&=\oplus_{s\in\mathfrak{I}_{n},~l_{\mathfrak{I}}(s)=j-1}\oplus_{s'\in\mathfrak{I}_{n},~l_{\mathfrak{I}}(s)=j}\phi_{s,s'},\quad
\phi_{s,s'}^{\theta}:X(s')\longrightarrow X(s)
\end{align} 
de mani\`ere d'obtenir un complexe diff\'erentiel,
\begin{align}\label{eq:suite8}
0\rightarrow \textbf{Speh}(\delta,n)\rightarrow I(\delta,n)\rightarrow\cdots\rightarrow X_{i-1,\theta}\xrightarrow{\phi_{i}} X_{i,\theta}
\rightarrow \cdots \rightarrow X_{l_{\mathfrak{I}}(s_{\text{max}}),\theta}\rightarrow 0.
\end{align}
Pour d\'efinir le morphisme $\phi_{i},~i\in[1,l_{\mathfrak{I}}(s_{\text{max}})]$ 
et ainsi pouvoir construire \textbf{(\ref{eq:suite8})} on proc\`ede de la mani\`ere suivant:\\

Pour tout $s\in\mathfrak{S}_{n}$ soit $\Omega_{s}$ la fonctionnelle de Whittaker du module standard $X(s)$. 
Prennons une paire $s,~s'$ tel que $l_{\mathfrak{S}}(s')=l_{\mathfrak{S}}(s)-1$ et $s'<_{B}s$. 
D'apr\`es le lemme \textbf{(\ref{lem:defmorsection2.5})} tout morphisme non nul $\phi\in\text{Hom}(X(s'),X(s))$ est surjective. Par cons\'equent  
la composition, 
\begin{align*}
\Omega_{s}\circ\phi, 
\end{align*}
d\'efinit une fonctionelle de Whittaker non nulle de $X(s')$ et de l'unicit\'e, multiplication par un scalair pr\`es, 
de la fonctionelle de Whittaker 
(Voir th\'eor\`eme \textbf{(2.2)} de \textbf{\cite{Hashizume}} ou th\'eor\`eme \textbf{(1.1)} de \textbf{\cite{ShahidiDuke}}), 
il existe une constante $C\in\mathbb{C}^{\ast}$ tel que, 
\begin{align*}
C\cdot\Omega_{s}\circ\phi=\Omega_{s'}.  
\end{align*}
Une pr\`emiere etape en la construction de \textbf{(\ref{eq:morphcasserie})} est donn\'ee par la d\'efinition suivant.
\begin{deftn}\label{deftn:etape1}
Soient $s,s'\in\mathfrak{S}_{n}$ v\'erifiant $s'\leq_{B}s$ et 
$l_{\mathfrak{S}}(s')=l_{\mathfrak{S}}(s)-1$. Alors on d\'efinit,
\begin{align}\label{eq:morW}
\phi_{s,s'}^{{wh}}:X(s')\rightarrow X(s),
\end{align}
comme l'unique \'el\'ement de $\text{Hom}(X(s'),X(s))$ tel que,
\begin{align}\label{eq:egmorW}
\Omega_{s}\circ\phi_{s,s'}^{{wh}}=\Omega_{s'}. 
\end{align}
Si au contraire $s'\nleq_{B} s$ on d\'efinit $\phi_{s,s'}^{{wh}}$ comme le morphisme nul.
On fait noter que d'apr\`es le lemme \textbf{(\ref{lem:defmorsection2.5})} pour toute paire $s$, $s'$
avec $l_{\mathfrak{S}}(s')=l_{\mathfrak{S}}(s)-1$ le morphisme $\phi_{s,s'}^{\mathrm{wh}}$ est 
proportionnel au morphisme de Johnson entre $X(s')$ et $X(s)$.
\end{deftn}
Avec le choix de morphisme donn\'ee par la d\'efinition pr\'ec\'edent on obtient le r\'esultat suivant.
\begin{lem}\label{lem:icc}
Soient $s,s'\in\mathfrak{S}_{N}$ v\'erifiant $s\leq_{B}s'$ et $l_{\mathfrak{S}}(s)-l_{\mathfrak{S}}(s')=m$. 
Alors pour toute paire de ch\^aines saturées,
\begin{align*}
s'&<_{B}s_{1}<_{B}\cdots<_{B}s_{m}<s\\ 
s'&<_{B}t_{1}<_{B}\cdots<_{B}t_{m}<s, 
\end{align*}
conduisant de $s$ \`a $s'$ on a l'\'egalit\'e,
\begin{align}\label{eq:egalitecheminelemme}
\phi_{s,s_{n}}^{wh}\circ\cdots\circ\phi_{s_{1},s'}^{{wh}}=\phi_{s,t_{n}}^{wh}\circ\cdots\circ\phi_{t_{1},s'}^{wh}. 
\end{align}
\end{lem}
\begin{proof}[\textbf{Preuve}]
Ce lemme en une conséquence directe du théorème \textbf{(\ref{theo:ELSHEsim})}.
En effet soit $C_{s',s}=\{C_{1},\cdots,C_{n}\}$ l'ensemble de ch\^aines saturées de longueur maximale de l'intervalle $[s',s]$.
L'intervalle $[s',s]$ etant EL-shellable et donc d'après la proposition \textbf{(\ref{theo:EL-SHE})} Shellable, l'ensemble $C_{s',s}$ est un ensemble totalement ordononné 
tel que pour tout ch\^aine $C_{i}\in C_{s',s}$ avec $C_{i}=s'<s_{1}^{i}<\cdots<s_{m-1}^{i}<s$ il existe un entier $k,~1\leq k\leq m-1$ tel que,  
$$s_{l}^{i}=s_{l}^{i+1},\text{ si }l\neq k~\text{ et }~s_{l}^{i}\neq s_{l}^{i+1},  \text{ si }l=k.$$ 
La paire $s_{k}^{i},s_k^{i+1}$ etant d'apr\`es le lemme \textbf{(\ref{lem:d2lt})}, 
les deux uniques éléments de $\mathfrak{S}_{N}$ tels que  $s_{k-1}^{i}<s_{k}^{i},s_k^{i+1}<s_{k+1}^{i}$, on peut d'après la résolution de Johnson écrire,
\begin{align*}
\phi_{s_{k+1}^{i},s_{k}^{i}}^{J}\circ\phi_{s_{k}^{i},s_{k-1}^{i}}^{J}=-\phi_{s_{k+1}^{i},s_{k}^{i+1}}^{J}\circ\phi_{s_{k}^{i+1},s_{k-1}^{i}}^{J}.
\end{align*}
Par conséquent,
\begin{align*}
\phi_{s_{k+1}^{i},s_{k}^{i}}^{{wh}}\circ\phi_{s_{k}^{i},s_{k-1}^{i}}^{{wh}}=-\phi_{s_{k+1}^{i},s_{k}^{i+1}}^{{wh}}\circ\phi_{s_{k}^{i+1},s_{k-1}^{i}}^{{wh}}.
\end{align*}
D'où on obtient,
$$\phi_{s,s_{m-1}^{i}}^{wh}\circ\cdots\circ\phi_{s_{1}^{i},s'}^{{wh}}=\phi_{s,s_{m-1}^{i+1}}^{wh}\circ\cdots\circ\phi_{s_{1}^{i+1},s'}^{wh}$$
Pour finir soit $C_{t}$ la cha\^ine de $C_{s',s}$ définie par \textbf{(\ref{eq:egalitecheminelemme})} et $C_{u}$ celle définie par 
\textbf{(\ref{eq:egalitecheminelemme})}. Sans perte de généralité on peut supposer $C_{t}<C_{u}$. 
Maintenant qu'on a montré la veracité de l'equation \textbf{(\ref{eq:egalitecheminelemme})} pour toute paire de cha\^ines $C_{i}$ et $C_{i+1}$ de $C_{s',s}$ avec $t<i<u$
un simple argument de transitivité nous donne l'égalité \textbf{(\ref{eq:egalitecheminelemme})} pour les cha\^ines $C_{t}$ et $C_{u}$. 

\end{proof}
Revenons \`a pr\'esent \`a l'ensemble des involutions $\mathfrak{I}_{n}$ et donnons
l'annalogue de la d\'efinition \textbf{(\ref{deftn:etape1})} pour $\mathfrak{I}_{n}$.
\begin{deftn}\label{eq:defetape2}
Soient $s,s'\in\mathfrak{I}_{n}$ v\'erifiant $s'<_{B}s$ et $l_{\mathfrak{I}}(s)-l_{\mathfrak{I}}(s')=1$. 
D'apr\`es le lemme \textbf{(\ref{lem:long})} on sait que $l_{\mathfrak{S}}(s)-l_{\mathfrak{S}}(s')\leq 3$. 
Si $l_{\mathfrak{S}}(s)-l_{\mathfrak{S}}(s')>1$, choisissons une cha\^ine d'éléments de $\mathfrak{S}_{n}$ aillent de $s'$ vers $s$, 
\begin{align*}
s'<_{B}s_{1}<_{B}s,&\text{ si }l_{\mathfrak{S}}(s)-l_{\mathfrak{S}}(s')=2,\\
s'<_{B}s_{1}<_{B}s_{2}<s,&\text{ si }l_{\mathfrak{S}}(s)-l_{\mathfrak{S}}(s')=3.
\end{align*}
Alors on d\'efinit, 
\begin{equation}\label{eq:morphisme}
\phi_{s,s'}^{{wh}}:X(s')\longrightarrow X(s),
\end{equation}
comme $\phi_{s,s'}^{{wh}}$ si $l_{\mathfrak{S}}(s)-l_{\mathfrak{S}}(s')=1$ ou comme la composition,
\begin{align*}
\phi_{s,s'}^{{wh}}:=\phi_{s,s_{1}}^{{wh}}\circ\phi_{s_{1},s'}^{{wh}},&\text{ si }l_{\mathfrak{S}}(s)-l_{\mathfrak{S}}(s')=2,\\
\phi_{s,s'}^{{wh}}:=\phi_{s,s_{2}}^{{wh}}\circ\phi_{s_{2},s_{1}}^{{wh}}
\circ\phi_{s_{1},s'}^{{wh}},&\text{ si }l_{\mathfrak{S}}(s)-l_{\mathfrak{S}}(s')=3.
\end{align*}
On fait noter que d'apr\`es le lemme \textbf{(\ref{lem:icc})} 
la d\'efinition de $\phi_{s,s'}^{{wh}}$ ne depend pas du chemin choisi.
Finalement si $s'\nleq s$, 
on d\'efinit $\phi_{s,s'}^{{wh}}:X(s')\longrightarrow X(s)$ comme le morphisme nul.
\end{deftn}
Du lemme \textbf{(\ref{lem:d2lt})} pour toute paire $s,s'\in\mathfrak{I}_{N}$
v\'eerifiant $s'<_{B}s$ et $l_{\mathfrak{S}}(s)-l_{\mathfrak{S}}(s')=2$  
existent exactement deux \'el\'ements $s_{1},s_{2}$ dans $\mathfrak{I}_{N}$, 
tels que $s'<_{B}s_{j}<_{B}s,~j=1,2$, et du lemme \textbf{(\ref{lem:icc})} on peut, 
car tout morphisme $\phi_{l,l'},~l,l'\in\mathfrak{I}_{n}$ est d\'efini
comme composition de morphismes associ\'es aux pairs d'\'el\'ements de 
$\mathfrak{S}_{n}$ avec diff\'eerence de longueur \'egale \`a 1, d\'eduire l'\'egalit\'e,
\begin{align*}
\phi_{s,s_{1}}^{\text{wh}}\circ\phi_{s_{1},s'}^{\text{wh}}=\phi_{s,s_{2}}^{\text{wh}}\circ\phi_{s_{2},s'}^{\text{wh}}. 
\end{align*}
Par cons\'equent si pour toute paire 
$l,l'\in \mathfrak{I}_{N}$ v\'erifiant $l'<_{B}l$ et $l_{\mathfrak{I}}(l)-l_{\mathfrak{I}}(l')=1$
nous prennons une constante $\lambda_{l,l'}\in\mathbb{C}$ et notons,
\begin{align*}
\phi_{l,l'}^{\lambda}=\lambda_{l,l'}\phi^{\text{wh}}_{l,l'}, 
\end{align*}
alors pour tout couple $s,s'\in \mathfrak{I}_{N}$ avec $s'<_{B}s$ 
et $l_{\mathfrak{I}}(s)-l_{\mathfrak{I}}(s')=2$ on aura,
\begin{align}\label{eq:ccd}
\phi_{s,s_{1}}^{\lambda}\circ \phi_{s_{1},s'}^{\lambda}+\phi_{s,s_{2}}^{\lambda}\circ \phi_{s_{1},s}^{\lambda}
&=\lambda_{s,s_{1}}\phi_{s,s_{1}}^{\text{wh}}\circ \lambda_{s_{1},s'}\phi_{s_{1},s'}^{\text{wh}}+
\lambda_{s,s_{2}}\phi_{s,s_{2}}^{\text{wh}}\circ \lambda_{s_{2},s'}\phi_{s_{1},s'}^{\text{wh}}\\
&=(\lambda_{s,s_{1}}\cdot\lambda_{s_{1},s'}+\lambda_{s,s_{2}}\cdot\lambda_{s_{2},s'})\phi_{s,s_{1}}^{\text{wh}}\circ \phi_{s_{1},s'}^{\text{wh}}\\
&=(\lambda_{s,s_{1}}\cdot\lambda_{s_{1},s'}+\lambda_{s,s_{2}}\cdot\lambda_{s_{2},s'})\phi_{s,s_{2}}^{\text{wh}}\circ \phi_{s_{2},s'}^{\text{wh}},
\end{align}
avec $s_{1}$ et $s_{2}$ la paire d'\'el\'ements donn\'ee par le lemme \textbf{(\ref{lem:d2lt})}.

La construction du complexe diff\'erentiel \textbf{(\ref{eq:suite8})} 
se ram\`ene donc \`a trouver pour toute paire  
$l,l'\in \mathfrak{I}_{N}$ avec $l_{\mathfrak{I}}(s)-l_{\mathfrak{I}}(s')=1$, des constantes $\lambda_{l,l'}\in\mathbb{C}$, avec $\lambda_{l,l'}\neq 0$ si $l<_{B}l'$,
de mani\`ere \`a ce que dans l'equation \textbf{(\ref{eq:ccd})}, la somme  
$\lambda_{s,s_{1}}\cdot\lambda_{s_{1},s'}+\lambda_{s,s_{2}}\cdot\lambda_{s_{2},s'}$
soit \'egal \`a z\'ero pour toute couple $s,s'\in \mathfrak{I}_{N},~l_{\mathfrak{I}}(s)-l_{\mathfrak{I}}(s')=2$.

Tout ce qui reste de cette section sera d\'edi\'e \`a choisir correctement les constants $\lambda_{l,l'}$.\\

Associ\'e au module  $\textbf{Speh}(\delta,n)$ on 
a un param\`etre de Arthur, 
$$\psi:W_{\mathbb{R}}\times\textbf{SL}(2,\mathbb{C})\longrightarrow \mathbf{GL}(2n,\mathbb{C}),$$ 
de caract\`ere infinit\'esimal r\'egulier et entier. 
Le module $\textbf{Speh}(\delta,n)$ étant $\theta_{N}$-invariant, 
il existe un groupe classique $\mathbf{G}$ d\'efini par une des options suivants,
\begin{align*}
\mathbf{G}=\left\{\begin{array}{cc}
                   \textbf{SO}(n+1,n)   &\text{si }\frac{p+n-1}{2}\in2\mathbb{Z}+1,\\
                   \textbf{SO}(n,n)     &\text{si }n\in2\mathbb{N}\text{ et }\frac{p+n-1}{2}\in2\mathbb{Z},\\
                   \textbf{SO}(n+1,n-1) &\text{si }n\in2\mathbb{N}+1\text{ et }\frac{p+n-1}{2}\in2\mathbb{Z},
                  \end{array}\right.
\end{align*}
et un $L$-param\`etre $\varphi_{\mathbf{G}}$ de $\mathbf{G}$ de fa\c con \`a ce que,
\begin{align*}
\psi:W_{\mathbb{R}}\xrightarrow{\psi_{\mathbf{G}}}~^{L}\mathbf{G}\xrightarrow{\iota}\mathbf{GL}(2n,\mathbb{C}), 
\end{align*}
o\`u $\iota$ d\'enote l'inclusion standard. 

Le param\`etre $\psi_{\mathbf{G}}$ v\'erifie les deux propri\'et\'es suivants,
\begin{enumerate}[i.]
\item  La restriction de $\psi_{\mathbf{G}}$ \`a $\mathbb{R}_{>0}\hookrightarrow\mathbb{C}^{\ast}\hookrightarrow W_{\mathbb{R}}$ est triviale.
\item  L'image d'un \'el\'ement unipotent r\'egulier de $SL(2,\mathbb{C})$ est un 
\'el\'ement unipotent r\'egulier dans le centralisateur $\text{Cent}(\psi_{\mathbf{G}}(\mathbb{C}^{\ast}),\widehat{\mathbf{G}})$. 
\end{enumerate}
Par cons\'equent, attach\'e \`a $\psi_{\mathbf{G}}$ il existe un paquet 
de repr\'esentations $\Pi_{\psi_{\mathbf{G}}}^{\text{AJ}}$ de $\mathbf{G}(\mathbb{R})$ 
appell\'e paquet d'Adams Johnson de $\varphi_{\mathbf{G}}$ (Voir \textbf{\cite{Adams-Johnson}}). 
Toute repr\'esentation dans $\Pi_{\psi_{\mathbf{G}}}^{\text{AJ}}$ 
est unitaire et avec de la $(\mathfrak{g},K)$-cohomologie non nulle. 
D'apr\`es la classification de Vogan-Zuckerman des repr\'esentations
unitaire irr\'eductibles et avec de la $(\mathfrak{g},K)$-cohomologie non nulle, pour toute repr\'esentation
$\pi$ de $\Pi_{\psi_{G}}^{\text{AJ}}$ existe un groupe unitaire $\mathbf{L}=\mathbf{U}(p,q),~p+q=n,$ 
et un caract\`ere $\lambda_{\pi}$ de $\mathbf{L}(\mathbb{R})$ de mani\`ere \`a ce que on puisse \'ecrire,
\begin{align*}
\pi=\mathcal{R}_{\mathfrak{q}}^{S}(\lambda_{\pi}),  
\end{align*}
o\`u $\mathcal{R}_{\mathfrak{q}}^{i}$ est le foncteur d'induction cohomologique de
Vogan-Zuckerman, $\mathfrak{q}$ une sous-alg\`ebre parabolique $\tau$-invariante
de levi $\mathfrak{l}$, l'algr\`ebre de lie de $\mathbf{L}$, et 
$S=(1/2)\dim(\mathfrak{k}/\mathfrak{l}\cap\mathfrak{k})$.

Dans tout ce qui suit, fixons $\pi\in\Pi_{\psi_{\mathbf{G}}}^{\text{AJ}}$ de fa\c con \`a 
ce que que le groupe de Levi associ\'e, $\mathbf{L}=\mathbf{U}(p,q)$, soit quasi-deploy\'e. 

Rappelons que l'ensemble de repr\'esentations irr\'eductible de $\mathbf{L}$ dont le caract\`ere infinit\'esimal est \'egal \`a celle
de $\lambda_{\pi}$ est param\`etre par l'ensemble $\mathfrak{I}_{n}^{p,q,\pm}$ d\'efini dans la section \textbf{(2.5)}. Dans cette param\`etrisation
deux param\`etres $(\eta,f_{\eta})$, $(\eta',f_{\eta'})\in\mathfrak{I}_{n}^{p,q,\pm}$ 
correspondent a d\'es repr\'esentations dans le m\^eme $L$-paquet si et seulement si les involutions
sous-jacentes $\eta$ et $\eta'$ dans $\mathfrak{I}_{n}$ sont \'egales. Ceci nous donne
une parm\'etrisation de l'ensemble des $L$-param\`etres de $\mathbf{L}$ par $\mathfrak{I}_{n}$.

Dans \textbf{\cite{Johnson}} J. Johnsons donne pour toute repr\'esentations $F$ irr\'eductible
et de dimension finie, une r\'esolution par de modules standards. Pour $\lambda_{\pi}$
la r\'esolution de Johnson est param\`etre par $\mathfrak{I}_{n}^{p,q,\pm}$ et s'\'ecrit,
\begin{align}\label{eq:suiteJohnsonunitaire}
0\rightarrow\lambda_{\pi}\rightarrow\cdots\rightarrow Y_{i}\rightarrow Y_{i+1}
\cdots\rightarrow Y_{l_{\mathfrak{I}}(s_{\text{max}})}\rightarrow 0. 
\end{align}
o\`u pour tout $i\in[0,l_{\mathfrak{I}}(s_{\text{max}})]$,
\begin{align*}
Y_{i}&:=\oplus_{\eta\in\mathfrak{I}_{n},l_{\mathfrak{I}}(\eta)=i}\Theta_{\eta}\\ 
\Theta_{\eta}&:=\oplus_{(\zeta,f_{\zeta})\in\mathfrak{I}_{n}^{p,q,\pm},\zeta=\eta}X(\overline{\zeta}).
\end{align*}
Soit $\varphi_{\eta,\mathbf{L}}$ le $L$-param\`etre de $\mathbf{L}$ 
associ\'e \`a l'\'el\'ement $\eta\in\mathfrak{I}_{n}$, c'est le $L$-param\`etre attach\'e au
pseudo-paquet des repr\'esentations intervenant dans la somme, 
$\Theta_{\eta}=\oplus_{(\zeta,f_{\zeta})\in\mathfrak{I}_{n}^{p,q,\pm},\zeta=\eta}X(\overline{\zeta})$.
Notons par $\iota_{\mathbf{G},\mathbf{L}}$ l'inclusion$~^{L}\mathbf{L}\hookrightarrow~^{L}\mathbf{G}$
et consid\'erons le $L$-param\`etre de $\mathbf{GL}(2n)$ donn\'e 
par la composition,
\begin{align*}
\varphi_{\eta}:W_{\mathbb{R}}\xrightarrow{\varphi_{\eta,\mathbf{L}}}
~^{L}\mathbf{L}\xrightarrow{\iota_{\mathbf{G},\mathbf{L}}}~^{L}\mathbf{G}\xrightarrow{\iota}\mathbf{GL}(2n,\mathbb{C}).
\end{align*}
Soit $X(\eta)$ le module standard de $\mathbf{GL}(2n,\mathbb{R})$ attach\'e \`a $\varphi_{\eta}$,
alors l'application,
\begin{align*}
\Theta_{\eta}\mapsto X(\eta), 
\end{align*}
d\'efinit une bijection entre l'ensemble des modules standard stables, $\Theta_{\eta}$, 
intervenant dans la r\'esolution de Johnson de $\lambda_{\pi}$ est l'ensemble
des modules standard $\theta_{N}$-invariants de $\mathbf{GL}(2n,\mathbb{R})$ 
intervenant dans la suite \textbf{(\ref{eq:suite8})}. 

En cons\'equence, pour trouver les scalaires n\'ecessaires \`a la contruction du complexe 
\textbf{(\ref{eq:suite8})} on s'appuie sur la suite \textbf{(\ref{eq:suiteJohnsonunitaire})}. 
Pour toute couple $\overline{\eta},~\overline{\eta}'\in\mathfrak{I}_{n}^{p,q,\pm}$ vérifiant 
$l_{\mathfrak{I}}(\eta')=l_{\mathfrak{I}}(\eta)-1$ notons par $\varphi_{\overline{\eta},\overline{\eta}'}^{J}:X(\overline{\eta}')\rightarrow X(\overline{\eta})$ 
le morphisme de Johnson nécessaire à la construction de la résolution de Johnson de $\lambda_{\pi}$. 
D'après la proposition \textbf{(\ref{prop:multipliciteinvo})} si $\overline{\eta}'<_{B}\overline{\eta}$, le morphisme
$\varphi_{\overline{\eta},\overline{\eta}'}^{J}$ est non nul. 

Les suivants résultats sur le morphisme de Johnson nous indiquerons comme choisir correctement les scalaires $\lambda_{l,l'},~l,l'\in\mathfrak{I}_{n}$, 
$l_{\mathfrak{I}}(l)-l_{\mathfrak{I}}(l')=1$.

\begin{lem}\label{lem:iccinvol}
Soient $\overline{\eta},\overline{\eta}'\in\mathfrak{I}_{n}^{p,q,\pm}$ v\'erifiant $\overline{\eta}\leq_{B}\overline{\eta}'$ et $l_{\mathfrak{I}}({\eta})-l_{\mathfrak{I}}({\eta}')=m$. 
Alors pour toute paire de cha\^ines saturées,
\begin{align*}
\overline{\eta}'&<_{B}\overline{\eta}_{1}<_{B}\cdots<_{B}\overline{\eta}_{m-1}<\overline{\eta},\\ 
\overline{\eta}'&<_{B}\overline{\zeta}_{1}<_{B}\cdots<_{B}\overline{\zeta}_{m-1}<\overline{\eta}, 
\end{align*}
il existe un entier $C\in\{-1,1\}$ de façon à ce que,
\begin{align}\label{eq:egalitecheminelemmein}
\varphi_{\overline{\eta},\overline{\eta}_{m-1}}^{J}\circ\cdots\circ\varphi_{\overline{\eta}_{1},\overline{\eta}'}^{J}=C(\varphi_{\overline{\eta},\overline{\zeta}_{m-1}}^{J}
\circ\cdots\circ\varphi_{\overline{\zeta}_{1},\overline{\eta}'})^{J}, 
\end{align}
\end{lem}
\begin{proof}[\textbf{Preuve}.]
De la description fait au chapitre \textbf{(5)} du diagrammes de Hasse de $\mathfrak{I}_{n}^{p,q,\pm}$, on déduit que si 
le couple $\overline{\mu},\overline{\mu}'\in\mathfrak{I}_{n}^{p,q,\pm}$ vérifie, $\overline{\mu}'\leq_{B}\overline{\mu}$ et $l_{\mathfrak{I}}({\mu})-l_{\mathfrak{I}}({\mu}')=1$, 
alors $\overline{\mu}'$ est l'unique élément d'involution associée $\mu'$ tel que $\overline{\mu}'\leq_{B}\overline{\mu}$. 
Par conséquent, si $\overline{\eta}'<_{B}\overline{\eta}_{1}<_{B}\cdots<_{B}\overline{\eta}_{m}<\overline{\eta}$ et $\overline{\eta}'<_{B}\overline{\zeta}_{1}<_{B}\cdots<_{B}\overline{\zeta}_{m}<\overline{\eta}$ sont deux cha\^ines saturées, alors pour tout $i,1\leq i\leq m-1$ avec $\overline{\eta}_{i} \neq \overline{\zeta}_{i}$ les involutions associées $\eta_{i}$ et $\zeta_{i}$ sont différents.  
Le lemme est donc une conséquence directe du fait que la restriction de l'ordre de Bruhat à l'ensemble des involutions est El-Shellable (théorème \textbf{(\ref{theo:ELSHEinv})}). 
De plus, l'argument utilisé dans la preuve du lemme \textbf{(\ref{lem:icc})} est aussi valable dans le cas présent avec $\mathfrak{I}_{N}$ au lieu de $\mathfrak{S}_{N}$,
et tel comme pour le lemme \textbf{(\ref{lem:icc})} on ramène la preuve au cas de cha\^ines saturées de longueur 2. L'égalité \textbf{(\ref{eq:egalitecheminelemmein})} avec $C=-1$ étant evident dans ce cas, d'après le lemme \textbf{(\ref{lem:d2lt})} et le fait
que la suite de Johnson est un complexe différentiel, le lemme est prouvé. 
\end{proof}
\begin{prop}\label{prop:whittakerunicite}
Soit $\mathbf{G}$ un groupe classique quasi-deployé; linéaire, unitaire orthogonal ou symplectique, et
$X$ un module standard générique de $\mathbf{G}(\mathbb{R})$ de caractère infinitesimal entier et régulier.
 Alors $X$ admet un unique sous-quotient irréductible générique, lequel appara\^it avec multiplicité un. 
\end{prop}
\begin{proof}[\textbf{Preuve}.]
Soit $X$ comme dans l'énoncé. Notons par $d\lambda$ le caractère infinitésimal 
de $X$. Il existe un sous-groupe $L$ de $\mathbf{G}(\mathbb{R})$, un groupe de Cartan $H$ de $L$ 
et un caractère $\delta\in\widehat{H}$ tel que $X=\mathcal{R}_{\mathfrak{q}}^{S}(\text{Ind}_{HN}^{L}(\delta))$.
 La proposition est une conséquence des deux résultats suivants: 

\textbf{i.} Le  théorème \textbf{(3.6.7)} de \textbf{\cite{Shahidi}} que dit:
\begin{theo}
Soit $\mathbf{G}$ un groupe quasi-deployé et $H=TA$ un groupe de Cartan de $\mathbf{G}(\mathbb{R})$. 
Supposons $\delta\otimes\nu\in\widehat{H}$ tel que la série principale $\text{Ind}_{HN}^{G}(\delta\otimes\nu)$ de $\mathbf{G}(\mathbb{R})$ est générique
Alors $\text{Ind}_{HN}^{G}(\delta)$ admet un unique sous-quotient irréductible générique, lequel appara\^it avec multiplicité un. Si en plus  $\text{Re}(\nu)\leq 0$,  
alors l'unique sous-quotient
générique de $\text{Ind}_{HN}^{G}(\delta)$ est un quotient. 
\end{theo}
\textbf{ii.} La restriction du foncteur $\mathcal{R}_{\mathfrak{q}}^{S}$ à l'espace de représentations de longueur fini et de caractère infinitésimal
$d\lambda-\rho(\mathfrak{l})$ de ${L}(\mathbb{R})$ est exact et envoie de représentation irréductibles de ${L}(\mathbb{R})$ dans de représentations irréductibles de $\mathbf{G}(\mathbb{R})$. De plus $\mathcal{R}_{\mathfrak{q}}^{S}$ préserve la propriété d'\^etre large (c'est-à-dire générique).
\end{proof}
\begin{lem}\label{lem:image}
Soient $\overline{\eta}=(\eta,f_{\eta}),\overline{\eta}'=(\eta',f_{\eta'})\in\mathfrak{I}_{n}^{p,q,\pm}$ 
tels que $\eta'<_{B}\eta$, $l_{\mathfrak{I}}(\eta')=l_{\mathfrak{I}}(\eta)-1$ 
et tels que les modules standards $X(\overline{\eta})$ et $X(\overline{\eta}')$ 
admettent de fonctionnelles de Whittaker non nulles. 
Alors $\overline{\eta}'<_{B}\overline{\eta}$ et $\text{Im}(\varphi_{\overline{\eta},\overline{\eta}'})$ possède un quotient générique.
\end{lem}
\begin{proof}[\textbf{Preuve}]
Soit $\overline{\mu}\in\mathfrak{I}_{n}^{p,q,\pm}$, la correspondance $\mathfrak{I}_{n}^{p,q,\pm}\rightarrow\mathscr{P}_{F}$ 
nous permet d'associer à $\overline{\mu}$ un ensemble $\Delta(\overline{\mu})^{+}$ de racines positives de $\mathfrak{u}(p,q)$. D'après 
le point \textbf{f.} du théorème \textbf{(6.2)} de \textbf{\cite{Vogan}} (voir aussi la page \textbf{587} de \textbf{\cite{Casselman-Shahidi}}) la représentation 
$X(\overline{\mu})$ est large (c'est-à-dire génerique) si $\Delta(\overline{\mu})^{+}$ satisfait les trois conditions suivants,
\begin{enumerate}[i.]
\item Si $\alpha\in\Delta(\overline{\mu})^{+}$ est une racine simple non réel, alors $\tau\alpha\notin\delta(\overline{\mu})^{+}$.
\item Si $\alpha\in\Delta(\overline{\mu})^{+}$ est une racine simple réel, alors $\alpha$ ne satisfait pas
la condition de parité énoncée dans la définition \textbf{(8.3.11)} de \textbf{\cite{VoganLivre}}.
\item Si $\alpha\in\Delta(\overline{\mu})^{+}$ est une racine simple imaginaire, alors $\alpha$ 
n'est pas compacte.
\end{enumerate}

Considérons un couple $\overline{\eta},\overline{\eta}'$ d'élémets de $\in\mathfrak{I}_{n}^{p,q,\pm}$ 
v\'erifiant $\overline{\eta}'<_{B}\overline{\eta}$ et $l_{\mathfrak{I}}(\eta')=l_{\mathfrak{I}}(\eta)-1$.
Rappelons que $\overline{\eta}$ est obtenu à partir de $\overline{\eta}'$, 
\begin{enumerate}[a.]
\item soit en remplaçant deux symboles consécutifs de la forme $aa,~a\in\mathbb{N}$ par $(+-)$ ou $(-+)$.
\item soit en permutant deux symboles qui ne sont pas tout deux des signes.
\item soit à partir de la condition d'échange. Notons que si  $\overline{\eta}$ est obtenu de $\overline{\eta}'$ à partir de cette condition d'échange,
alors dans l'écriture symbolique de $\overline{\eta}'$ on trouve une configuration $a+-a$ (ou $a-+a$) qui dans $\overline{\eta}$ est remplacée par une configuration 
$aabb$. 
\end{enumerate} 

Il n'est donc pas difficile de voir que si $\Delta(\overline{\eta}')^{+}$ ne satisfait pas une des trois conditions, i. ii. iii.,
alors il en est de m\^eme pour $\Delta(\overline{\eta})^{+}$. 
Mais tout pseudo-paquet de $\mathbf{U}(p,q,\mathbb{R})$ contient un unique module standard générique, 
par conséquent, si $X(\overline{\eta}')$ admet une fonctionnelle de Whittaker non nulle on peut dire de m\^eme 
de $X(\overline{\eta})$, si  $\overline{\eta}$ est obtenu de $\overline{\eta}'$ à partir du point \textbf{b.} ou \textbf{c.}, 
et de une de deux possibilités pour 
$X(\overline{\eta})$, si  $\overline{\eta}$ est obtenu de $\overline{\eta}'$ à partir du point \textbf{a}. Ce qui montre la première affirmation du lemme.\\

\`A présent soit $\overline{\eta}_{\text{max}}$ l'élément de $\mathfrak{I}_{n}^{p,q,\pm}$ d'involution
associée égal à $s_{\text{max}}$ tel que $X(\overline{\eta}_{\text{max}})$ admet une fonctionelle de Whittaker non nulle.
Notons aussi par $\overline{\eta}_{0}$ l'élément de $\mathfrak{I}_{n}^{p,q,\pm}$ d'écriture
symbolique égal à $(1,2,\cdots,n/2,n/2,\cdots,2,1)$ si $n$ est pair et égal à $(1,2,\cdots,(n-1)/2,+,(n-1)/2,\cdots,2,1)$ si $n$ impair.
Dans tout ce qui suit, soit $m=l_{\mathfrak{I}}(s_{\text{max}})$. Supposons qu'il existe une cha\^ine saturée,
\begin{align}\label{eq:chainewhittaker}
\overline{\eta}_{0}<_{B}\overline{\eta}_{1}<_{B}\cdots<_{B}\overline{\eta}_{m-1}<_{B}\overline{\eta}_{\text{max}},
\end{align}
tel que pour tout $i,~0\leq i\leq m$, $X(\overline{\eta}_{i})$ admet une fonctionelle de Whittaker non nulle 
et tel que $\varphi_{\overline{\eta}_{\text{max}},\overline{\eta}_{m-1}}^{J}\circ\cdots\circ\varphi_{\overline{\eta}_{1},\overline{\eta}_{0}}^{J}$ 
soit différent de zéro. 

Soit $\overline{\eta}$ et $\overline{\eta}'$ une paire d'élémets de $\mathfrak{I}_{n}^{p,q,\pm}$ 
v\'erifiant $\overline{\eta}'<_{B}\overline{\eta}$, $l_{\mathfrak{I}}(\eta')=l_{\mathfrak{I}}(\eta)-1$ 
et tels que les modules standards $X(\overline{\eta})$ et $X(\overline{\eta}')$ 
admettent de fonctionnelles de Whittaker non nulles. 
\'Etendons $\overline{\eta}'<_{B}\overline{\eta}$ en une cha\^ine saturée, 
$$\overline{\eta}_{0}<_{B}\overline{\zeta}_{1}<_{B}\cdots<_{B}\overline{\eta}_{j}=\overline{\eta}'<_{B}\overline{\eta}=\overline{\zeta}_{j+1}\cdots<_{B}\overline{\eta}_{\text{max}},$$ 
de $\mathfrak{I}_{N}^{p,q\pm}$ tel que pour tout $i,~0\leq i\leq m$, $X(\overline{\eta}_{i})$ admet une fonctionelle de Whittaker non nulle. 
D'après le lemme \textbf{(\ref{lem:iccinvol})} il existe une entiere $C\in\{-1,1\}$ tel que,
$$\varphi_{\overline{\eta}_{\text{max}},\overline{\eta}_{m-1}}^{J}\circ\cdots\circ\varphi_{\overline{\eta}_{1},\overline{\eta}_{0}}^{J}=C(\varphi_{\overline{\eta}_{\text{max}},\overline{\zeta}_{m-1}}^{J}\circ\cdots\circ\varphi_{\overline{\zeta}_{1},\overline{\eta}_{0}}^{J}),$$ 
comme en plus le cot\^e gauche de l'égalité précédent a été suppossé non-nul on peut dire de m\^eme du cot\^e droite. 
La restriction de $\varphi_{\overline{\eta}_{\text{max}},\overline{\zeta}_{m-1}}^{J}\circ\cdots\circ\varphi_{\overline{\eta}_{j+2},\overline{\eta}}^{J}$ 
à $\text{Im}(\varphi_{\overline{\eta},\overline{\eta}'}^{J})$ est donc différent de zéro, d'où 
$\varphi_{\overline{\eta}_{\text{max}},\overline{\zeta}_{m-1}}^{J}\circ\cdots\circ\varphi_{\overline{\eta}_{j+2},\overline{\eta}}^{J}$
est une surjection, car $X(\overline{\eta}_{\text{max}})$ est irréductible.  
Par conséquent  $\text{Im}(\varphi_{\overline{\eta},\overline{\eta}'}^{J})/
\ker(\varphi_{\overline{\eta}_{\text{max}},\overline{\zeta}_{m-1}}^{J}\circ\cdots\circ\varphi_{\overline{\eta}_{j+2}\overline{\eta}}^{J})
\cong X(\overline{\eta}_{\text{max}})$ et $\text{Im}(\varphi_{\overline{\eta},\overline{\eta}'}^{J})$ possède un quotient générique.    

Le lemme sera donc prouve si on est capable de définir une cha\^ine saturée de longueur 
maximale de $\mathfrak{I}_{n}^{p,q,\pm}$ satisfaisant les m\^emes propriétés que \textbf{(\ref{eq:chainewhittaker})}.  

Supposons d'abord que $n$ est pair. Alors comme cha\^ine saturée on prend celle définie par,
\begin{align}
\overline{\eta}_{0}<_{B}\overline{\eta}_{1}<_{B}\cdots<_{B}\overline{\eta}_{\text{max}}
\end{align}
où $\overline{\eta}_{i}$ est, pour tout $i,~1\leq i\leq m-n/2$, l'involution signé sans points fixes d\'efinie 
(si on utilise la notation symbolique pour décrire les éléments de $\mathfrak{I}_{n}^{p,q,\pm}$) en permutant
dans $\overline{\eta}_{i-1}$ le couple $(\overline{m(i)},m(i))$ avec 
$\overline{m(i)}$ l'entier inmediatement à gauche de $m(i)$ dans $\overline{\eta}_{i-1}$ 
et $m(i)$ défini recursivement comme suit: $m(0)=0$, 
$m(1)$ la première copie de 1 qui appara\^it dans $\overline{\eta}_{0}$ 
en le lissant de droite à gauche. Et pour $i\geq 1$, 
si $m(i-1)\neq m(k)$ pour tout $k<i-1$ ou si à 
droite de $m(i-1)$ dans $\overline{\eta}_{i-1}$ il y a un entier suivi
du m\^eme entier alors $m(i)$ est défini comme la première copie, en lissant $\overline{\eta}_{i-1}$ de droite à gauche, 
du plus petit entier $l$ dans $\overline{\eta}_{i-1}$ 
avec à gauche un entier différent de $l$, si au contraire, $m(i)=m(k)$, pour un entier $k<i-1$, alors $m(i)$ est la première copie de 
$m(i-1)+1$ qui appara\^it dans $\overline{\eta}_{i-1}$ 
en le lissant de droite à gauche. Par exemple pour $p=3$, $q=3$, la cha\^ine des éléments de $\mathfrak{I}_{6}^{3,3,\pm}$ ainsi définie
est donnée par, $(123321)<_{B}(123312)<_{B}(123132)<_{B}(123123)<_{B}(121323)<_{B}(121233)<_{B}(112233)$.


Si maintenant  nous considérons $j=m-n/2+i,~1\leq i\leq n/2$, alors $\overline{\eta}_{j}$ est donné par l'unique élément 
dans $\mathfrak{I}_{n}^{p,q,\pm}$, d'involution associée $\eta_{j}$ définie par, 
\begin{align*}
\eta_{j}(k)=k,\text{ si }1\leq k\leq 2j~\text{ et }~\eta_{j}(2j+2k+1)=2j+2k+2
\text{ si }0\leq k\leq (n-2j)/2,
\end{align*}
tel que $X(\overline{\eta}_{j})$ admet une fonctionelle de Whittaker non nulle.\\

Supposons à présent que est $n$ est impaire. Alors comme cha\^ine saturée on prend celle définie par,
\begin{align}
\overline{\eta}_{0}<_{B}\overline{\eta}_{1}<_{B}\cdots<_{B}\overline{\eta}_{\text{max}}
\end{align}
où $\overline{\eta}_{i}$ est pour tout $i,~1\leq i\leq m-(n-1)/2$, donné par l'unique élément 
dans $\mathfrak{I}_{n}ˆ{p,q,\pm}$, d'involution associée $\eta_{j}$ définie par,
\begin{align*}
\eta_{j}(k)&=\left\{\begin{array}{cl}
n+1-k& \text{ si }k\geq j+2,\\
n+1-i-1& \text{ si }k\leq j,\\
n& \text{ si }k=j+2,
\end{array}\right.\\
\overline{\eta}_{m-(n-1)/2+1}&=(1,2,3,\cdots,3,2,1,+). 
\end{align*}
Finalement si $j=m-(n-1)/2+1+i,~1\leq i\leq (n-1)/2+1$, alors $\overline{\eta}_{j}$ est défini par, $\overline{\eta}_{j}=(\overline{\eta}_{i}^{p-1,q},+)$ 
avec $\overline{\eta}_{i}^{p-1,q}$ le $i-$eme élément de la cha\^ine saturée définie précédement pour $\mathfrak{I}_{n-1}^{p-1,q,\pm}$.   

Maintenant que la cha\^ine saturée \textbf{(\ref{eq:chainewhittaker})} a été definie il 
ne nous reste que à montrer que la composition, 
$\varphi_{\overline{\eta}_{\text{max}},\overline{\eta}_{m-1}}^{J}\circ\cdots\circ\varphi_{\overline{\eta}_{1},\overline{\eta}_{0}}^{J}$, 
est différent de zéro. 

Il suffit pour cela de montrer pour tout $i,~1\leq i\leq m,$ l'affirmation suivant: 
$$\textit{L'image de }\phi_{\overline{\eta}_{i},\overline{\eta}_{i-1}}^{J}\textit{ contient un sous-quotient générique.}$$

Comme le foncteur d'induction parabolique est exacte et que pour tout $i$, l'élément $\overline{\eta}_{i}$ s'obtient à partir de 
$\overline{\eta}_{i-1}$ en permutant deux symboles qui ne sont pas tout deux des signes, la fonction de Johnson 
entre $X(\overline{\eta}_{i-1})$ et $X(\overline{\eta}_{i})$ s'obtient après induction d'un morphisme $\phi$, donné par une des trois 
posibilités suivants,   
\begin{align*}
\textbf{i.)}&~X(\{abba\})\xrightarrow{\phi} X(\{abab\}),\\
\textbf{ii.)}&~X(\{abab\})\xrightarrow{\phi} X(\{aabb\}),\\
\textbf{iii.)}&~X(\{a+a\})\xrightarrow{\phi} X(\{aa+\}).
\end{align*}
La preuve de l'affirmation précédent se ramène en conséquence 
à montrer que l'image du morphisme $\phi$ contient un sous-quotient générique.

Notons que dans le deux premiers cas le morphisme $\phi$ est en fait surjetif ce qui montre l'affirmation.

Pour le point trois,  rappelons que le diagramme de Hasse pour $\mathfrak{I}_{3}^{2,1,\pm}$ 
est donné dans  le diagramme \textbf{(\ref{eq:Hasse})}.

De plus, du corollaire \textbf{(7.2)} de \textbf{\cite{VoganIII}} et les formule de récurrences enoncées dans la proposition \textbf{(6.14)} 
du m\^eme article, pour $\overline{X(\{a+a\})}$, on obtient,
\begin{align*}
\overline{X(\{a+a\})}={X(\{a+a\})}-{X(\{aa+\})}-{X(\{+aa\})}+{X(\{-++\})}+{X(\{+-+\})}+{X(\{++-\})},
\end{align*}
comme aussi,
\begin{align*}
{X(\{aa+\})}=\overline{X(\{aa+\})}-{X(\{+-+\})}-{X(\{-++\})},\\
{X(\{+aa\})}=\overline{X(\{+aa\})}-{X(\{+-+\})}-{X(\{++-\})},
\end{align*}
on conclut,
$$X(\{a+a\})=\overline{X(\{a+a\})}+\overline{X(\{+aa\})}+\overline{X(\{aa+\})}+{X(\{+ - +\})}.$$
Par conséquent $X(\{+-+\})$ est l'unique quotient irréductible de $X(\{a+a\})$,
ce qui fait de $X(\{+-+\})$ un quotient de $\text{im}(\phi)\cong X(\{a+a\})/\text{ker}(\phi)$,
montrant ainsi l'affirmation, car $X(\{+-+\})$ est générique.  
\end{proof}
\begin{cor}\label{cor:constante} 
Pour toute paire $l,l'\in\mathfrak{I}_{n}$
avec $l'<_{B}l$ et $l_{\mathfrak{I}}(l')=l_{\mathfrak{I}}(l)-1$ il existe une constante $\lambda_{l,l'}\in\mathbb{C}^{\ast}$,
de fa\c con \`a ce que pour tout quadruple, $\{\overline{\eta},\overline{\eta}',\overline{\eta}_{1},\overline{\eta}_{2}\}$, vérifiant
$\eta'<\eta_{i}<\eta$, $l_{\mathfrak{I}}(\eta')=l_{\mathfrak{I}}(\eta_{i})-1=l_{\mathfrak{I}}(\eta)-2,~i=1,2$ et tel que pour tout $\overline{\mu}$ dans le quadruple
$X(\overline{\mu})$
admet une fonctionelle de Whittaker non nulle, on ait, 
\begin{align*}
\beta_{\eta,\eta_{1}}\varphi_{\overline{\eta},\overline{\eta}_{1}}^{J}\circ\beta_{\eta_{1},\eta}\varphi_{\overline{\eta}_{1},\overline{\eta}'}^{J}=
\beta_{\eta,\eta_{2}}\varphi_{\overline{\eta},\overline{\eta}_{2}}^{J}\circ\beta_{\eta_{2},\eta}\varphi_{\overline{\eta}_{2},\overline{\eta}'}^{J}.
\end{align*}
\end{cor}
\begin{proof}[\textbf{Preuve.}]
Soient $\overline{\eta},~\overline{\eta}'$ et $\overline{\eta}_{1},~\overline{\eta}_{2}$ 
comme dans l'\'enonc\'e du lemme. Notons par $\Omega_{\overline{\eta}},\Omega_{\overline{\eta}'},\Omega_{\overline{\eta_{1}}}$ et $\Omega_{\overline{\eta}_{2}}$
les fonctionnelles de Whittaker respectifs.

D'après le lemme \textbf{(\ref{lem:image})}, pour toute paire $\zeta,\zeta'\in\mathfrak{I}_{N}$ v\'erifiant $\zeta'<_{B}\zeta$ et 
$l_{\mathfrak{I}}(\zeta')=l_{\mathfrak{I}}(\zeta)-1$, l'image $\text{Im}(\varphi_{\overline{\zeta},\overline{\zeta}'}^{J})$ contient un sous-quotient générique
si on peut dire de m\^eme de $X(\overline{\zeta})$ et $X(\overline{\zeta}')$. 
En conséquence, de la proposition \textbf{(\ref{prop:whittakerunicite})}, 
\begin{align*}
\Omega(\overline{\eta}_{i})\circ\varphi_{\overline{\eta}_{i},\overline{\eta}'}^{J}\quad\text{ et }\quad \Omega(\overline{\eta})
\circ\varphi_{\overline{\eta},\overline{\eta}_{i}}^{J},
\end{align*}
d\'efinisent des fonctionelles de Whittaker non nulles 
pour $X(\overline{\eta}_{i})~i=1,2,$ et $X(\overline{\eta})$, respectivement. 
Par l'unicit\'e de la fonctionelle de Whittaker 
il existent de constants $\beta_{\eta_{i},\eta'}$ et $\beta_{\eta,\eta_{i}},~i=1,2,$ 
de mani\`ere \`a ce que,
\begin{align*}
\Omega(\overline{\eta}')=\Omega(\overline{\eta}_{i})\circ\beta_{\eta_{i},\eta'}\varphi_{\overline{\eta}_{i},\overline{\eta}'}^{J}
\quad\text{ et }\quad\Omega(\overline{\eta}_{i})=\Omega(\overline{\eta})\circ\beta_{\eta,\eta_{i}}\varphi_{\overline{\eta},\overline{\eta}_{i}}^{J}.
\end{align*}
Notons,
\begin{align*}
\varphi_{\overline{\eta}_{i},\overline{\eta}'}^{{wh}}
=\beta_{\eta_{i},\eta'}\varphi_{\overline{\eta}_{i},\overline{\eta}'}^{J}\quad\text{ et }
\quad\varphi_{\overline{\eta},\overline{\eta}_{i}}^{{wh}}
=\beta_{\eta,\eta_{i}}\varphi_{\overline{\eta},\overline{\eta}_{i}}^{J}. 
\end{align*}
Comme,
\begin{align*}
\varphi_{\overline{\eta},\overline{\eta}_{1}}^{J}\circ \varphi_{\overline{\eta}_{1},\overline{\eta}'}^{J}+
\varphi_{\overline{\eta},\overline{\eta}_{2}}^{J}\circ \varphi_{\overline{\eta}_{2},\overline{\eta}'}^{J}=0, 
\end{align*}
on obtient,
\begin{align*}
\varphi_{\overline{\eta},\overline{\eta}_{1}}^{{wh}}\circ \varphi_{\overline{\eta}_{1},\overline{\eta}'}^{{wh}}=
-\frac{\beta_{\eta,\eta_{2}}\cdot\beta_{\eta_{2},\eta'}}{\beta_{\eta,\eta_{1}}\cdot\beta_{\eta_{1},\eta'}}
\cdot\left(\varphi_{\overline{\eta},\overline{\eta}_{2}}^{{wh}}\circ \varphi_{\overline{\eta}_{2},\overline{\eta}'}^{{wh}}\right) 
\end{align*}
mais,
\begin{align*}
\Omega_{\overline{\eta}'}&=\Omega_{\overline{\eta}}\circ\varphi_{\overline{\eta},\overline{\eta}_{1}}^{{wh}}\circ \varphi_{\overline{\eta}_{1},\overline{\eta}'}^{{wh}}\\  
&=-\frac{\lambda_{\eta,\eta_{2}}\cdot\lambda_{\eta_{2},\eta'}}{\beta_{\eta,\eta_{1}}\cdot\beta_{\eta_{1},\eta'}}
\cdot\left(\Omega_{\overline{\eta}'}\circ\varphi_{\overline{\eta},\overline{\eta}_{2}}^{{wh}}\circ \varphi_{\overline{\eta}_{2},\overline{\eta}'}^{{wh}}\right)\\
&=-\frac{\beta_{\eta,\eta_{2}}\cdot\beta_{\eta_{2},\eta'}}{\beta_{\eta,\eta_{1}}\cdot\beta_{\eta_{1},\eta'}}\cdot\Omega_{\overline{\eta}'},
\end{align*}
d'o\`u $-{\beta_{\eta,\eta_{2}}\cdot\beta_{\eta_{2},\eta'}}\cdot\beta_{\eta,\eta_{1}}^{-1}\cdot\beta_{\eta_{1},\eta'}^{-1}=1$ 
et,
\begin{align*}
\varphi_{\overline{\eta},\overline{\eta}_{1}}^{{wh}}\circ \varphi_{\overline{\eta}_{1},\overline{\eta}'}^{{wh}}=
\varphi_{\overline{\eta},\overline{\eta}_{2}}^{{wh}}\circ \varphi_{\overline{\eta}_{2},\overline{\eta}'}^{{wh}}. 
\end{align*}
Par cons\'equent,
\begin{align*}
0&=\varphi_{\overline{\eta},\overline{\eta}_{1}}^{J}\circ \varphi_{\overline{\eta}_{1},\overline{\eta}'}^{J}+
\varphi_{\overline{\eta},\overline{\eta}_{2}}^{J}\circ \varphi_{\overline{\eta}_{2},\overline{\eta}'}^{J}\\
 &=\beta_{\eta,\eta_{1}}^{-1}\varphi_{\overline{\eta},\overline{\eta}_{1}}^{{wh}}\circ\beta_{\eta_{1},\eta'}^{-1}\varphi_{\overline{\eta}_{1},\overline{\eta}'}^{{wh}}
 +\beta_{\eta,\eta_{2}}^{-1}\varphi_{\overline{\eta},\overline{\eta}_{2}}^{{wh}}\circ\beta_{\eta_{2},\eta'}^{-1}\varphi_{\overline{\eta}_{2},\overline{\eta}'}^{{wh}}\\
 &=(\beta_{\eta,\eta_{1}}^{-1}\cdot\beta_{\eta_{1},\eta'}^{-1}+\beta_{\eta,\eta_{2}}^{-1}\cdot\beta_{\eta_{2},\eta'}^{-1})
 (\varphi_{\overline{\eta},\overline{\eta}_{1}}^{{wh}}\circ \varphi_{\overline{\eta}_{1},\overline{\eta}'}^{{wh}}+
\varphi_{\overline{\eta},\overline{\eta}_{2}}^{{wh}}\circ \varphi_{\overline{\eta}_{2},\overline{\eta}'}^{{wh}})
\end{align*}
ce qui nous permet de conclure l'\'egalit\'e,
\begin{align}
\beta_{\eta,\eta_{1}}^{-1}\cdot\beta_{\eta_{1},\eta'}^{-1}+\beta_{\eta,\eta_{2}}^{-1}\cdot\beta_{\eta_{2},\eta'}^{-1}=0. 
\end{align}
\end{proof}
Revenons \`a la suite \textbf{(\ref{eq:suite8})}. Pour finir avec la construction de celle ci,
donnons la d\'efinition suivant.
\begin{deftn}
Soient $s,s'\in\mathfrak{I}_{n}$ v\'erifiant $s'<_{B}s$ 
et $l_{\mathfrak{I}}(s')=l_{\mathfrak{I}}(s)-1$. 
Soit $\lambda_{s,s'}=\beta_{s,s'}^{-1}$ la constante
associ\'ee au pair $s,s'$ par le corollaire \textbf{(\ref{cor:constante})},
alors on d\'efinit $\phi_{s,s'}:X(s')\longrightarrow X(s)$ par,
\begin{align}\label{eq:morphismefinal}
\phi_{s,s'}=\lambda_{s,s'}\phi_{s,s'}^{\text{wh}}
\end{align}
et,
\begin{align}\label{eq:flechefinal}
\phi_{j}&:X_{i-1,\theta}\longrightarrow X_{i,\theta}\nonumber\\
\phi_{j}&=\oplus_{s\in\mathfrak{I}_{N},l_{\mathfrak{I}}(s)=i}\oplus_{s'\in\mathfrak{I}_{N},l_{\mathfrak{I}}(s')=i-1}
\phi_{s,s'}
\end{align}
\end{deftn}
Avec les fl\`eches de \textbf{(\ref{eq:suite8})} d\'efinies 
par \textbf{(\ref{eq:flechefinal})} on obtient d'apr\`es la remarque que suit \`a la d\'efinition \textbf{(\ref{eq:defetape2})} et le 
corollaire \textbf{(\ref{cor:constante})} le th\'eor\`eme suivant.
\begin{theo}
La suite, 
\begin{align}
0\rightarrow \mathbf{Speh}(\delta,n)\rightarrow I(\delta,n)\rightarrow\cdots\rightarrow X_{i-1,\theta}\rightarrow X_{i,\theta}
\rightarrow \cdots \rightarrow X_{l_{\mathfrak{I}}(s_{\text{max}}),\theta}\rightarrow 0,
\end{align}
d\'efinie \`a partir des morphismes, 
\begin{align*}
\phi_{i}:X_{i+1,\theta}\rightarrow X_{i,\theta} 
\end{align*}
o\`u,
\begin{align*}
\phi_{i}:=\oplus_{\gamma\in\Gamma_{i}}\oplus_{\gamma'\in \Gamma_{i+1}}\phi_{\gamma,\gamma'} 
\end{align*}
est un complexe diff\'erentiel.
\end{theo}

\subsection{$\theta$-exactitude}
Prenons, une nouvelle fois, $p,n\in\mathbb{N}^{\times}$ avec $p>n-1$ et 
consid\'erons $\textbf{Speh}(\delta,n)$, o\`u $\delta:=\delta(p/2,-p/2)$.
Dans la section pr\'ec\'edent on a montr\'e que la suite,
\begin{align}\label{eq:suitesectionexactitude}
0\rightarrow \textbf{Speh}(\delta,n)\rightarrow I(\delta,n)\rightarrow\cdots\rightarrow X_{i-1,\theta}\xrightarrow{\phi_{i}} X_{i,\theta}
\rightarrow \cdots \rightarrow X_{l_{\mathfrak{I}}(s_{\text{max}}),\theta}\rightarrow 0,
\end{align}
avec les fleches, $\phi_{i}$, d\'efinies par \textbf{(\ref{eq:flechefinal})} est un complexe diff\'erentiel. 
Cette suite est certainement non exacte, mais on aimerait bien pouvoir montrer 
qu'elle est $\theta$-exacte, c'est-\`a-dire tel que pour 
tout $i\in[1,l_{\mathfrak{I}}(s_{\text{max}})]$ on a l'\'egalit\'e,
\begin{align*}
\text{Tr}_{\theta}(\ker(\phi_{i})/\text{im}(\phi_{i-1}))=0, 
\end{align*}
o\`u la trace tordue est prise par rapport \`a l'action de $\theta_{N}$ normalis\'ee \`a l'Arthur, comme expliqu\'e dans \textbf{\cite{Arthur}} section \textbf{(2.2)}. La $\theta$-exactitude de \textbf{(\ref{eq:suitesectionexactitude})} a comme cons\'equence l'\'egalit\'e,
\begin{align*}
\text{Tr}_{\theta}(\textbf{Speh}(\delta,n))=\sum_{i=0}^{l_{\mathfrak{I}}(s_{\text{max}})}(-1)^{i}\text{Tr}_{\theta}(X_{i}). 
\end{align*}
Cette \'egalit\'e est obtenue diff\'eremment dans \textbf{\cite{AMR}}, voir th\'eor\`emes \textbf{(9.5)} et \textbf{(9.7)} pour une preuve.
La m\'ethode utilis\'ee dans \textbf{\cite{AMR}} se base sur le calcul de l'action de l'automorphisme exterieur $\theta_{N}$ de $\mathbf{GL}(N)$
sur le complexe de Johnson  qui r\'esout $\textbf{Speh}(\delta,n)$.

Dans cet article nous pr\'esentons la preuve de la $\theta$-exactitude de \textbf{(\ref{eq:suitesectionexactitude})} pour 
$\textbf{Speh}(\delta,n)$ avec $n$ plus petit ou \'egal \`a 4.

Par example si $n=2$, alors $\textbf{Speh}(\delta,2)=\delta(p/2,-p/2)$ et il n'y a rien \`a prouver, 
car $\delta(p/2,-p/2)$ est irr\'eductible et $\theta_{4}$-invariante. 

Si maintenant nous nous situons dans $GL(4,\mathbb{R})$. 
Comme $X(s_{\text{max}})$ est l'unique sous-quotient de $X(s_{0})$ 
la suite \textbf{(\ref{eq:suitesectionexactitude})} pour $\textbf{Speh}(\delta,2)$ devient,
\begin{align}\label{eq:suitegl4}
0\rightarrow \textbf{Speh}(\delta,2)\rightarrow X(s_{0})\rightarrow  X(s_{\text{max}})\rightarrow 0. 
\end{align}
Or, \textbf{(\ref{eq:suitegl4})} est aussi la resolution de Johnson pour $\textbf{Speh}(\delta,2)$.
Par cons\'equent \textbf{(\ref{eq:suitegl4})} est exacte et donc $\theta$-exacte.\\

Les deux derni\`eres sections seront consacr\'ees \`a montrer la $\theta$-exactitude
de \textbf{(\ref{eq:suitesectionexactitude})} pour $\textbf{Speh}(\delta,n)$ avec $n=3$ et $n=4$.

Pour simplifier un peut tout ce qui suit nous allons adopter la notation suivant.
Pour tout $i\in[1,l_{\mathfrak{I}}(s_{\mathfrak{I}})]$ soit, 
\begin{align*}
m_{i}:=\frac{p+n-1}{2}-i. 
\end{align*}
Alors pour tout $s\in\mathfrak{S}_{N}$ nous allons noter,
\begin{align*}
X(s)=\{(m_{1},-m_{s_{1}}),\cdots,(m_{n},-m_{s_{n}})\}, 
\end{align*}
au lieu de,
\begin{align*}
X(s)=\delta(m_{1},-m_{s_{1}})\times\cdots\times\delta(m_{n},-m_{s_{n}}). 
\end{align*}
De plus, \`a chaque fois que le contexte le permette, on ne fera pas de diff\'erence 
entre l'\'el\'ement $s$ de $\mathfrak{S}_{N}$ et le module
standard $X(s)$ de $\mathbf{GL}(2n,\mathbb{R})$ associ\'e. 

Finalement pour tout $s\in\mathfrak{S}_{N}$ d\'efinissons,
\begin{align}\label{eq:ultimaequationde4}
K_{s}:=\bigcap_{s<_{B}s',~l_{\mathfrak{S}}(s)=l_{\mathfrak{I}}(s)-1} \ker(\phi_{s',s}^{J}). 
\end{align}

\subsection{$\theta$-exactitude dans $\mathbf{GL}(6,\mathbb{R})$}
Pour $\textbf{Speh}(\delta,3)$ le complexe diff\'erentiel \textbf{(\ref{eq:suitesectionexactitude})} est 
donn\'e par le diagramme ci-dessous. Dans celui-ci on a abusivement not\'e
une fl\`eche pleine au lieu du morphisme défini par l'equation \textbf{(\ref{eq:morphisme})}
 alors qu'une flèche pointilléé est l'opposée du morphisme précédent. 
 \begin{align}\label{eq:suitecasGl6}
 \xymatrix {\relax &0 \ar[d]& \\\relax &\overline{X(s^2)}\ar[d]&\\
  \relax &X(s^2):={\left[\begin{smallmatrix}(m_{1}-m_{3})\\(m_{2}-m_{2})\\(m_{3}-m_{1})\end{smallmatrix}\right]} \ar[dl] \ar[dr]&
  \\ X(s^1_{1}):={\left[\begin{smallmatrix}(m_{1}-m_{1})\\(m_{2}-m_{3})\\(m_{3}-m_{2})\end{smallmatrix}\right]}
  \ar[dr]& \empty&X(s_{2}^1):={\left[\begin{smallmatrix}(m_{1}-m_{2})\\(m_{2}-m_{1})\\(m_{3}-m_{3})\end{smallmatrix}\right]} \ar@{.>}[dl]
  \\&
  X(s_{\text{max}}):={\left[\begin{smallmatrix}(m_{1}-m_{1})\\(m_{2}-m_{2})\\(m_{3}-m_{3})\end{smallmatrix}\right]}\ar[d]&\\
  &0&}
 \end{align} 
Passons maintenant \`a la preuve de la $\theta$-exactitude de \textbf{(\ref{eq:suitecasGl6})}.
Situons nous dans le cran $\mathbf{(2\rightarrow 1)}$,
\begin{align*}
X(s^{2})\rightarrow \oplus_{i=1}^{2}X(s_{i}^{1}), 
\end{align*}
Notons,
\begin{align*}
X(\mu):=\{(m_{1},-m_{2}),(m_{2},-m_{3}),(m_{3},-m_{1})\},\\
X(\theta\mu):=\{(m_{1},-m_{3}),(m_{2},-m_{1}),(m_{3},-m_{2})\}.
\end{align*}
Soit $w\in\ker(\phi_{s_{1}^{1},s^{2}})\cap\ker(\phi_{s_{2}^{1},s^{2}})$. 
On a respectivement, $\phi_{\mu,s^{2}}^{J}(w)\in K_{\mu}$ et 
$\phi_{\theta\mu,s^{2}}^{J}(w)\in K_{\theta\mu}$. L'exactitude de la suite Johnson
nous permet de choisir $w_{1}\in\ker(\phi_{\theta\mu,s^{2}}^{J})$, 
respectivement $w_{2}\in \ker(\phi_{\mu,s^{2}}^{J})$
tel que,
\begin{align*}
\phi_{\mu,s^{2}}^{J}(w_{1})=\phi_{\mu,s^{2}}^{J}(w)~\text{et}~
\phi_{\theta\mu,s^{2}}^{J}(w_{2})=\phi_{\theta\mu,s^{2}}^{J}(w).
\end{align*}
Soit $y=w-(w_{1}+w_{2})$ alors $y\in\ker(\phi_{\mu,s^{2}}^{J})\cap\ker(\phi_{\theta\mu,s^{2}}^{J})$
et on peut \'ecrire, 
\begin{align*}
w=(w_{1}+y)+w_{2}.
\end{align*}
Ce qui d\'emontre le r\'esultat suivant;
\begin{lem}\label{lem:lemmeGL6}
\begin{align*}
\ker(\phi_{1})&=\ker(\phi_{\mu,s^{2}}^{J})+\ker(\phi_{\theta\mu,s^{2}}^{J})
\end{align*}
\end{lem}
Du lemme \textbf{(\ref{lem:lemmeGL6})} ils nous est donc possible conclure,
\begin{align*}
\text{Tr}_{\theta}(\ker(\phi_{1}))&=\text{Tr}_{\theta}(\ker(\phi_{\mu,s^{2}}^{J})+
\ker(\phi_{\theta\mu,s^{2}}^{J}))\\
&=\text{Tr}_{\theta}(\ker(\phi_{\mu,s^{2}}^{J})\cap \ker(\phi_{\theta\mu,s^{2}}^{J}))\\
&=\text{Tr}_{\theta}(\overline{X(s^{2})}).
\end{align*}
et clairement,
\begin{align*}
\text{Tr}_{\theta}(\ker(\phi_{1})/\overline{X(s^{2})})=0. 
\end{align*}
Cran $\mathbf{1\rightarrow 0}$,
\begin{align*}
\oplus_{j=1}^{2}X(s_{j}^{1})\rightarrow X(s_{\text{max}}), 
\end{align*}
où on rappelle que,
\begin{align*}
X(s_{\text{max}}):=\{(m_{1},-m_{1}),(m_{2},-m_{2}),(m_{3},-m_{3})\}. 
\end{align*}
Pour chaque $1\leq j\leq 2$ on a, car la multiplicité de $X(s_{\text{max}})$ en tant que sous-quotient de $X(s_{j}^{1})$ est égal à 1, l'\'egalit\'e,
\begin{align*}
\ker(\phi_{s_{\text{max}},s_{j}^{1}})=\overline{X(s_{j}^{1})}.
\end{align*}
Comme en plus chaque $\overline{X(s_{j}^{1})},~1\leq j\leq 2,$ 
est contenu dans, $\text{im}(\phi_{1})$, on conclut,
\begin{align*}
\ker(\phi_{0})/\text{im}(\phi_{1})=0.
\end{align*}

\subsection{$\theta$-Resolution dans $\textbf{GL}(8,\mathbb{R})$}\label{sectiongl8}
Situons nous \`a pr\'esent dans $\mathbf{GL}(8,\mathbb{R})$ et consid\'erons $\textbf{Speh}(\delta,4)$. 
Pour $\textbf{Speh}(\delta,4)$ la suite \textbf{(\ref{eq:suitesectionexactitude})} est donn\'ee par le diagramme ci-dessous.
Dans celui-ci on a, comme pour  $\mathbf{GL}(6,\mathbb{R})$, abusivement not\'e  
une fl\`eche pleine au lieu du morphisme défini par l'equation \textbf{(\ref{eq:morphismefinal})},
alors qu'une flèche pointillée est l'opposée du morphisme précédent.  
\begin{align}\label{eq:suiteexactcasGl8} 
  \xymatrix {\relax &0 \ar[d]& \\\relax &\overline{X(s^4)}\ar[d]&\\
  \relax &X(s^4):={\left[\begin{smallmatrix}(m_{1}-m_{4})\\(m_{2}-m_{3})\\(m_{3}-m_{2})\\(m_{4}-m_{1})\end{smallmatrix}\right]} \ar[dl] \ar [dr]&
  \\ X(s^3_{1}):={\left[\begin{smallmatrix}(m_{1}-m_{4})\\(m_{2}-m_{2})\\(m_{3}-m_{3})\\(m_{4}-m_{1})\end{smallmatrix}\right]}
  \ar[d]\ar[dr]\ar[rrd]& \empty&X(s_{2}^3):={\left[\begin{smallmatrix}(m_{1}-m_{2})\\(m_{2}-m_{4})\\(m_{3}-m_{1})\\(m_{4}-m_{2})\end{smallmatrix}\right]} \ar@{.>}[lld]\ar@{.>}[dl]\ar@{.>}[d]
  \\
  X(s_{1}^2):={\left[\begin{smallmatrix}(m_{1}-m_{3})\\(m_{2}-m_{2})\\(m_{3}-m_{1})\\(m_{4}-m_{4})\end{smallmatrix}\right]}\ar[d]\ar[dr]&X(s_{2}^2):={\left[\begin{smallmatrix}(m_{1}-m_{2})\\(m_{2}-m_{1})\\(m_{3}-m_{4})\\(m_{4}-m_{3})\end{smallmatrix}\right]} \ar@{.>}[dl]\ar@{.>}[rd]&X(s_{3}^2):={\left[\begin{smallmatrix}(m_{1}-m_{1})\\(m_{2}-m_{4})\\(m_{3}-m_{3})\\(m_{4}-m_{2})\end{smallmatrix}\right]} \ar@{.>}[dl] \ar[d]
  \\
  X(s_{1}^1):={\left[\begin{smallmatrix}(m_{1}-m_{2})\\(m_{2}-m_{1})\\(m_{3}-m_{3})\\(m_{4}-m_{4})\end{smallmatrix}\right]}\ar[dr]&X(s_{2}^1):={\left[\begin{smallmatrix}(m_{1}-m_{1})\\(m_{2}-m_{3})\\(m_{3}-m_{2})\\(m_{4}-m_{4})\end{smallmatrix}\right]}\ar@{.>}[d]&X(s_{3}^1):={\left[\begin{smallmatrix}(m_{1}-m_{1})\\(m_{2}-m_{2})\\(m_{3}-m_{4})\\(m_{4}-m_{3})\end{smallmatrix}\right]}\ar@{.>}[dl]
  \\&
  X(s_{\text{max}}):={\left[\begin{smallmatrix}(m_{1}-m_{1})\\(m_{2}-m_{2})\\(m_{3}-m_{3})\\(m_{4}-m_{4})\end{smallmatrix}\right]}\ar[d]&\\
  &0&
  }.\end{align}
Il nous faut \`a pr\'esent d\'emontrer que la suite \textbf{(\ref{eq:suiteexactcasGl8})} est $\theta$-exacte.

Situons nous dans le cran $\mathbf{(4\rightarrow 3)}$,
\begin{align*}
X(s^{4})\rightarrow \oplus_{i=1}^{2}X(s_{i}^{3}).
\end{align*}
Notons,
\begin{align*}
X(\mu):=\{(m_{1},-m_{3}),(m_{2},-m_{4}),(m_{3},-m_{2}),(m_{4},-m_{1})\},\\
X(\theta\mu):=\{(m_{1},-m_{4}),(m_{2},-m_{3}),(m_{3},-m_{1}),(m_{4},-m_{2})\}.
\end{align*}
On aura besoin du r\'esultat suivant;
\begin{lem}\label{lem:lemcran43}
\begin{align*}
\ker(\phi_{3})=
(\ker(\phi_{\mu,s^{4}}^{J})\cap \ker(\phi_{s_{1}^{3},s^{4}}))+
(\ker(\phi_{\theta\mu,s^{4}}^{J})\cap \ker(\phi_{s_{1}^{3},s^{4}})) 
\end{align*}
\end{lem}
\begin{proof}[\textbf{Preuve.}]
Soit $w\in\ker(\phi_{s_{1}^{3},s^{4}})\cap\ker(\phi_{s_{2}^{3},s^{4}})$. 
Comme la suite de Johnson d\'efinit un complexe diff\'erentiel, $\phi_{\mu,s^{4}}^{J}(w)\in K_{\mu}$ 
(Voir equation \textbf{(\ref{eq:ultimaequationde4})} pour la définition de $K_{\mu}$) et 
$\phi_{\theta\mu,s^{4}}^{J}(w)\in K_{\theta\mu}$. L'exactitude
de la suite de Johnson nous permet de choisir $w_{1}\in\ker(\phi_{\theta\mu,s^{4}}^{J})\cap\ker(\phi_{s_{1}^{3},s^{4}})$, 
respectivement $w_{2}\in \ker(\phi_{\mu,s^{3}}^{J})\cap\ker(\phi_{s_{1}^{3},s^{3}})$
tel que,
\begin{align*}
\phi_{\mu,s^{4}}^{J}(w_{1})=\phi_{\mu,s^{4}}^{J}(w)~\text{et}~
\phi_{\theta\mu,s^{4}}^{J}(w_{2})=\phi_{\theta\mu,s^{4}}^{J}(w).
\end{align*}
Soit $y=w-(w_{1}+w_{2})$ alors $y\in\ker(\phi_{\mu,s^{4}}^{J})\cap\ker(\phi_{\theta\mu,s^{4}}^{J})\cap\ker(\phi_{s_{1}^{3},s^{4}})$
et on peut \'ecrire, 
\begin{align*}
w=(w_{1}+y)+w_{2}.
\end{align*}
Ce qui nous donne l'inclusion,
$\text{ker}(\phi_{3})\subset(\ker(\phi_{\mu,s^{4}}^{J})\cap \ker(\phi_{s_{1}^{3},s^{4}}))+
(\ker(\phi_{\theta\mu,s^{4}}^{J})\cap \ker(\phi_{s_{1}^{3},s^{4}}))$.
L'inclusion oppos\'e est evident du fait que la suite de Johnson est un complexe diff\'erentiel.
\end{proof}
Du lemme \textbf{(\ref{lem:lemcran43})} ils nous est donc possible conclure,
\begin{align*}
\text{Tr}_{\theta}(\ker(\phi_{3,4}))&=\text{Tr}_{\theta}(\ker(\phi_{\mu,s^{4}}^{J})\cap \ker(\phi_{s_{1}^{3},s^{4}})+
(\ker(\phi_{\theta\mu,s^{4}}^{J})\cap \ker(\phi_{s_{1}^{3},s^{4}}))\\
&=\text{Tr}_{\theta}(\ker(\phi_{\mu,s^{4}}^{J})\cap \ker(\phi_{\theta\mu,s^{4}}^{J})\cap \ker(\phi_{s_{1}^{3},s^{4}}))\\
&=\text{Tr}_{\theta}(\overline{X(s^{4})}).
\end{align*}
et clairement,
\begin{align*}
\text{Tr}_{\theta}(\ker(\phi_{3,4})/\text{im}(\phi_{4}))=\text{Tr}_{\theta}(\ker(\phi_{3,4})/\overline{X(s^{4})})=0. 
\end{align*}
Cran $\mathbf{3\rightarrow 2}$,
\begin{align*}
\oplus_{k=1}^{2}X(s_{k}^{3})\rightarrow \oplus_{j=1}^{3}X(s_{j}^{2}), 
\end{align*}
o\`u on rappelle que,
\begin{align*}
X(s_{1}^{2})=\{(m_{1},-m_{3}),(m_{2},-m_{2}),(m_{3},-m_{1}),(m_{4},-m_{4})\},\\
X(s_{2}^{2})=\{(m_{1},-m_{2}),(m_{2},-m_{1}),(m_{3},-m_{4}),(m_{4},-m_{3})\},\\
X(s_{3}^{2})=\{(m_{1},-m_{1}),(m_{2},-m_{4}),(m_{3},-m_{3}),(m_{4},-m_{2})\}.
\end{align*}
Le fait que tout \'el\'ement de $X(s_{1}^{3})$ soit possible d'\^etre \'ecrit comme la somme d'un \'el\'ement
dans $\text{im}(\phi_{3})$ plus un \'el\'ement dans $X(s_{2}^{3})$ nous permet de
ram\`ener le probl\`eme de l'\'etude de 
$\ker(\phi_{2})/\text{im}(\phi_{3})$ \`a l'\'etude du space $\ker(\phi_{2})\cap X(s_{2}^{3})/\text{im}(\phi_{3})\cap X(s_{2}^{3})$.
Notons,
\begin{align*}
X(\nu):=\{(m_{1},-m_{2}),(m_{2},-m_{4}),(m_{3},-m_{1}),(m_{4},-m_{3})\},\\ 
X(\theta\nu):=\{(m_{1},-m_{3}),(m_{2},-m_{1}),(m_{3},-m_{4}),(m_{4},-m_{2})\}.
\end{align*}
Comme pour le cran pr\'ec\'edent on d\'emontre,
\begin{lem}\label{lem:lemcran32}
\begin{align*}
\ker(\phi_{2})\cap X(s_{2}^{3})&=(\ker(\phi_{\nu,s_{2}^{3}}^{J})\cap\ker(\phi_{s_{1}^{2},s_{2}^{3}})\cap 
\ker(\phi_{s_{3}^{2},s_{2}^{3}}))\\ 
&\qquad\qquad+(\ker(\phi_{\theta\nu,s_{2}^{3}}^{J})\cap\ker(\phi_{s_{1}^{2},s_{2}^{3}})\cap \ker(\phi_{s_{3}^{2},s_{2}^{3}})). 
\end{align*}
\end{lem}
\begin{proof}[\textbf{Preuve.}]
Comme,
\begin{align*}
\ker(\phi_{s_{2}^{2},s_{2}^{3}})=&\ker(\phi_{\nu,s_{2}^{3}}^{J})+\ker(\phi_{\theta\nu,s_{2}^{3}}^{J})\\
&+(\phi_{\nu,s_{2}^{3}}^{J})^{-1}(\ker(\phi_{s_{2}^{2},\nu})\setminus\{0\})\cap (\phi_{\theta\nu,s_{2}^{3}}^{J})^{-1}(\ker(\phi_{s_{2}^{2},\theta\nu}^{J})\setminus\{0\}),
\end{align*}
on en déduit que,
\begin{align*}
\ker(\phi_{2})\cap X(s_{2}^{3})&=\ker(\phi_{s_{1}^{2},s_{2}^{3}})\cap\ker(\phi_{s_{3}^{2},s_{2}^{3}})
\cap[\ker(\phi_{\nu,s_{2}^{3}}^{J})+\ker(\phi_{\theta\nu,s_{2}^{3}}^{J})\\
&+(\phi_{\nu,s_{2}^{3}}^{J})^{-1}(\ker(\phi_{s_{2}^{2},\nu})\setminus\{0\})\cap (\phi_{\theta\nu,s_{2}^{3}}^{J})^{-1}(\ker(\phi_{s_{2}^{2},\theta\nu}^{J})\setminus\{0\})].
\end{align*}
Choisissons 
$x\in(\phi_{\nu,s_{2}^{3}}^{J})^{-1}(\ker(\phi_{s_{2}^{2},\nu}^{J})\setminus\{0\})\cap (\phi_{\theta\nu,s_{2}^{3}}^{J})^{-1}(\ker(\phi_{s_{2}^{2},\theta\nu}^{J})\setminus\{0\})$, 
$y\in\ker(\phi_{\nu,s_{2}^{3}}^{J})$ 
et $z\in \ker(\phi_{\theta\nu,s_{2}^{3}}^{J})$, tels que 
$x+y+z\in\ker(\phi_{s_{1}^{2},s_{2}^{3}})\cap \ker(\phi_{s_{3}^{2},s_{2}^{3}})$. 
Soit $\eta\in\mathfrak{S}_{4}$ tel que $\nu>_{B}\eta$, $l_{\mathfrak{S}}(\eta)=l_{\mathfrak{S}}(\nu)-1$ et
$\eta\neq s_{2}^{2}$. Notons aussi par $\eta'$ 
l'unique élémént de $\mathfrak{S}_{4}$ diff\'erent de $\nu$ tel que $s_{2}^{3}>_{B}\eta'>_{B}\eta$, alors
$\eta'=s_{1}^{2}$ ou $\eta'=s_{3}^{2}$. On a,
\begin{align*}
\phi_{\eta,\nu}^{J}\circ\phi_{\nu,s_{2}^{3}}^{J}(x+z)&=
\phi_{\eta,\nu}^{J}\circ\phi_{\nu,s_{2}^{3}}^{J}(x+y+z)\\
&=-\phi_{\eta,\eta'}^{J}\circ\phi_{\eta',s_{2}^{3}}^{J}(x+y+z)\\
&=-\phi_{\eta,\eta'}^{J}(0)=0.
\end{align*}
En plus,
\begin{align*}
\phi_{s_{2}^{2},\nu}^{J}\circ\phi_{\nu,s_{2}^{3}}^{J}(x+z)=
-\phi_{s_{2},\theta\nu}^{J}\circ\phi_{\theta\nu,s_{2}^{3}}^{J}(x)=0.
\end{align*} 
Donc,
\begin{align*}
\phi_{\nu,s_{2}^{3}}^{J}(x)+\phi_{\nu,s_{2}^{3}}^{J}(z)&\in K_{\nu}.
\end{align*}
Le m\^eme calcul est possible d'\^etre fait, cette fois-ci, pour $\phi_{\theta\nu,s_{2}^{3}}^{J}(x+y)$, 
d'o\`u on obtient,
\begin{align*}
\phi_{\theta\nu,s_{2}^{3}}^{J}(x)+\phi_{\theta\nu,s_{2}^{3}}^{J}(y)&\in K_{\theta\nu}.
\end{align*}
Du lemme suivant (dont on donne une démonstration \`a la fin de la présent preuve), 
\begin{lem}\label{lem:dulem}
\begin{align*}
K_{\nu}&= \phi_{\nu,s_{2}^{3}}^{J}(\ker(\phi_{\theta\nu,s_{2}^{3}}^{J})\cap\ker(\phi_{s_{1}^{2},s_{2}^{3}})
\cap \ker(\phi_{s_{3}^{2},s_{2}^{3}})),\\ 
K_{\theta\nu}&=\phi_{\theta\nu,s_{2}^{3}}^{J}(\ker(\phi_{\nu,s_{2}^{3}}^{J})\cap\ker(\phi_{s_{1}^{2},s_{2}^{3}})
\cap \ker(\phi_{s_{3}^{2},s_{2}^{3}})),
\end{align*}
\end{lem}
il existe $w_{1}\in \ker(\phi_{\theta\nu,s_{2}^{3}}^{J})\cap\ker(\phi_{s_{1}^{2},s_{2}^{3}})
\cap \ker(\phi_{s_{3}^{2},s_{2}^{3}})$, 
$w_{2}\in \ker(\phi_{\nu,s_{2}^{3}}^{J})\cap\ker(\phi_{s_{1}^{2},s_{2}^{3}})
\cap \ker(\phi_{s_{3}^{2},s_{2}^{3}})$, 
tel que,
\begin{align*}
\phi_{\nu,s_{2}^{3}}^{J}(w_{1})=\phi_{\nu,s_{2}^{3}}^{J}(x)+\phi_{\nu,s_{2}^{3}}^{J}(z)\qquad 
\phi_{\theta\nu,s_{2}^{3}}^{J}(w_{2})=\phi_{\theta\nu,s_{2}^{3}}^{J}(x)+\phi_{\theta\nu,s_{2}^{3}}^{J}(y). 
\end{align*}
Il nous est d\'esormais possibl\'e d'\'ecrire,
\begin{align*}
x+y+z=\frac{x+y+z-w_{1}+w_{2}}{2}+\frac{x+y+z+w_{1}-w_{2}}{2},
\end{align*}
o\`u le premi\`ere terme \`a droite de l'\'egalit\'e est contenu dans $\ker(\phi_{\nu,s_{2}^{3}}^{J})$ et 
le deuxi\`eme dans $\ker(\phi_{\theta\nu,s_{2}^{3}}^{J})$. Donc,
\begin{align*}
\ker(\phi_{2,3})\cap X(s_{2}^{3})&=\ker(\phi_{s_{1}^{2},s_{2}^{3}})\cap \ker(\phi_{s_{3}^{2},s_{2}^{3}})
\cap(\ker(\phi_{\nu,s_{2}^{3}}^{J})+\ker(\phi_{\theta\nu,s_{2}^{3}}^{J})). 
\end{align*}
Consid\'erons maintenant $x\in\ker(\phi_{\nu,s_{2}^{3}}^{J})$ et $y\in\ker(\phi_{\theta\nu,s_{2}^{3}}^{J})$ tels que,
$x+y\in\ker(\phi_{s_{1}^{2},s_{2}^{3}})\cap \ker(\phi_{s_{3}^{2},s_{2}^{3}})$.
On d\'efinit,
\begin{align*}
z_{1}=\phi_{s_{1}^{2},s_{2}^{3}}(x)=-\phi_{s_{1}^{2},s_{2}^{3}}(y)
\end{align*}
et
\begin{align*}
z_{2}=\phi_{s_{3}^{2},s_{2}^{3}}(x)=-\phi_{s_{3}^{2},s_{2}^{3}}(y).
\end{align*}
Notons,
\begin{align*}
X(\alpha):=\{(m_{1},-m_{2}),(m_{2},-m_{3}),(m_{3},-m_{1}),(m_{4},-m_{4})\},\\
X(\theta\alpha):=\{(m_{1},-m_{3}),(m_{2},-m_{1}),(m_{3},-m_{2}),(m_{4},-m_{4})\}.
\end{align*}
Comme,
\begin{align*}
\phi_{\alpha,s_{1}^{2}}^{J}(z_{1})&=\phi_{\alpha,s_{1}^{2}}^{J}\circ\phi_{s_{1}^{2},s_{2}^{3}}^{J}(x)\\
&=-\phi_{\alpha,\nu}^{J}\circ\phi_{\nu,s_{2}^{3}}^{J}(x)=0
\end{align*}
et
\begin{align*}
\phi_{\theta\alpha,s_{1}^{2}}^{J}(z_{1})&=\phi_{\theta\alpha,s_{1}^{2}}^{J}\circ\phi_{s_{1}^{2},s_{2}^{3}}^{J}(-y)\\
&=\phi_{\alpha,\theta\nu}^{J}\circ\phi_{\theta\nu,s_{2}^{3}}^{J}(-y)=0.
\end{align*}
on conclut $z_{1}\in K_{s_{1}^{2}}$.
Le m\^eme calcul est possible d'\^etre fait, cette fois-ci, pour $z_{2}$, 
alors, $z_{2}\in K_{s_{3}^{2}}$.
Au cours de la preuve de la $\theta$-exactitude pour le cran $\mathbf{(2\rightarrow 1)}$ 
on donne une démonstration des deux \'egalit\'es suivants, 
\begin{align*}
K_{s_{1}^{2}}=\overline{X(s_{1}^{2})},\\
K_{s_{3}^{2}}=\overline{X(s_{3}^{2})}.
\end{align*} 
On peut donc choisir, 
$\bar{z}_{1}\in \ker(\phi_{s_{3}^{2},s_{2}^{3}})\cap\ker(\phi_{\nu,s_{2}^{3}}^{J})\cap \ker(\phi_{\theta\nu,s_{2}^{3}}^{J})$,
$\bar{z}_{2}\in \ker(\phi_{s_{1}^{2},s_{2}^{3}})\cap\ker(\phi_{\nu,s_{2}^{3}}^{J})\cap \ker(\phi_{\theta\nu,s_{2}^{3}}^{J})$
tel que $\phi_{s_{1}^{2},s_{2}^{3}}(\bar{z}_{1})=z_{1}$ et $\phi_{s_{3}^{2},s_{2}^{3}}(\bar{z}_{2})=z_{2}$. 
Ce qui nous permet d'\'ecrire, 
\begin{align*}
x+y=(x-\bar{z}_{1}-\bar{z}_{2})+(y+\bar{z}_{1}+\bar{z}_{2}),
\end{align*}
o\`u 
$x-\bar{z}_{1}-\bar{z}_{2}\in 
\ker(\phi_{\nu,s_{2}^{3}}^{J})\cap\ker(\phi_{s_{1}^{2},s_{2}^{3}})\cap \ker(\phi_{s_{3}^{2},s_{2}^{3}})$ 
et 
$y+\bar{z}_{1}+\bar{z}_{2}\in 
\ker(\phi_{\theta\nu,s_{2}^{3}}^{J})\cap\ker(\phi_{s_{1}^{2},s_{2}^{3}})\cap \ker(\phi_{s_{3}^{2},s_{2}^{3}})$.

Il ne nous reste maintenant qu'\`a d\'emontrer le lemme \textbf{(\ref{lem:dulem})}.
Soit $z\in K_{\nu}$. 
Notons,
\begin{align*}
X(\lambda):=\{(m_{1},-m_{4}),(m_{2},-m_{2}),(m_{3},-m_{1}),(m_{4},-m_{3})\},\\
X(\theta\lambda):=\{(m_{1},-m_{3}),(m_{2},-m_{2}),(m_{3},-m_{4}),(m_{4},-m_{1})\},
\end{align*}
et,
\begin{align*}
X(\zeta):=\{(m_{1},-m_{2}),(m_{2},-m_{4}),(m_{3},-m_{3}),(m_{4},-m_{1})\},\\
X(\theta\zeta):=\{(m_{1},-m_{4}),(m_{2},-m_{1}),(m_{3},-m_{3}),(m_{4},-m_{2})\}.
\end{align*}
Par l'exactitude de la suite de Johnson il existe $a\in X(\lambda)$,  
$\bar{a}\in X(\theta\lambda)$, $b\in X(\zeta)$, $\bar{b}\in X(\theta\zeta)$
et $c\in X(s_{2}^{3})$, tel que,
\begin{align*}
z=\phi_{\nu,\lambda}^{J}(a)+\phi_{\nu,\zeta}^{J}(b)+\phi_{\nu,s_{2}^{3}}^{J}(c), 
\end{align*}
et pour tout $\eta\in\mathfrak{S}_{4}$ avec $\eta\neq\nu$ et $l_{\mathfrak{S}}(\eta)=l_{\mathfrak{S}}(\nu)$,
\begin{align*}
0=\phi_{\eta,\lambda}^{J}(a)+\phi_{\eta,\theta\lambda}^{J}(\bar{a})+\phi_{\eta,\zeta}^{J}(b)+
\phi_{\eta,\theta\zeta}^{J}(\bar{b})+\phi_{\eta,s_{2}^{3}}^{J}(c). 
\end{align*}
Soit $u\in X(\mu)$, respectivement $\bar{u}\in X(\theta\mu)$, 
tel que $\phi_{\theta\lambda,\mu}^{J}(u)=-\bar{a}$, respectivement 
$\phi_{\theta\zeta,\theta\mu}^{J}(\bar{u})=-\bar{b}$. 
On \'ecrit $a_{1}=a+\phi_{\lambda,\theta\mu}^{J}(\bar{u}),$ $b_{1}=b+\phi_{\zeta,\mu}^{J}(u)$
et $c_{1}=c+\phi_{s_{2}^{3},\mu}^{J}(u)+\phi_{s_{2}^{3},\theta\mu}^{J}(\bar{u})$. Alors
\begin{align*}
z&=\phi_{\nu,\lambda}^{J}(a_{1})+\phi_{\nu,\zeta}^{J}(b_{1})+\phi_{\nu,s_{2}^{3}}^{J}(c_{1}),\\ 
0&=\phi_{\theta\nu,s_{2}^{3}}^{J}(c_{1}),
\end{align*}
et pour tout $\eta\in\mathfrak{S}_{4}$ avec $\eta\neq\nu,~\eta\neq\theta\nu$ et $l_{\mathfrak{S}}(\eta)=l_{\mathfrak{S}}(\nu)$,
\begin{align*}
0=\phi_{\eta,\lambda}^{J}(a_{1})+\phi_{\eta,\zeta}^{J}(b_{1})+\phi_{\eta,s_{2}^{3}}^{J}(c_{1}). 
\end{align*}
Notons,
\begin{align*}
X(\alpha):=\{(m_{1},-m_{2}),(m_{2},-m_{3}),(m_{3},-m_{1}),(m_{4},-m_{4})\},\\
X(\theta\alpha):=\{(m_{1},-m_{3}),(m_{2},-m_{1}),(m_{3},-m_{2}),(m_{4},-m_{4})\}.
\end{align*}
Et,
\begin{align*}
X(\tau):=\{(m_{1},-m_{4}),(m_{2},-m_{1}),(m_{3},-m_{2}),(m_{4},-m_{3})\},\\
X(\theta\tau):=\{(m_{1},-m_{2}),(m_{2},-m_{3}),(m_{3},-m_{4}),(m_{4},-m_{1})\},
\end{align*}
On a, car $s_{2}^{3}\ngtr_{B}\tau$ et $\zeta\ngtr_{B}\tau$, 
l'\'egalit\'e $\phi_{\tau,\lambda}^{J}(a_{1})=0$, donc,
\begin{align*}
\phi_{\theta\alpha,s_{1}^{2}}^{J}\circ \phi_{s_{1}^{2},\lambda}^{J}(a_{1})=
-\phi_{\theta\alpha,\tau}^{J}\circ \phi_{\tau,\lambda}^{J}(a_{1})=0.
\end{align*}
et l'\'el\'ement $\phi_{s_{1}^{2},\lambda}^{J}(a_{1})$ appartient \`a 
$\ker(\phi_{\theta\alpha,s_{1}^{2}}^{J})$, par cons\'equent, 
\begin{align*}
\phi_{\alpha,s_{1}^{2}}^{J}\circ\phi_{s_{1}^{2},\lambda}^{J}(a_{1})&\in\ker(\phi_{s_{1}^{1},\alpha}^{J})\cap
\ker(\phi_{s_{2}^{1},\alpha}^{J}).
\end{align*}
Notons par $\phi_{s_{1}^{1},\alpha}^{GL(6)}$, respectivement $\phi_{s_{2}^{1},\alpha}^{GL(6)}$, 
le morphismes de Johnson entre,
\begin{align*}
\{(m_{1},-m_{2}),(m_{2},-m_{3}),(m_{3},-m_{1})\}\text{ et }\{(m_{1},-m_{2}),(m_{2},-m_{1}),(m_{3},-m_{3})\},
\end{align*}
respectivement,
\begin{align*}
\{(m_{1},-m_{2}),(m_{2},-m_{3}),(m_{3},-m_{1})\})\text{ et }\{(m_{1},-m_{1}),(m_{3},-m_{2}),(m_{2},-m_{3})\},
\end{align*} 
qu'intervient dans la résolution de Johnson de $\overline{(m_{1},-m_{3}),(m_{2},-m_{2}),(m_{3},-m_{1})}$.

Alors,
\begin{align*}
\ker(\phi_{s_{1}^{1},\alpha}^{J})\cap\ker(\phi_{s_{2}^{1},\alpha}^{J})
=\ker\phi_{s_{1}^{1},\alpha}^{GL(6)}\cap\ker\phi_{s_{2}^{1},\alpha}^{GL(6)}\times X(m_{4},-m_{4}).
\end{align*}
La multiplicité de $\{(m_{1},-m_{1}),(m_{2},-m_{2}),(m_{3},-m_{3})\}$ 
en tant que sous-quotient de $\{(m_{1},-m_{2}),(m_{2},-m_{3}),(m_{3},-m_{1})\}$ est, 
d'après le théorème \textbf{(6.2)} de \textbf{\cite{Casselman-Shahidi}}, 
égal à 1,  par conséquent,
\begin{align*}
\ker\phi_{s_{1}^{1},\alpha}^{GL(6)}\cap\ker\phi_{s_{2}^{1},\alpha}^{GL(6)}=\overline{X(\{\delta(m_{1},-m_{2}),\delta(m_{2},-m_{3}),\delta(m_{3},-m_{1})\})}, 
\end{align*}
et on peut \'ecrire,
\begin{align*}
\ker(\phi_{s_{1}^{1},\alpha}^{J})\cap\ker(\phi_{s_{2}^{1},\alpha}^{J})
&=\overline{X(\{\delta(m_{1},-m_{2}),\delta(m_{2},-m_{3}),\delta(m_{3},-m_{1})\})}\times X(m_{4},-m_{4})\\
&=\overline{X(\alpha)}.
\end{align*}
En cons\'equence,
\begin{align*}
\phi_{\alpha,s_{1}^{2}}^{J}\circ\phi_{s_{1}^{2},\lambda}^{J}(a_{1})\in\overline{X(\alpha)},
\end{align*}
et on peut choisir un \'el\'ement $\bar{v}\in X(\theta\mu)$ tel que, $\phi_{\lambda,\theta\mu}^{J}(\bar{v})=-a_{1}$
et $\phi_{\theta\zeta,\theta\mu}^{J}(v)=0$. De mani\`ere analogue on choisit 
un \'el\'ement $v\in X(\mu)$ tel que, $\phi_{\zeta,\mu}^{J}(v)=-b_{1}$
et $\phi_{\theta\lambda,\mu}^{J}(v)=0$. On
\'ecrit $c_{2}=c_{1}+\phi_{s_{2}^{3},\mu}^{J}(v)+\phi_{s_{2}^{3},\theta\mu}^{J}(\bar{v})$,
alors $c_{2}\in \ker(\phi_{\theta\nu,s_{2}^{3}}^{J})\cap\ker(\phi_{s_{1}^{2},s_{2}^{3}})
\cap \ker(\phi_{s_{3}^{2},s_{2}^{3}})$
et $\phi_{\nu,s_{2}^{3}}^{J}(c_{2})=z$.
\end{proof}
Du lemme \textbf{(\ref{lem:lemcran32})} ils nous est d\'esormais possible d'\'ecrire,
\begin{align*}
\text{Tr}_{\theta}(\ker(\phi_{2,3})\cap X(s_{2}^{3}))
&=\text{Tr}_{\theta}[\ker(\phi_{\nu,s_{2}^{3}}^{J})\cap\ker(\phi_{s_{1}^{2},s_{2}^{3}})\cap \ker(\phi_{s_{3}^{2},s_{2}^{3}})\\ 
&\qquad\qquad+\ker(\phi_{\theta\nu,s_{2}^{3}}^{J})\cap\ker(\phi_{s_{1}^{2},s_{2}^{3}})\cap \ker(\phi_{s_{3}^{2},s_{2}^{3}})]\\
&=\text{Tr}_{\theta}(K_{s_{2}^{3}})
\end{align*}
Comme en plus,
\begin{align*}
K_{s_{2}^{3}}=\phi_{s_{2}^{3},\mu}^{J}\left(\bigcap_{\nu'\neq s_{2}^{3}}\ker(\phi_{\nu',\mu}^{J})\right)+
\phi_{s_{2}^{3},\theta\mu}^{J}\left(\bigcap_{\nu'\neq s_{2}^{3}}\ker(\phi_{\nu',\theta\mu}^{J})\right) 
\end{align*}
on obtient,
\begin{align*}
\text{Tr}_{\theta}(\ker(\phi_{2,3})\cap X(s_{2}^{3}))=\text{Tr}_{\theta}[\phi_{s_{2}^{3},\mu}(\bigcap_{\nu'\neq s_{2}^{3}}\ker(\phi_{\nu',\mu}^{J}))
\cap\phi_{s_{2}^{3},\theta\mu}^{J}(\bigcap_{\nu'\neq s_{2}^{3}}\ker(\phi_{\nu',\theta\mu}^{J}))].
\end{align*}
Soit $x\in \bigcap_{\nu'\neq s_{2}^{3}}\ker(\phi_{\nu',\mu}^{J})$ et $y\in \bigcap_{\nu'\neq s_{2}^{3}}\ker(\phi_{\nu',\theta\mu}^{J})$ tel que,
\begin{align*}
\phi_{s_{2}^{3},\mu}^{J}(x)=\phi_{s_{2}^{3},\theta\mu}^{J}(y). 
\end{align*}
Par l'exactitude de la suite de Johnson il existe un \'el\'ement $z\in X(s^{4})\cap \ker(\phi_{s_{1}^{3},s^{4}})$ 
tel que, 
\begin{align*}
\phi_{\mu,s^{4}}^{J}(z)=x, \qquad \phi_{\theta\mu,s^{4}}^{J}(z)=-y,
\end{align*}
d'o\`u,
\begin{align*}
\phi_{s_{2}^{3},\mu}^{J}\left(\bigcap_{\nu'\neq s_{2}^{3}}\ker(\phi_{\nu',\mu}^{J})\right)
\cap\phi_{s_{2}^{3},\theta\mu}^{J}\left(\bigcap_{\nu'\neq s_{2}^{3}}\ker(\phi_{\nu',\theta\mu}^{J})\right)=
\phi_{3}(X(s^{4})\cap \ker(\phi_{s_{1}^{3},s^{4}}))
\end{align*}
et on peut clairement conclure que,
\begin{align*}
\text{Tr}_{\theta}(X(s_{2}^{3})\cap\ker(\phi_{2})/X(s_{2}^{3})\cap\text{im}(\phi_{3}))=0. 
\end{align*}
Cran $\mathbf{2\rightarrow 1}$,
\begin{align*}
\oplus_{k=1}^{3}X(s_{k}^{2})\rightarrow \oplus_{j=1}^{3}X(s_{j}^{1}), 
\end{align*}
o\`u on rappelle que,
\begin{align*}
X(s_{1}^{1}):=\{(m_{1},-m_{2}),(m_{2},-m_{1}),(m_{3},-m_{3}),(m_{4},-m_{4})\},\\
X(s_{2}^{1}):=\{(m_{1},-m_{1}),(m_{2},-m_{3}),(m_{3},-m_{2}),(m_{4},-m_{4})\},\\
X(s_{3}^{1}):=\{(m_{1},-m_{1}),(m_{2},-m_{2}),(m_{3},-m_{4}),(m_{4},-m_{3})\}.
\end{align*}
Le fait que tout \'el\'ement dans $X(s_{1}^{2})$ soit possible d'\^etre \'ecrit comme la somme d'un \'el\'ement
dans $\text{im}(\phi_{2,3})$ plus un \'el\'ement dans $X(s_{2}^{2})\oplus X(s_{3}^{2})$ nous permet de 
ram\`ener la preuve à l'\'etude de l'espace $\ker(\phi_{1})\cap (X(s_{2}^{2})\oplus X(s_{3}^{2}))$.
Soit $(v_{s_{2}^{2}},v_{s_{3}^{2}})\in\ker(\phi_{1})\cap (X(s_{2}^{2})\oplus X(s_{3}^{2}))$.
Alors $\phi_{s_{3}^{1},s_{2}^{2}}(v_{s_{2}^{2}})=-\phi_{s_{3}^{1},s_{3}^{2}}(v_{s_{3}^{2}})$. 
Comme la multiplicité de $X(s_{\text{max}})$ en tant que sous-quotient de $X(s_{3}^{1})$ est égal à 1 et la suite de Johnson d\'efinit
un complexe diff\'erentiel, on obtient 
$\phi_{s_{3}^{1},s_{2}^{2}}(v_{s_{2}^{2}})=-\phi_{s_{3}^{1},s_{3}^{2}}(v_{s_{3}^{2}})\in \overline{X(s_{3}^{1})}$. 
Par cons\'equent, il existe $z\in X(s_{2}^{3})\cap \ker(\phi_{s_{1}^{2},s_{2}^{3}})$
tel que $\phi_{s_{2}^{2},s_{2}^{3}}(z)=v_{s_{2}^{2}}$ et 
$v_{s_{3}^{2}}-\phi_{s_{3}^{2},s_{2}^{3}}(z)\in K_{s_{3}^{2}}$. Pour 
$(v_{s_{2}^{2}},v_{s_{3}^{2}})$ on a donc l'\'egalit\'e,
\begin{align*}
(v_{s_{2}^{2}},v_{s_{3}^{2}})=(v_{s_{2}^{2}},\phi_{s_{3}^{2},s_{2}^{3}}(z))+(0,v_{s_{3}^{2}}-
\phi_{s_{3}^{2},s_{2}^{3}}(z)), 
\end{align*}
avec $(v_{s_{2}^{2}},\phi_{s_{3}^{2},s_{2}^{3}}(z))\in\text{im}(\phi_{2})$
et $v_{s_{3}^{2}}-\phi_{s_{3}^{2},s_{2}^{3}}(z)\in K_{s_{3}^{2}}$.
Ce qui nous permet de ram\`ener l'\'etude de l'espace 
$\ker(\phi_{1})\cap X(s_{2}^{2})\oplus X(s_{2}^{2})/\text{im}(\phi_{2,3})\cap X(s_{2}^{2}\oplus X(s_{2}^{2})$
\`a l'\'etude de,
\begin{align*}
\ker(\phi_{1})\cap X(s_{3}^{2})/\text{im}(\phi_{2})\cap X(s_{3}^{2}).
\end{align*}
Notons maintenant,
\begin{align*}
X(\beta):=\{(m_{1},-m_{1}),(m_{2},-m_{3}),(m_{3},-m_{4}),(m_{4},-m_{2})\},\\
X(\theta\beta):=\{(m_{1},-m_{1}),(m_{2},-m_{4}),(m_{3},-m_{2}),(m_{4},-m_{3})\}.
\end{align*}
Soit $\phi_{\beta,s_{3}^{2}}^{GL(6)}$, respectivement $\phi_{\theta\beta,s_{3}^{2}}^{GL(6)}$, le morphismes de Johnson entre,
\begin{align*}
\{(m_{2},-m_{4}),(m_{3},-m_{3}),(m_{4},-m_{2})\}\text{ et } \{(m_{2},-m_{3}),(m_{3},-m_{4}),(m_{4},-m_{2})\},
\end{align*}
respectivement,
\begin{align*}
\{(m_{2},-m_{4}),(m_{3},-m_{3}),(m_{4},-m_{2})\}\text{ et }\{(m_{2},-m_{4}),(m_{3},-m_{2}),(m_{4},-m_{3})\},
\end{align*}
qu'intervient dans la résolution de Johnson de $\overline{\{(m_{2},-m_{4}),(m_{3},-m_{3}),(m_{4},-m_{2})\}}$.

De la d\'efinition du morphisme $\phi_{\beta,s_{3}^{2}}^{J}$, respectivement $\phi_{\theta\beta,s_{3}^{2}}^{J}$, 
on obtient l'\'egalit\'e,
\begin{align*}
\ker(\phi_{\beta,s_{3}^{2}}^{J})=\delta(m_{1},-m_{1})\times \ker\phi_{\beta,s_{3}^{2}}^{GL(6)},
\end{align*}
respectivement,
\begin{align*}
\ker(\phi_{\theta\beta,s_{3}^{2}}^{J})=\delta(m_{1},-m_{1})\times \ker\phi_{\theta\beta,s_{3}^{2}}^{GL(6)}.
\end{align*}
Consid\'erons aussi les morphismes de Johnson suivants,
\begin{align*}
\phi_{s_{2}^{1},s_{3}^{2}}^{GL(6)}:\{(m_{2},-m_{4}),(m_{3},-m_{3}),&(m_{4},-m_{2})\}\rightarrow\\ 
&\{(m_{2},-m_{3}),(m_{3},-m_{2}),(m_{4},-m_{4})\},\\
\phi_{s_{3}^{1},s_{3}^{2}}^{GL(6)}:\{(m_{2},-m_{2}),(m_{3},-m_{4}),&(m_{4},-m_{3})\}\rightarrow\\ 
&\{(m_{2},-m_{4}),(m_{3},-m_{2}),(m_{4},-m_{3})\}.
\end{align*}
Du lemme \textbf{(\ref{lem:lemmeGL6})} on sait que,
\begin{align*}
\ker\phi_{s_{2}^{1},s_{3}^{2}}^{GL(6)}\cap\ker\phi_{s_{3}^{1},s_{3}^{2}}^{GL(6)}=
\ker\phi_{\beta,s_{3}^{2}}^{GL(6)}+\ker\phi_{\theta\beta,s_{3}^{2}}^{GL(6)}.
\end{align*}
Par cons\'equent, on peut \'ecrire,
\begin{align*}
\ker(\phi_{1,2})\cap X(s_{3}^{2})&=\ker(\phi_{s_{2}^{1},s_{2}^{2}})\cap \ker(\phi_{s_{3}^{1},s_{2}^{2}})\\
&=\ker(\phi_{\beta,s_{3}^{2}}^{J})+\ker(\phi_{\theta\beta,s_{3}^{2}}^{J}).  
\end{align*}
D'o\`u,
\begin{align}\label{eq:gauche}
\text{Tr}_{\theta}(\ker(\phi_{1})\cap X(s_{2}^{2}))&=\text{Tr}_{\theta}(\ker(\phi_{\beta,s_{3}^{2}}^{J})\cap\ker(\phi_{\theta\beta,s_{3}^{2}}^{J})).
\end{align}
Comme en plus,
\begin{align*}
\ker(\phi_{\beta,s_{3}^{2}}^{J})\cap\ker(\phi_{\theta\beta,s_{3}^{2}}^{J})&=
\ker\phi_{\beta,s_{3}^{2}}^{GL(6)}\cap\ker\phi_{\theta\beta,s_{3}^{2}}^{GL(6)}\times X(m_{4},-m_{4}),
\end{align*}
et de l'exactitude de la suite de Johnson pour $\overline{\{(m_{2},-m_{4}),(m_{3},-m_{3}),(m_{4},-m_{2})\}}$ on a, 
\begin{align*}
\ker\phi_{\beta,s_{3}^{2}}^{GL(6)}\cap\ker\phi_{\theta\beta,s_{3}^{2}}^{GL(6)}=\overline{X(\delta(m_{1},-m_{3})\otimes\delta(m_{2},-m_{2})\otimes\delta(m_{3},-m_{1}))}. 
\end{align*}
on en déduit que,
\begin{align}\label{eq:droite}
\ker(\phi_{\beta,s_{3}^{2}}^{J})\cap\ker(\phi_{\theta\beta,s_{3}^{2}}^{J})&=\overline{X(\delta(m_{1},-m_{3})\otimes\delta(m_{2},-m_{2})\otimes\delta(m_{3},-m_{1}))} 
\times\overline{X(m_{4},-m_{4})}
\nonumber\\
&=\overline{X(s_{3}^{2})}.
\end{align}
Ainsi, des \'egalit\'es \textbf{(\ref{eq:gauche})} et \textbf{(\ref{eq:droite})}, on obtient,
\begin{align*}
\text{Tr}_{\theta}(\ker(\phi_{1})\cap X(s_{3}^{2}))=\overline{X(s_{3}^{2})}. 
\end{align*}
Ce qui nous permet 
de conclure que,
\begin{align*}
\ker(\phi_{1})/\text{im}(\phi_{2})=0.
\end{align*}
Cran $\mathbf{1\rightarrow 0}$,
\begin{align*}
\oplus_{j=1}^{3}X(s_{j}^{1})\rightarrow X(s_{\text{max}}), 
\end{align*}
o\`u on rappelle que,
\begin{align*}
X(s_{\text{max}}):=\{(m_{1},-m_{1}),(m_{2},-m_{2}),(m_{3},-m_{3}),(m_{4},-m_{4})\}. 
\end{align*}
Pour chaque $1\leq j\leq 3$ on a, car la multiplicité de $X(s_{\text{max}})$ en tant que sous-quotient de $X(s_{j}^{1})$ est égal à 1, 
l'\'egalit\'e,
\begin{align*}
\ker(\phi_{s_{\text{max}},s_{j}^{1}})=\overline{X(s_{j}^{1})}.
\end{align*}
Comme en plus chaque $\overline{X(s_{j}^{1})},~1\leq j\leq 3$ 
est contenu dans $\text{im}(\phi_{1,2})$, on conclut,
\begin{align*}
\text{Tr}_{\theta}(\ker(\phi_{1})\cap X(s_{3}^{2})/\text{im}(\phi_{2})\cap X(s_{3}^{2}))=0.
\end{align*}
\bibliographystyle{alpha-fr}
\bibliography{BibliographieNAR}

\end{document}